\documentclass[12pt]{article}
\usepackage{fullpage}
\usepackage{amsmath,amssymb,amsfonts,amsthm}
\usepackage[all,arc,knot,poly]{xy}
\usepackage{enumerate}
\usepackage{mathrsfs}
\usepackage{amsxtra}
\usepackage{needspace}
\usepackage{graphicx}

\DeclareMathOperator{\sgn}{\operatorname{sgn}}
\DeclareMathOperator{\Hom}{\operatorname{Hom}}
\DeclareMathOperator{\cross}{\operatorname{cross}}
\DeclareMathOperator{\Fun}{\operatorname{Fun}}

\makeatletter
\def\Ddots{\mathinner{\mkern1mu\raise\p@
\vbox{\kern7\p@\hbox{.}}\mkern2mu
\raise4\p@\hbox{.}\mkern2mu\raise7\p@\hbox{.}\mkern1mu}}
\makeatother

\newenvironment{step}[1][]{
\medskip
\emph{Step #1.}
}
{\medskip}

\theoremstyle{definition}
\newtheorem{theorem}{Theorem}[section]
\newtheorem{definition}[theorem]{Definition}
\newtheorem{lemma}[theorem]{Lemma}
\newtheorem{proposition}[theorem]{Proposition}

\newtheorem{example}[theorem]{Example}

\title{Laminations from the symplectic double}

\author{Dylan G.L. Allegretti}

\date{}

\begin{document}

\maketitle

\begin{abstract}
Let $S$ be a compact oriented surface with boundary together with finitely many marked points on the boundary, and let $S^\circ$ be the same surface equipped with the opposite orientation. We consider the \emph{double} $S_\mathcal{D}$ obtained by gluing the surfaces $S$ and~$S^\circ$ along corresponding boundary components. We define a notion of lamination on the double and construct coordinates on the space of all such laminations. We show that this space of laminations is a tropical version of the symplectic double introduced by Fock and Goncharov. There is a canonical pairing between our laminations and the positive real points of the symplectic double. We derive an explicit formula for this pairing using the $F$-polynomials of Fomin and Zelevinsky.
\end{abstract}

\tableofcontents

\section{Introduction}

\subsection{Previous work on laminations}

As part of their investigation of higher Teichm\"uller theory, Fock and Goncharov studied two versions of the Teichm\"uller space of a surface $S$ with boundary~\cite{IHES,dual}. These are called the Teichm\"uller $\mathcal{A}$- and $\mathcal{X}$-spaces and are denoted $\mathcal{A}^+(S)$ and $\mathcal{X}^+(S)$, respectively. The Teichm\"uller $\mathcal{A}$-space, also known as the decorated Teichm\"uller space, was originally studied by Penner~\cite{P}. It arises as the set of positive real points of a certain algebraic variety constructed in~\cite{IHES}. The Teichm\"uller $\mathcal{X}$-space, sometimes called the enhanced Teichm\"uller space, is a slight variant of the classical Teichm\"uller space of a punctured surface. It also arises as the set of positive real points of an algebraic variety~\cite{IHES}.

In addition to the Teichm\"uller $\mathcal{A}$- and $\mathcal{X}$-spaces, Fock and Goncharov studied two types of laminations on a surface~\cite{IHES,dual}. The two types of laminations are called $\mathcal{A}$- and $\mathcal{X}$-laminations, and the spaces of these laminations are denoted $\mathcal{A}_L(S)$ and $\mathcal{X}_L(S)$, respectively. These laminations are closely related to Thurston's transversally measured laminations~\cite{Th}. A rational lamination is defined in~\cite{IHES,dual} as a collection of finitely many simple nonintersecting curves on~$S$ with rational weights, subject to certain conditions and equivalence relations. The lamination spaces arise as the sets of rational tropical points of the two algebraic varieties constructed in~\cite{IHES}.

In~\cite{IHES,dual}, Fock and Goncharov defined special coordinates $A_i$ ($i\in I$) and $X_i$ ($i\in I$) on the Teichm\"uller spaces $\mathcal{A}^+(S)$ and $\mathcal{X}^+(S)$, respectively. The $A_i$ are the same as Penner's lambda lengths~\cite{P}, while the $X_i$ are related to shear coordinates. Here $I$ is the set of edges of an ideal triangulation of the surface $S$, and there are certain rules for transforming between the coordinates associated with different triangulations. It was shown in~\cite{GSV,IHES} that in the case of the $\mathcal{A}$-space these transformation rules are examples of the cluster mutation formulas from the theory of cluster algebras pioneered by Fomin and Zelevinsky~\cite{FZI}.

Similarly, there are coordinates $a_i$ ($i\in I$) and $x_i$ ($i\in I$) on $\mathcal{A}_L(S)$ and $\mathcal{X}_L(S)$~\cite{IHES,dual}. As in the case of the Teichm\"uller spaces, these coordinates transform according to certain rules when we change the triangulation of the surface. These rules are obtained from the cluster transformation rules by a tropicalization procedure that will be made precise later.

In addition to defining special coordinates on the Teichm\"uller and lamination spaces, Fock and Goncharov showed that the Teichm\"uller and lamination spaces are dual in a certain sense~\cite{IHES}. They defined natural subsets $\mathcal{A}_L(S,\mathbb{Z})\subseteq\mathcal{A}_L(S)$ and $\mathcal{X}_L(S,\mathbb{Z})\subseteq\mathcal{X}_L(S)$ and canonical maps 
\begin{align*}
\mathbb{I}&:\mathcal{A}_L(S,\mathbb{Z})\times\mathcal{X}^+(S)\rightarrow\mathbb{R}_{>0}, \\
\mathbb{I}&:\mathcal{X}_L(S,\mathbb{Z})\times\mathcal{A}^+(S)\rightarrow\mathbb{R}_{>0}, 
\end{align*}
which we denote by the same symbol $\mathbb{I}$. We usually think of these maps as operations which assign to a lamination $l$ the function on Teichm\"uller space given by $\mathbb{I}(l)(m)=\mathbb{I}(l,m)$.

The functions $\mathbb{I}(l)$ obtained in this way are interesting for algebraic reasons. As shown in~\cite{IHES}, the function $\mathbb{I}(l)$ can always be written as a Laurent polynomial with positive integral coefficients in the variables $X_i^{1/2}$ or $A_i$. Moreover, if we consider a lamination $l$ that has \emph{integral coordinates}, then $\mathbb{I}(l)$ can in fact be written as a Laurent polynomial with positive integral coefficients in the coordinates $X_i$ or $A_i$.

In subsequent work, Fock and Goncharov studied an algebraic variety called the symplectic double~\cite{dilog} and found that its properties are related to the geometry of a doubled surface~\cite{double}. To define the double $S_\mathcal{D}$ of a compact oriented surface~$S$, we let $S^\circ$ be the same surface equipped with the opposite orientation. Then $S_\mathcal{D}$ is defined to be the surface obtained by gluing $S$ and~$S^\circ$ along corresponding boundary components. In~\cite{double}, Fock and Goncharov considered a space $\mathcal{D}^+(S)$ which we will call the Teichm\"uller $\mathcal{D}$-space. It is similar to the Teichm\"uller space of~$S_\mathcal{D}$, and it arises as the set of positive real points of the symplectic double.

To prove this fact, Fock and Goncharov showed in~\cite{double} that $\mathcal{D}^+(S)$ has a system of coordinates $B_i$,~$X_i$ ($i\in I$) associated to an ideal triangulation of $S$. There are certain formulas describing how the coordinates change when one changes the triangulation, and these are the same formulas used in~\cite{dilog} to define the symplectic double.

\subsection{The present work}

The goal of the present paper is to define and study a new class of objects called \emph{$\mathcal{D}$-laminations}. A $\mathcal{D}$-lamination is typically represented by finitely many curves on the double~$S_\mathcal{D}$. In the case where there are no marked points on $S$, one has the following definition.

\begin{definition}
A \emph{rational $\mathcal{D}$-lamination} on $S_\mathcal{D}$ is the homotopy class of a collection of finitely many simple, nonintersecting, and noncontractible closed curves with positive rational weights and a choice of orientation for each component of $\partial S$ which meets or is homotopic to a curve. A lamination containing homotopic curves of weights $a$ and $b$ is equivalent to the lamination with one curve removed and the weight $a+b$ on the other.
\end{definition}

One of the main results of this paper is the construction of a system of coordinates $b_i,x_i$ ($i\in I$) on the space $\mathcal{D}_L(S)$ of all rational $\mathcal{D}$-laminations. As in the case of $\mathcal{A}$- and $\mathcal{X}$-laminations, the coordinates transform by certain rules when we change the triangulation of the surface, and these rules can be obtained by tropicalizing the formulas that define the symplectic double.

In addition to defining special coordinates on the lamination space $\mathcal{D}_L(S)$, one of the goals of this paper is to describe a duality between $\mathcal{D}$-laminations and the Teichm\"uller $\mathcal{D}$-space. We describe a subset $\mathcal{D}_L(S,\mathbb{Z})\subseteq\mathcal{D}_L(S)$ and a canonical map 
\[
\mathbb{I}:\mathcal{D}_L(S,\mathbb{Z})\times\mathcal{D}^+(S)\rightarrow\mathbb{R}_{>0}.
\]
This pairing is similar to the pairings between $\mathcal{A}$- and $\mathcal{X}$-laminations and the points of the Teichm\"uller $\mathcal{X}$- and $\mathcal{A}$-spaces. We usually view the above pairing as a map from $\mathcal{D}$-laminations to functions on $\mathcal{D}^+(S)$ given by $\mathbb{I}_\mathcal{D}(l)(m)=\mathbb{I}(l,m)$. 

We show that the function $\mathbb{I}_\mathcal{D}(l)$ can be expressed in a particularly nice way as an algebraic function in the variables $B_i$ and $X_i$ for any $l\in\mathcal{D}_L(S,\mathbb{Z})$. Moreover, if we consider a lamination $l$ that has \emph{integral coordinates} then $\mathbb{I}_\mathcal{D}(l)$ can in fact be written as a rational function in the coordinates $B_i$ and $X_i$.

To make these statements more precise, let us introduce some notation and terminology. Label the elements of $I$ by the numbers $\{1,\dots,n\}$ so that a point of $\mathcal{D}^+(S)$ has coordinates $B_1,\dots,B_n,X_1,\dots,X_n$. A curve on $S_\mathcal{D}$ will be called an \emph{intersecting curve} if every curve in its homotopy class intersects the image of the boundary~$\partial S$. Any $\mathcal{D}$-lamination $l\in\mathcal{D}_L(S,\mathbb{Z})$ can be represented by a collection of curves of weight~1. If this collection contains homotopic curves of weights~$a$ and~$b$ which are not intersecting curves, let us replace these by a single curve of weight $a+b$. In this way, we obtain a new collection of curves representing $l$. If we now cut the surface along the image of $\partial S$, we obtain a collection $\mathcal{C}$ of curves on~$S$ and a collection $\mathcal{C}^\circ$ of curves on $S^\circ$.

The illustration below shows an example of an intersecting curve on the double~$S_{\mathcal{D}}$ of a genus-2 surface $S$ with two holes. In this case, the set $\mathcal{C}$ consists of the curves labeled $c_1$ and~$c_3$, while the set $\mathcal{C}^\circ$ consists of the curves labeled $c_2$ and~$c_4$. In Section~\ref{sec:TheCanonicalPairing}, we will see how to associate, to any element of $\mathcal{C}$ or $\mathcal{C}^\circ$, a polynomial $F_c$. For curves such as the ones below, which come from an intersecting curve, these $F_c$ are in fact the $F$-polynomials of Fomin and Zelevinsky~\cite{FZIV}. We will prove the following theorem, which expresses the function $\mathbb{I}_\mathcal{D}(l)$ in terms of these polynomials.
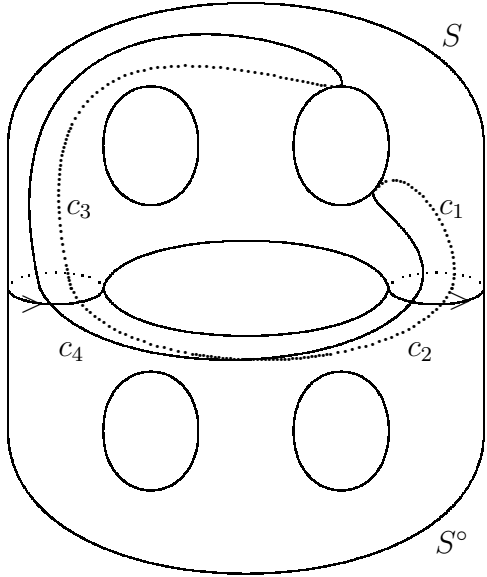
\[
\xy 0;/r.50pc/: 
(-6,-4.5)*\ellipse(3,1){.}; 
(-6,-4.5)*\ellipse(3,1)__,=:a(-180){-}; 
(6,-4.5)*\ellipse(3,1){.}; 
(6,-4.5)*\ellipse(3,1)__,=:a(-180){-}; 
(-13.5,-10)*{>};
(13.5,-9.8)*{>};
(-9,0)*{}="1";
(-3,0)*{}="2";
(3,0)*{}="3";
(9,0)*{}="4";
(-15,0)*{}="A2";
(15,0)*{}="B2";
"1";"2" **\crv{(-9,-5) & (-3,-5)};
"1";"2" **\crv{(-9,5) & (-3,5)};
"3";"4" **\crv{(3,-5) & (9,-5)};
"3";"4" **\crv{(3,5) & (9,5)};
"A2";"B2" **\crv{(-15,12) & (15,12)};
(-9,-9)*{}="A";
(9,-9)*{}="B";
"A";"B" **\crv{(-8,-5) & (8,-5)}; 
(-15,-9)*{}="A1";
(15,-9)*{}="B1";
"B2";"B1" **\dir{-}; 
"A2";"A1" **\dir{-};
"A";"B" **\crv{(-8,-13) & (8,-13)}; 
(-15,-18)*{}="A0";
(15,-18)*{}="B0";
"B1";"B0" **\dir{-}; 
"A1";"A0" **\dir{-};
"A0";"B0" **\crv{(-15,-30) & (15,-30)};
(-9,-18)*{}="5";
(-3,-18)*{}="6";
"5";"6" **\crv{(-9,-23) & (-3,-23)};
"5";"6" **\crv{(-9,-13) & (-3,-13)};
(3,-18)*{}="7";
(9,-18)*{}="8";
"7";"8" **\crv{(3,-23) & (9,-23)};
"7";"8" **\crv{(3,-13) & (9,-13)};
(-11,-9)*{}="V1";
(13,-9)*{}="V2";
(-13,-9)*{}="V3";
(11,-9)*{}="V4";
(6,3.8)*{}="V5";
(8,-3)*{}="V6";
"V1";"V2" **\crv{~*=<2pt>{.} (-8,-15) & (10,-15)};
"V3";"V4" **\crv{(-10,-15) & (8,-15)};
"V3";"V5" **\crv{(-18,13) & (7,7)};
"V1";"V5" **\crv{~*=<2pt>{.} (-16,12) & (4,3)};
"V2";"V6" **\crv{~*=<2pt>{.} (14,-6) & (10,0)};
"V4";"V6" **\crv{(12,-6) & (8,-4)};
(13,7)*{S}; 
(13,-25)*{S^\circ}; 
(13,-4)*{c_1}; 
(11,-13)*{c_2}; 
(-10.5,-4)*{c_3}; 
(-11,-13)*{c_4}; 
\endxy
\]

\begin{theorem}
\label{thm:intromain}
Let $l$ be a point of $\mathcal{D}_L(S,\mathbb{Z})$ and suppose that the orientation of any component of $\partial S$ that meets a curve of~$l$ agrees with the orientation of~$S$. Then 
\[
\mathbb{I}_\mathcal{D}(l)=\frac{\prod_{c\in\mathcal{C}^\circ}F_c(\widehat{X}_1,\dots,\widehat{X}_n)}{\prod_{c\in\mathcal{C}}F_c(X_1,\dots,X_n)}B_1^{g_{l,1}}\dots B_n^{g_{l,n}}X_1^{h_{l,1}}\dots X_n^{h_{l,n}}
\]
where the $F_c$ are polynomials, the $g_{l,i}$ are integers, the $h_{l,i}$ are half integers, and the $\widehat{X}_i$ are given by 
\[
\widehat{X}_i=X_i\prod_j B_j^{\varepsilon_{ij}}.
\]
\end{theorem}

All of the notation appearing in this theorem will be defined precisely below. As we will see, the above formula is in a sense a generalization of Corollary~6.3 in~\cite{FZIV}. If $c$ is a closed loop belonging to~$\mathcal{C}$ or~$\mathcal{C}^\circ$, then $F_c$ does not arise as one of Fomin and Zelevinsky's $F$-polynomials, but Theorem~\ref{thm:intromain} suggests that the $F_c$ should be viewed as ``generalized'' $F$-polynomials. Such generalized $F$-polynomials have appeared previously in the work of Musiker, Schiffler, and Williams~\cite{MW,MSW2} on cluster algebras associated to surfaces.

Because of the factors of $X_i^{1/2}$ in the above result, the function $\mathbb{I}_{\mathcal{D}}(l)$ is in general not expressible as a rational function in the $B_i$ and $X_i$. In addition to proving the above result, we will show that the canonical mapping provides rational functions for a certain class of $\mathcal{D}$-laminations.

\begin{theorem}
Let $l$ be a $\mathcal{D}$-lamination. Then $\mathbb{I}_{\mathcal{D}}(l)$ is a rational function in the variables $B_i$ and $X_i$ if and only if $l$ has integral coordinates.
\end{theorem}

In fact, we will prove a more detailed statement which also gives a homological condition for the coordinates $b_i$ and $x_i$ of the lamination~$l$ to be integers.

In the final part of this paper, we discuss a map 
\[
\mathcal{I}_{\mathcal{D}}:\mathcal{D}_L(S,\mathbb{Z})\times\mathcal{D}_L(S,\mathbb{Z})\rightarrow\mathbb{Q}
\]
which we call the \emph{intersection pairing}. Roughly speaking, if $l$ and $m$ are $\mathcal{D}$-laminations, we let $k$ be half the minimal number of intersections between the portions of $l$ and $m$ that lie on~$S$, we let $k^\circ$ be half the minimal number of intersections between the portions that lie on~$S^\circ$, and we define $\mathcal{I}_{\mathcal{D}}(l,m)=k^\circ-k$. In the final part of this paper, we will define this pairing more precisely and prove that it is the tropicalization of $\mathbb{I}_{\mathcal{D}}$ in an appropriate sense.

\subsection{Organization}

In Section~\ref{sec:BackgroundOnClusterVarieties}, we review the results that we will need from~\cite{IHES,ensembles,dual,P}, modifying the presentation to suit our purposes. We begin with a discussion of Teichm\"uller and lamination spaces and construct coordinates on these spaces. We explain how the properties of these coordinates lead to the definition of cluster varieties, and we explain the sense in which the Teichm\"uller and lamination spaces are dual. 

In Section~\ref{sec:TheSymplecticDouble}, we review the results of~\cite{dilog,double} and define the space~$\mathcal{D}^+(S)$. We construct coordinates on this space and derive formulas for transforming between different coordinate systems. We then introduce the symplectic double and describe the problem we wish to solve.

In Section~\ref{sec:D-laminations}, we define $\mathcal{D}$-laminations on the double of a surface. We construct coordinates on the space of all $\mathcal{D}$-laminations and derive formulas for transforming between different coordinate systems. We will see that the space of $\mathcal{D}$-laminations is a tropical version of the symplectic double. We introduce the space of real $\mathcal{D}$-laminations and describe a boundary for the Teichm\"uller $\mathcal{D}$-space using results of~\cite{infinity}.

In Section~\ref{sec:ClusterAlgebrasAndFPolynomials}, we review some results on cluster algebras and $F$-polynomials that play a role in our discussion of the canonical pairing. The references for this section are~\cite{FZIV,FST,MSW1}. In addition, we introduce a particular cluster algebra that we call the \emph{cluster $\mathcal{D}$-algebra}.

Finally, in Section~\ref{sec:TheCanonicalPairing}, we define a canonical pairing of $\mathcal{D}$-laminations with points of the Teichm\"uller $\mathcal{D}$-space. Our main result in this section is a formula expressing this pairing in terms of $F$-polynomials. We show that any lamination with integral coordinates provides a rational function on~$\mathcal{D}^+(S)$. We conclude by defining a map $\mathcal{I}_{\mathcal{D}}:\mathcal{D}_L(S,\mathbb{Z})\times\mathcal{D}_L(S,\mathbb{Z})\rightarrow\mathbb{Q}$ which we call the \emph{intersection pairing}. We show that it is the tropicalization of the canonical pairing~$\mathbb{I}_{\mathcal{D}}$ in an appropriate sense.

\section{Background on cluster varieties}
\label{sec:BackgroundOnClusterVarieties}

\subsection{Teichm\"uller spaces}

\subsubsection{Preliminaries on surfaces}

In this section, we review the basic theory of cluster varieties~\cite{IHES,ensembles,dual}. We begin by reviewing some geometric constructions on surfaces.

A \emph{decorated surface} is a smooth oriented surface of genus $g\geq0$ with $r\geq0$ punctures and $s\geq0$ smooth boundary components, where each smooth boundary component has finitely many marked points.

\begin{definition}
Let $S$ be a decorated surface. An \emph{ideal triangulation} $T$ of $S$ is a triangulation whose vertices are the marked points and the punctures.
\end{definition}

From now on, we will consider only decorated surfaces $S$ that admit an ideal triangulation. Note that in general the sides of a triangle in an ideal triangulation may not be distinct. In this case, the triangle is said to be \emph{self-folded}.

An edge of an ideal triangulation $T$ is called \emph{external} if it lies along the boundary of $S$, connecting two marked points, and is called \emph{internal} otherwise. We will write $J$ for the set of internal edges of the triangulation~$T$ and $I$ for the set of all edges. For a triangulation without self-folded triangles, we define a skew-symmetric matrix $\varepsilon_{ij}$ ($i$, $j\in I$) by the formula 
\begin{align*}
\varepsilon_{ij} &= \sum_{t\in T}\langle i, t, j\rangle
\end{align*}
where $\langle i, t, j\rangle$ equals $+1$ (respectively, $-1$) if $i$ and~$j$ are sides of the triangle $t$ and $i$ lies in the counterclockwise (respectively, clockwise) direction from $j$ with respect to their common vertex. Otherwise, we set $\langle i, t, j\rangle=0$.

If $k$ is an internal edge of the ideal triangulation $T$, then a \emph{flip} at $k$ is the transformation of $T$ that removes the edge $k$ and replaces it by the unique different edge that, together with the remaining edges, forms a new ideal triangulation:
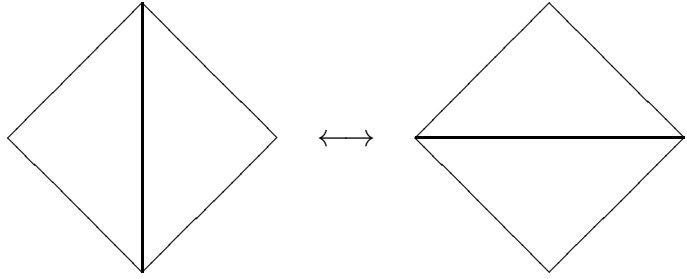
\[
\xy /l1.5pc/:
{\xypolygon4"A"{~:{(2,2):}}},
{"A1"\PATH~={**@{-}}'"A3"},
\endxy
\quad
\longleftrightarrow
\quad
\xy /l1.5pc/:
{\xypolygon4"A"{~:{(2,2):}}},
{"A2"\PATH~={**@{-}}'"A4"}
\endxy
\]
A flip will be called \emph{regular} if none of the triangles above is self-folded. It is a fact that any two isotopy classes of ideal triangulations on a surface are related by a sequence of flips. Note that there is a natural bijection between the edges of an ideal triangulation and the edges of the triangulation obtained by a flip at some edge. If we use this bijection to identify edges of the flipped triangulation with the set $I$, then it is straightforward to show that a flip at an edge~$k$ of an ideal triangulation changes the matrix~$\varepsilon_{ij}$ to the matrix
\[
\varepsilon_{ij}'=
\begin{cases}
-\varepsilon_{ij} & \mbox{if } k\in\{i,j\} \\
\varepsilon_{ij}+\frac{|\varepsilon_{ik}|\varepsilon_{kj}+\varepsilon_{ik}|\varepsilon_{kj}|}{2} & \mbox{if } k\not\in\{i,j\}.
\end{cases}
\]

\subsubsection{Teichm\"uller $\mathcal{A}$- and $\mathcal{X}$-spaces}

Recall that the classical \emph{Teichm\"uller space} $\mathcal{T}(S)$ of a punctured surface $S$ without marked points can be viewed as the quotient 
\[
\Hom'(\pi_1(S),PSL(2,\mathbb{R}))/PSL(2,\mathbb{R})
\]
where $\Hom'(\pi_1(S),PSL(2,\mathbb{R}))$ is the set of all discrete and faithful representations of $\pi_1(S)$ into $PSL(2,\mathbb{R})$ such that the image of a loop surrounding a puncture is parabolic. The group~$PSL(2,\mathbb{R})$ acts on this set by conjugation.

Let $\rho:\pi_1(S)\rightarrow PSL(2,\mathbb{R})$ be an element of the set described above. Then we can represent~$S$ as a quotient 
\[
S=\mathbb{H}/\Delta
\]
where $\mathbb{H}$ is the upper half plane and $\Delta=\rho(\pi_1(S))$ is a discrete subgroup of $PSL(2,\mathbb{R})$. By definition, the map $\rho$ takes any loop surrounding a puncture to a parabolic transformation. If $x\in\partial\mathbb{H}$ is the fixed point of the parabolic transformation corresponding to a puncture~$p$ in~$S$, then a horocycle in $\mathbb{H}$ centered at $x$ projects to a curve on $S$ which we also call a \emph{horocycle} at $p$.

\begin{definition}
If $S$ is a decorated surface with no marked points, then we define the \emph{Teichm\"uller $\mathcal{A}$-space} $\mathcal{A}^+(S)$ to be the space that parametrizes pairs $(\rho,\mathcal{S})$ where $\rho$ is a point of $\mathcal{T}(S)$ and $\mathcal{S}$ is a set of horocycles, one at each puncture.
\end{definition}

More generally, suppose that $S$ is any decorated surface. Delete the marked points on the boundary of $S$ and double the resulting surface along its boundary arcs. This produces a punctured surface $S'$ where each marked point in the original surface gives rise to a puncture in~$S'$. The doubled surface $S'$ comes equipped with a natural involution $\iota:S'\rightarrow S'$.

\begin{definition}
The \emph{Teichm\"uller $\mathcal{A}$-space} $\mathcal{A}^+(S)$ is defined as the $\iota$-invariant subspace of $\mathcal{A}^+(S')$.
\end{definition}

Note that this space can be identified with the one defined previously in the special case where there are no marked points on $S$. We write $\mathcal{A}_0^+(S)$ for the set of points in $\mathcal{A}^+(S)$ such that if $e$ is the segment of $\partial S$ between two marked points, then the horocycles at the ends of~$e$ are tangent. When there is no possibility of confusion, we will simply write $\mathcal{A}^+$ and~$\mathcal{A}_0^+$.

The other version of Teichm\"uller space that we will consider parametrizes more general surface group representations. More precisely, we consider the set 
\[
\Hom''(\pi_1(S),PSL(2,\mathbb{R}))/PSL(2,\mathbb{R})
\]
where $\Hom''(\pi_1(S),PSL(2,\mathbb{R}))$ is the set of all discrete and faithful representations of $\pi_1(S)$ into $PSL(2,\mathbb{R})$ such that the image of a loop surrounding a puncture is either parabolic or hyperbolic. The group~$PSL(2,\mathbb{R})$ again acts by conjugation.

Suppose we are given a representation $\rho:\pi_1(S)\rightarrow PSL(2,\mathbb{R})$ in the set described above. A puncture $p$ in the surface $S$ will be called a \emph{hole} if $\rho$ maps the homotopy class of a loop surrounding $p$ to a hyperbolic transformation.

\begin{definition}
If $S$ is a decorated surface with no marked points, then we define the \emph{Teichm\"uller $\mathcal{X}$-space} $\mathcal{X}^+(S)$ to be the space that parametrizes pairs $(\rho,\mathcal{S})$ where $\rho$ is an element of the above quotient and $\mathcal{S}$ is a set of orientations, one for each hole.
\end{definition}

More generally, if $S$ is any decorated surface, then by doubling $S$, we can construct the punctured surface $S'$ with the natural involution $\iota:S'\rightarrow S'$ as before. One then has the following definition.

\begin{definition}
The \emph{Teichm\"uller $\mathcal{X}$-space} $\mathcal{X}^+(S)$ of $S$ is the $\iota$-invariant subspace of $\mathcal{X}^+(S')$.
\end{definition}

This space can be identified with the one defined previously in the special case where there are no marked points on $S$. When there is no possibility of confusion, we will simply write~$\mathcal{X}^+$.

\subsubsection{Construction of coordinates}

To construct coordinates on the Teichm\"uller $\mathcal{A}$-space, fix a point $m\in\mathcal{A}^+$ and let $i$ be an edge of an ideal triangulation. The point $m$ allows us to write $S'$ as a quotient $S'=\mathbb{H}/\Delta$ where $\Delta$ is a discrete subgroup of $PSL(2,\mathbb{R})$. We can then deform $i$ into a geodesic and lift this geodesic to the upper half plane $\mathbb{H}$. By definition of the Teichm\"uller $\mathcal{A}$-space, we have horocycles at the ends of the resulting geodesic in $\mathbb{H}$. We define $A_i$ as the exponentiated half length (respectively, negative half length) of the segment of the lifted curve between the intersection points with the horocycles if these horocycles do not intersect (respectively, if they do intersect).
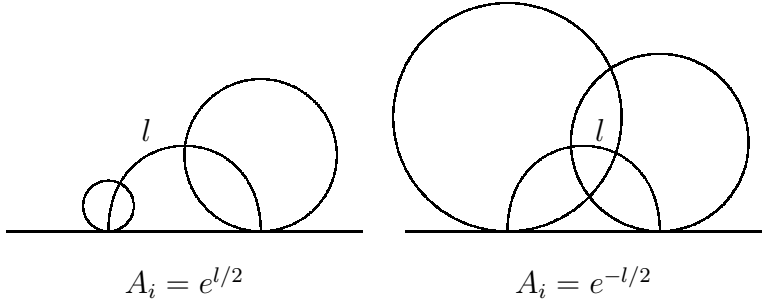
\[
\xy 0;/r.40pc/: 
(0,-6)*{}="1"; 
(12,-6)*{}="2"; 
"1";"2" **\crv{(0,3) & (12,3)}; 
(-8,-6)*{}="X"; 
(20,-6)*{}="Y"; 
"X";"Y" **\dir{-}; 
(0,-4)*\xycircle(2,2){-};
(12,0)*\xycircle(6,6){-};
(6,-10)*{A_i=e^{l/2}}; 
(3,2)*{l}; 
\endxy
\quad
\xy 0;/r.40pc/: 
(0,-6)*{}="1"; 
(12,-6)*{}="2"; 
"1";"2" **\crv{(0,3) & (12,3)}; 
(-8,-6)*{}="X"; 
(20,-6)*{}="Y"; 
"X";"Y" **\dir{-}; 
(0,3)*\xycircle(9,9){-};
(12,1)*\xycircle(7,7){-};
(7.3,2)*{l}; 
(6,-10)*{A_i=e^{-l/2}}; 
\endxy
\]
Doing this for every edge of an ideal triangulation of $S$, we get a collection of numbers $A_i$ ($i\in I$) corresponding to the point $m$.

\begin{proposition}
The numbers $A_i$ ($i\in I$) provide a bijection 
\[
\mathcal{A}^+(S)\rightarrow\mathbb{R}_{>0}^{|I|}.
\]
\end{proposition}

\begin{proof}
We will construct an inverse to this map. Let us suppose that we are given a positive number $A_i$ for each edge $i\in I$. Let $\tilde{S}$ denote the topological universal cover of $S$. Then we can lift the ideal triangulation of $S$ to a triangulation of $\tilde{S}$, and we can associate to each edge of this triangulation the number associated to its projection.

Let $t_0$ be any triangle in the triangulation of~$\tilde{S}$. By~\cite{P}, Chapter~1, Corollary~4.8, there exists an ideal triangle $u_0$ in $\mathbb{H}$ and horocycles around the endpoints of $u_0$ realizing the $A$-coordinates associated to the edges of $t_0$. Next, consider a triangle $t$ adjacent to $t_0$ in the triangulation of~$\tilde{S}$. The common edge $t\cap t_0$ corresponds to an edge in~$\mathbb{H}$ with horocycles around its endpoints. By~\cite{P}, Chapter~1, Lemma~4.14, there is a unique ideal triangle $u$ in~$\mathbb{H}$ adjacent to $u_0$ with horocycles around its endpoints so that these horocycles agree with the ones already constructed and realize the $A$-coordinates associated to the edges of $t$.

Continuing in this way, we obtain a collection of ideal triangles in~$\mathbb{H}$ where each triangle corresponds to a triangle in~$\tilde{S}$. Now any element $\gamma\in\pi_1(S)$ corresponds to a deck transformation of $\tilde{S}$, and there is a unique element of $PSL(2,\mathbb{R})$ that realizes this deck transformation as an isometry of $\mathbb{H}$ preserving the triangulation. In this way, we obtain a representation $\rho:\pi_1(S)\rightarrow PLS(2,\mathbb{R})$. One can check that this construction provides a two-sided inverse of the map $\mathcal{A}^+(S)\rightarrow\mathbb{R}_{>0}^{|I|}$.
\end{proof}

Now consider a point $m\in\mathcal{X}^+$ and an ideal triangulation $T$. Modify the neighborhood of each hole in $S$ to get a new surface~$S'$ with geodesic boundary. The edges of $T$ correspond to arcs on this surface $S'$, and there is a canonical way to deform these arcs. If an arc ends on a hole, we wind its endpoint around the hole infinitely many times in the direction prescribed by the orientation so that the arcs spiral into the holes in~$S'$. More precisely, each geodesic boundary component of~$S'$ lifts to a geodesic in~$\mathbb{H}$ and we deform the preimage of an edge by dragging its endpoints along these geodesics until they coincide with points of~$\partial\mathbb{H}$.

Once we have deformed the edges of our triangulation in this way, we can lift the triangulation to the upper half plane to get a collection of ideal triangles. Let $k$ be any internal edge and consider the two ideal triangles that share this edge. Together they form an ideal quadrilateral, and we number the vertices of this ideal quadrilateral in counterclockwise order as shown below so that $k$ joins vertices 1 and~3.
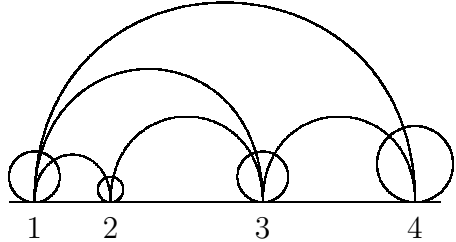
\[
\xy 0;/r.40pc/: 
(-6,-8)*{1}; 
(0,-8)*{2}; 
(12,-8)*{3}; 
(24,-8)*{4}; 
(-6,-6)*{}="1"; 
(0,-6)*{}="2"; 
(12,-6)*{}="3"; 
(24,-6)*{}="4"; 
"1";"2" **\crv{(-6,-1) & (0,-1)}; 
"2";"3" **\crv{(0,3) & (12,3)}; 
"3";"4" **\crv{(12,3) & (24,3)}; 
"1";"4" **\crv{(-6,15) & (24,15)}; 
"1";"3" **\crv{(-6,8) & (12,8)}; 
(-6,-4)*\xycircle(2,2){-};
(0,-5)*\xycircle(1,1){-};
(12,-4)*\xycircle(2,2){-};
(24,-3)*\xycircle(3,3){-};
(-8,-6)*{}="X"; 
(26,-6)*{}="Y"; 
"X";"Y" **\dir{-}; 
\endxy
\]

Choose a horocycle at each vertex, and put $A_{ij}=e^{l_{ij}/2}$ where $l_{ij}$ is the signed length of the segment between the horocycles at $i$ and $j$. We then define the \emph{cross ratio} 
\[
X_k=\frac{A_{12}A_{34}}{A_{14}A_{23}}.
\]
It is easy to see that there are two ways of numbering the vertices, and both give the same value for the cross ratio. One can also show that the numbers $X_k$ ($k\in J$) are independent of the chosen horocycles. They are the numbers that we associate to the point~$m$. (Note that our quantities $X_k$ are inverse to the coordinates defined in some sources, for example~\cite{P}, Chapter~2, Definition~4.1.)

\begin{proposition}
The numbers $X_j$ ($j\in J$) provide a bijection 
\[
\mathcal{X}^+(S)\rightarrow\mathbb{R}_{>0}^{|J|}.
\]
\end{proposition}

\begin{proof}
Suppose we are given a positive number $X_j$ for each edge $j\in J$. To recover the orientation of a hole~$h$, we compute the number $\sum\log X_j$, where the sum is over all edges~$j$ incident to~$h$. This sum is negative (respectively, positive) when the orientation is induced from the orientation of $S$ (respectively, the opposite of this orientation).

Let $\tilde{S}$ denote the topological universal cover of the surface $S$. We can lift the ideal triangulation of $S$ to an ideal triangulation of $\tilde{S}$, and we can associate to each edge of this triangulation the number associated to its projection.

Let $t_0$ be any triangle in the triangulation of~$\tilde{S}$, and choose a corresponding ideal triangle $u_0$ in $\mathbb{H}$. Next, consider a triangle $t$ adjacent to $t_0$ in the triangulation of~$\tilde{S}$. There is an $X$-coordinate associated to the common edge $t\cap t_0$. Since the cross ratio is a complete invariant of four points on~$\partial\mathbb{H}$, there is a unique ideal triangle $u$ in~$\mathbb{H}$ adjacent to $u_0$ so that these triangles realize the $X$-coordinate associated to the common edge.

Continuing in this way, we obtain a triangulation of a region in~$\mathbb{H}$ by geodesic triangles where each triangle corresponds to a triangle in~$\tilde{S}$. Now any element $\gamma\in\pi_1(S)$ determines a deck transformation of $\tilde{S}$, and there is a unique element of $PSL_2(\mathbb{R})$ that realizes this deck transformation as an isometry of $\mathbb{H}$ preserving the triangulation. In this way, we obtain a representation $\rho:\pi_1(S)\rightarrow PSL_2(\mathbb{R})$. One can check that this construction provides a two-sided inverse of the map $\mathcal{X}^+(S)\rightarrow\mathbb{R}_{>0}^{|J|}$.
\end{proof}

One can show that a regular flip at an edge $k$ of the ideal triangulation changes the coordinates~$A_i$ and~$X_i$ to new coordinates~$A_i'$ and~$X_i'$ given by the formulas 
\[
A_i' =
\begin{cases}
A_k^{-1}\biggr(\prod_{j|\varepsilon_{kj>0}}A_j^{\varepsilon_{kj}} + \prod_{j|\varepsilon_{kj<0}}A_j^{-\varepsilon_{kj}}\biggr) & \mbox{if } i=k \\
A_i & \mbox{if } i\neq k
\end{cases}
\]
and
\[
X_i'=
\begin{cases}
X_k^{-1} & \mbox{if } i=k \\
X_i{(1+X_k^{-\sgn(\varepsilon_{ik})})}^{-\varepsilon_{ik}} & \mbox{if } i\neq k.
\end{cases}
\]

The canonical map $p:\mathcal{A}^+\rightarrow\mathcal{X}^+$ which forgets the horocycles is given in terms of coordinates and the matrix $\varepsilon_{ij}$ by the formula $p^*(X_i)=\prod_{j\in I}A_j^{\varepsilon_{ij}}$.

\subsection{Lamination spaces}

\subsubsection{$\mathcal{A}$- and $\mathcal{X}$-laminations}

Let $S$ be a decorated surface. By a \emph{curve} on $S$, we mean an embedding $C\rightarrow S$ of a compact, connected, one-dimensional manifold $C$ with (possibly empty) boundary into~$S$. We require that any endpoints of $C$ map to punctures or points on the boundary of~$S$ disjoint from the marked points. When we talk about homotopies, we mean homotopies within the class of such curves. A curve is called \emph{special} if it is retractable to a puncture or to an interval on~$\partial S$ containing exactly one marked point. A curve is \emph{contractible} if it can be retracted to a point within this class of curves.

\begin{definition}
A \emph{rational $\mathcal{A}$-lamination} on $S$ is the homotopy class of a collection of finitely many nonintersecting noncontractible curves on $S$, either closed or ending on a segment of the boundary bounded by adjacent marked points, with rational weights and subject to the following conditions and equivalence relations:
\begin{enumerate}
\item The weight of a curve is nonnegative unless the curve is special.
\item A lamination containing a curve of weight zero is equivalent to the lamination with this curve removed.
\item A lamination containing homotopic curves of weights $a$ and $b$ is equivalent to the lamination with one curve removed and the weight $a+b$ on the other.
\end{enumerate}
\end{definition}

The set of all rational $\mathcal{A}$-laminations on $S$ is denoted $\mathcal{A}_L(S,\mathbb{Q})$. We will write $\mathcal{A}_L(S,\mathbb{Z})$ for the set of all $\mathcal{A}$-laminations on $S$ that can be represented by a collection of curves with integral weights. We write $\mathcal{A}_L^0(S,\mathbb{Z})$ for the subset of all laminations such that if $e$ is the segment of $\partial S$ between two marked points, then the total weight of the curves ending on $e$ vanishes. When there is no possibility of confusion, we will simply write $\mathcal{A}_L$ and $\mathcal{A}_L^0$.

\begin{definition}
A \emph{rational $\mathcal{X}$-lamination} on $S$ is the homotopy class of a collection of finitely many nonintersecting noncontractible and non-special curves on~$S$ with positive rational weights and a choice of orientation for each puncture in $S$ that meets a curve. A lamination containing homotopic curves of weights $a$ and $b$ is equivalent to the lamination with one curve removed and the weight $a+b$ on the other.
\end{definition}

The set of all rational $\mathcal{X}$-laminations on $S$ is denoted $\mathcal{X}_L(S,\mathbb{Q})$. We will write $\mathcal{X}_L(S,\mathbb{Z})$ for the set of all $\mathcal{X}$-laminations on $S$ that can be represented by a collection of curves with integral weights. When there is no possibility of confusion, we will simply write $\mathcal{X}_L$.

\subsubsection{Construction of coordinates}

To construct coordinates on $\mathcal{A}_L(S,\mathbb{Q})$, fix an $\mathcal{A}$-lamination $l$ and an ideal triangulation of~$S$. Deform the curves of $l$ so that each curve intersects each edge of the triangulation in the minimal number of points. Then we define $a_i$ to be half the total weight of curves that intersect the edge $i$.

\begin{proposition}
The numbers $a_i$ ($i\in I$) provide a bijection 
\[
\mathcal{A}_L(S,\mathbb{Q})\rightarrow\mathbb{Q}^{|I|}.
\]
\end{proposition}

\begin{proof}
Suppose we are given a rational number $a_i$ for each edge $i\in I$. It is enough to construct an $\mathcal{A}$-lamination with coordinates 
\[
\tilde{a_i}=pa_i+q
\]
for some rational numbers $p$ and~$q$. Indeed, if we construct a lamination with these coordinates, then we can add a special curve of weight $-q$ around each puncture and each marked point and then divide the weight of each curve by $p$ to get a lamination with coordinates~$a_i$.

Let $\tilde{S}$ denote the universal cover of $S$. Then we can lift the triangulation of $S$ to a triangulation of $\tilde{S}$, and we can associate to each edge of this triangulation the number $\tilde{a_i}$ associated to its projection.

Let $t$ be any triangle in the triangulation of $\tilde{S}$. By an argument in~\cite{dual}, we can fix $p$ and $q$ so that there exists a corresponding topological triangle $u$ with finitely many nonintersecting curves joining each pair of adjacent edges where the edge corresponding to $i$ intersects the curves exactly $2\tilde{a_i}$ times. By gluing the triangles $u$ and matching the curves at each edge, we recover the universal cover of $S$ together with a collection of curves. Quotienting the resulting space by the group of deck transformations, we obtain the desired $\mathcal{A}$-lamination on the surface~$S$.
\end{proof}

Now consider an $\mathcal{X}$-lamination $l$ and an ideal triangulation $T$. Remove a small neighborhood of each puncture in $S$ that meets a curve to get a new surface $S'$ with boundary. The edges of $T$ correspond to arcs on this surface $S'$. If an arc ends on a hole, we wind its endpoint around the hole infinitely many times in the direction prescribed by the orientation so that the arcs spiral into the holes in~$S'$.

Let $k$ be an internal edge of the resulting triangulation of $S'$, and consider the quadrilateral on the surface with diagonal $k$. There will be finitely many curves that connect opposite sides of the quadrilateral and possibly infinitely many curves that join adjacent sides. Number the vertices of this quadrilateral in counterclockwise order as shown below so that the edge $k$ joins vertices~1 and~3.
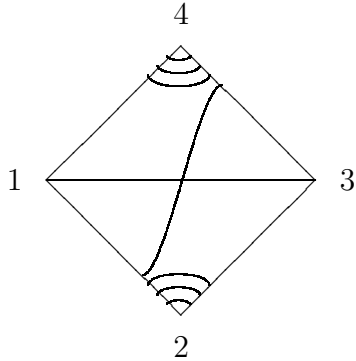
\[
\xy /l1.5pc/:
{\xypolygon4"A"{~:{(2,2):}}};
{\xypolygon4"B"{~:{(2.5,0):}~>{}}};
{\xypolygon4"C"{~:{(0.8,0.8):}~>{}}};
{"A2"\PATH~={**@{-}}'"A4"};
(4.5,0)*{1};
(1,3.5)*{2};
(-2.5,0)*{3}; 
(1,-3.5)*{4};
(1.8,2)*{}="A0";
(0.4,-2.2)*{}="A1";
(1.7,-2.2)*{}="A2";
(0.6,-2.4)*{}="A3";
(1.5,-2.4)*{}="A4";
(0.8,-2.6)*{}="A5";
(1.3,-2.6)*{}="A6";
(0.15,-2)*{}="B0";
(0.4,2.2)*{}="B1";
(1.7,2.2)*{}="B2";
(0.6,2.4)*{}="B3";
(1.5,2.4)*{}="B4";
(0.8,2.6)*{}="B5";
(1.3,2.6)*{}="B6";
"A0";"B0" **\crv{(1.3,1.9) & (0.65,-1.9)};
"A1";"A2" **\crv{(0.5,-1.9) & (1.6,-1.9)};
"A3";"A4" **\crv{(0.7,-2.2) & (1.4,-2.2)};
"A5";"A6" **\crv{(0.9,-2.5) & (1.2,-2.5)};
"B1";"B2" **\crv{(0.5,1.9) & (1.6,1.9)};
"B3";"B4" **\crv{(0.7,2.2) & (1.4,2.2)};
"B5";"B6" **\crv{(0.9,2.5) & (1.2,2.5)};
\endxy
\]

Let $p$ be a vertex of this quadrilateral. If there are infinitely many curves connecting the edges that meet at $p$, then we can choose one such curve $\alpha_p$ and delete all of the curves between $\alpha_p$ and the point $p$. By doing this for each vertex, we remove all but finitely many curves from the quadrilateral and get an $\mathcal{A}$-lamination on a disk with four marked points. Let $a_{ij}$ be the coordinate of this $\mathcal{A}$-lamination corresponding to the edge connecting vertices $i$ and $j$. Define 
\[
x_k=a_{12}+a_{34}-a_{14}-a_{23}.
\]
It is easy to see that this number is independent of all choices made in the construction. It is the number associated to the edge $k$.

\begin{proposition}
The numbers $x_j$ ($j\in J$) provide a bijection 
\[
\mathcal{X}_L(S,\mathbb{Q})\rightarrow\mathbb{Q}^{|J|}.
\]
\end{proposition}

\begin{proof}
Suppose we are given a rational number $x_j$ for each edge $j\in J$. To recover the orientation of a puncture $h$, we compute the number $\sum x_i$, where the sum is over all edges $j$ incident to~$h$. This sum is negative (respectively, positive) when the orientation is induced from the orientation of~$S$ (respectively, the opposite of this orientation).

Let $\tilde{S}$ denote the universal cover of the surface $S$. We can lift the ideal triangulation of~$S$ to an ideal triangulation of~$\tilde{S}$, and we can associate to each edge of this triangulation the number associated to its projection.

Let $t_0$ be any triangle in the triangulation of $\tilde{S}$, and choose a corresponding ideal triangle $u_0$ in the hyperbolic plane $\mathbb{H}$. There is a unique triple of horocycles about its endpoints that are pairwise tangent. The points of tangency provide three canonical points on the edges of~$u_0$. Parametrize the edges of this triangle by $\mathbb{R}$, respecting the orientation induced by the orientation of the triangle, so that the point with parameter $s\in\mathbb{R}$ lies at distance $|s|$ from the distinguished point. Connect points with parameter $s\in\frac{1}{2}+\mathbb{Z}_{\geq0}$ on one side to points with parameter~$-s$ on the next side in the clockwise direction by a horocyclic arc. In this way, we obtain the triangle $u_0$ with infinitely many arcs connecting adjacent sides.
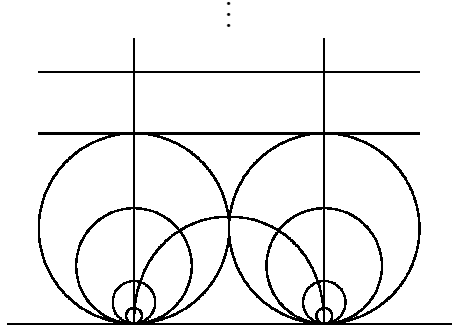
\[
\xy 0;/r.50pc/: 
(6,14)*{\vdots}; 
(0,-6)*{}="2"; 
(12,-6)*{}="3"; 
"2";"3" **\crv{(0,3) & (12,3)}; 
(0,12)*{}="B"; 
(12,12)*{}="C";
(0,-6)*{}="B1"; 
(12,-6)*{}="C1";
(-8,-6)*{}="X"; 
(20,-6)*{}="Y"; 
"B";"B1" **\dir{-}; 
"C";"C1" **\dir{-}; 
"X";"Y" **\dir{-}; 
(-6,6)*{}="X1"; 
(18,6)*{}="Y1"; 
"X1";"Y1" **\dir{-}; 
(-6,9.892)*{}="X2"; 
(18,9.892)*{}="Y2"; 
"X2";"Y2" **\dir{-}; 
(0,0)*\xycircle(6,6){-};
(0,-2.361)*\xycircle(3.639,3.639){-};
(0,-4.661)*\xycircle(1.339,1.339){-};
(0,-5.507)*\xycircle(0.493,0.493){-};
(12,0)*\xycircle(6,6){-};
(12,-2.361)*\xycircle(3.639,3.639){-};
(12,-4.661)*\xycircle(1.339,1.339){-};
(12,-5.507)*\xycircle(0.493,0.493){-};
\endxy
\]

Next, consider a triangle $t$ adjacent to $t_0$ in the ideal triangulation of $\tilde{S}$. There is a number $x_j$ associated to the common edge $j$, and we can find an ideal triangle $u$ adjacent to~$u_0$ so that the cross ratio of the resulting quadrilateral is $e^{x_j}$. We can draw infinitely many horocyclic arcs on $u$ as we did for $u_0$. By~\cite{P}, Chapter~1, Corollary~4.16, these arcs connect to the ones already drawn on $u_0$.

Continuing in this way, we obtain a triangulation of a region in $\mathbb{H}$ by ideal triangles where each triangle corresponds to a triangle in $\tilde{S}$. Quotienting this region by the action of the fundamental group, we obtain a surface homeomorphic to $S$ with curves drawn on it. One can check that this construction provides a two-sided inverse of the map $\mathcal{X}_L(S)\rightarrow\mathbb{Q}^{|J|}$.
\end{proof}

One can show that a regular flip at an edge $k$ of the ideal triangulation changes the coordinates~$a_i$ and~$x_i$ to new coordinates~$a_i'$ and~$x_i'$ given by the formulas 
\[
a_i' =
\begin{cases}
\max\biggr(\sum_{j|\varepsilon_{kj}>0}\varepsilon_{kj}a_j,-\sum_{j|\varepsilon_{kj}<0}\varepsilon_{kj}a_j\biggr)-a_k & \mbox{if } i=k \\
a_i & \mbox{if } i\neq k 
\end{cases}
\]
and
\[
x_i' =
\begin{cases}
-x_k & \mbox{if } i=k \\ 
x_i+\varepsilon_{ki}\max\left(0,\sgn(\varepsilon_{ki})x_k\right) & \mbox{if } i\neq k.
\end{cases}
\]

Since these transformation rules are continuous with respect to the standard topologies on $\mathbb{Q}^{|I|}$ and $\mathbb{Q}^{|J|}$, the coordinates define natural topologies on the spaces of $\mathcal{A}$- and $\mathcal{X}$-laminations. We define the space of \emph{real $\mathcal{A}$-laminations} or \emph{real $\mathcal{X}$-laminations} as the metric space completion of the corresponding space of rational laminations.

The canonical map $p:\mathcal{A}_L\rightarrow\mathcal{X}_L$ which forgets the special curves is given in terms of coordinates and the matrix $\varepsilon_{ij}$ by the formula $p^*(x_i)=\sum_{j\in I}\varepsilon_{ij}a_j$.

\subsection{Cluster varieties}

\subsubsection{$\mathcal{A}$- and $\mathcal{X}$-varieties}

We will now define the cluster $\mathcal{A}$- and $\mathcal{X}$-varieties. These are algebro-geometric objects defined using the transformation formulas given above.

\begin{definition}
A \emph{seed} $\mathbf{i}=(I,J,\varepsilon_{ij},d_i)$ consists of a finite set $I$, a subset $J\subseteq I$, a $\mathbb{Q}$-valued function $\varepsilon_{ij}$ on $I\times I$ such that $\varepsilon_{ij}\in\mathbb{Z}$ if $i\in J$ or $j\in J$, and positive rational numbers $d_i$~($i\in I$) such that $\widehat{\varepsilon}_{ij}=\varepsilon_{ij}d_j^{-1}$ is skew-symmetric. The function $\varepsilon_{ij}$ is called the \emph{exchange function}, and the set $I-J$ is called the set of \emph{frozen elements} of $I$.
\end{definition}

Given a seed $\mathbf{i}=(I,J,\varepsilon_{ij},d_i)$, we get two split algebraic tori $\mathcal{X}_\mathbf{i} = (\mathbb{G}_m)^{|J|}$ and $\mathcal{A}_\mathbf{i} = (\mathbb{G}_m)^{|I|}$ called the \emph{seed $\mathcal{X}$-torus} and \emph{seed $\mathcal{A}$-torus}, respectively. Let $\{X_j\}$ be the natural coordinates on the seed $\mathcal{X}$-torus and $\{A_i\}$ the natural coordinates on the seed $\mathcal{A}$-torus.

\begin{definition}
Let $\mathbf{i}=(I,J,\varepsilon_{ij},d_i)$ be a seed and $k\in J$ a non-frozen element. Then we define a new seed $\mu_k(\mathbf{i})=\mathbf{i}'=(I',J',\varepsilon_{ij}',d_i')$, called the seed obtained by \emph{mutation} in the direction $k$ by setting $I'=I$, $J'=J$, $d_i'=d_i$, and 
\[
\varepsilon_{ij}'=
\begin{cases}
-\varepsilon_{ij} & \mbox{if } k\in\{i,j\} \\
\varepsilon_{ij}+\frac{|\varepsilon_{ik}|\varepsilon_{kj}+\varepsilon_{ik}|\varepsilon_{kj}|}{2} & \mbox{if } k\not\in\{i,j\}.
\end{cases}
\]
\end{definition}

A seed mutation also induces birational maps on the seed $\mathcal{A}$- and $\mathcal{X}$-tori defined by the formulas 
\[
\mu_k^*A_i' =
\begin{cases}
A_k^{-1}\biggr(\prod_{j|\varepsilon_{kj>0}}A_j^{\varepsilon_{kj}} + \prod_{j|\varepsilon_{kj<0}}A_j^{-\varepsilon_{kj}}\biggr) & \mbox{if } i=k \\
A_i & \mbox{if } i\neq k
\end{cases}
\]
and
\[
\mu_k^*X_i'=
\begin{cases}
X_k^{-1} & \mbox{if } i=k \\
X_i{(1+X_k^{-\sgn(\varepsilon_{ik})})}^{-\varepsilon_{ik}} & \mbox{if } i\neq k
\end{cases}
\]
where $A_i'$ and $X_i'$ are the coordinates on $\mathcal{A}_{\mathbf{i}'}$ and $\mathcal{X}_{\mathbf{i}'}$.

Two seeds will be called \emph{mutation equivalent} if they are related by a sequence of mutations. We will denote the mutation equivalence class of a seed $\mathbf{i}$ by $|\mathbf{i}|$. A  transformation of the $\mathcal{A}$- or $\mathcal{X}$-tori obtained by composing the above birational maps is called a \emph{cluster transformation}.

\begin{definition}
The \emph{cluster $\mathcal{A}$-variety} $\mathcal{A}=\mathcal{A}_{|\mathbf{i}|}$ is a scheme over $\mathbb{Z}$ obtained by gluing all seed $\mathcal{A}$-tori for seeds mutation equivalent to the seed $\mathbf{i}$ using the above birational maps. The \emph{cluster $\mathcal{X}$-variety} $\mathcal{X}=\mathcal{X}_{|\mathbf{i}|}$ is obtained by gluing all seed $\mathcal{X}$-tori for seeds mutation equivalent to the seed $\mathbf{i}$ using the above birational maps.
\end{definition}

There is a canonical mapping $p:\mathcal{A}\rightarrow\mathcal{X}$ given in any cluster coordinate system by $p^*(X_i)=\prod_{j\in I}A_j^{\varepsilon_{ij}}$.

\subsubsection{Positive real and tropical points}

We can now explain how the Teichm\"uller and lamination spaces arise as the positive real and tropical points of cluster varieties.

\begin{definition}
A \emph{semifield} $\mathbb{P}$ is a set equipped with binary operations $+$ and $\cdot$ such that $+$ is commutative and associative, $\mathbb{P}$ is an abelian group under $\cdot$, and the usual distributive law holds: $(a+b)\cdot c=a\cdot c+b\cdot c$ for all~$a$,~$b$,~$c\in\mathbb{P}$.
\end{definition}

\begin{example}
The following examples of a semifields will play an important role in our discussion.
\begin{enumerate}
\item Let $\mathbb{P}=\mathbb{Q}_{>0}$ or $\mathbb{P}=\mathbb{R}_{>0}$, the set of positive rational or real numbers. Then $\mathbb{P}$ is a semifield under the usual operations of addition and multiplication.
\item Let $\mathbb{P}=\mathbb{Z}$, $\mathbb{Q}$, or~$\mathbb{R}$. Then $\mathbb{P}$ is a semifield with the operations $\oplus$ and $\otimes$ given by 
\[
a\oplus b=\max(a,b), \quad a\otimes b=a+b
\]
for all $a$, $b\in\mathbb{P}$. The semifields defined in this way are called \emph{tropical semifields} in the works of Fock and~Goncharov. They are denoted $\mathbb{Z}^t$, $\mathbb{Q}^t$, and~$\mathbb{R}^t$.
\item Let $\mathbb{P}=\mathbb{Q}_{\mathrm{sf}}(u_1,\dots,u_n)$ be the set of subtraction-free rational functions in the variables $u_1,\dots,u_n$. This set consists of all rational functions in $u_1,\dots,u_n$ that are expressible as a ratio of two polynomials with positive integral coefficients. It is a semifield whose operations are ordinary addition and multiplication of rational functions.
\item Let $\mathbb{P}=\mathrm{Trop}(y_1,\dots,y_n)$ be the free multiplicative abelian group generated by $y_1,\dots,y_n$ with the auxiliary addition defined by 
\[
\prod_{i=1}^n y_i^{a_i}\oplus \prod_{i=1}^n y_i^{b_i}=\prod_{i=1}^n y_i^{\min(a_i,b_i)}.
\]
This operation makes $\mathbb{P}$ into a semifield which Fomin and Zelevinsky call the \emph{tropical semifield} generated by $y_1,\dots,y_n$.
\end{enumerate}
\end{example}

Given a semifield $\mathbb{P}$ and a split algebraic torus $H$, we can form the set $H(\mathbb{P})=X_*(H)\otimes_\mathbb{Z}\mathbb{P}$. Here $X_*(H)$ is the group of cocharacters of $H$ and we are using the abelian group structure of $\mathbb{P}$. For example, we have the set $\mathcal{A}_\mathbf{i}(\mathbb{P})$ obtained from the seed $\mathcal{A}$-torus $\mathcal{A}_\mathbf{i}$. Now the maps $\psi_{\mathbf{i},\mathbf{i}'}:\mathcal{A}_\mathbf{i}\rightarrow\mathcal{A}_{\mathbf{i}'}$ that we used to glue the seed $\mathcal{A}$-tori induce maps $\psi_{\mathbf{i},\mathbf{i}'*}:\mathcal{A}_\mathbf{i}(\mathbb{P})\rightarrow\mathcal{A}_{\mathbf{i}'}(\mathbb{P})$, so we can define the quotient 
\[
\mathcal{A}(\mathbb{P})=\coprod \mathcal{A}_\mathbf{i}(\mathbb{P})/(\text{identifications }\psi_{\mathbf{i},\mathbf{i}'*}).
\]
The set of $\mathbb{P}$-points $\mathcal{X}(\mathbb{P})$ of the cluster $\mathcal{X}$-variety is defined similarly.

Since the coordinates of the Teichm\"uller $\mathcal{A}$- and $\mathcal{X}$-spaces are positive by construction, we see that the sets $\mathcal{A}(\mathbb{R}_{>0})$ and $\mathcal{X}(\mathbb{R}_{>0})$ are identified with the Teichm\"uller $\mathcal{A}$- and $\mathcal{X}$-spaces. By examining how the coordinates transform under mutation, we see also that the sets $\mathcal{A}(\mathbb{Q}^t)$ and $\mathcal{X}(\mathbb{Q}^t)$ are identified with the $\mathcal{A}$- and $\mathcal{X}$-lamination spaces.

One can show that $l\in\mathcal{X}_L(S,\mathbb{Z})$ if and only if $l$ has integral coordinates, and hence 
\[
\mathcal{X}_L(S,\mathbb{Z})=\mathcal{X}(\mathbb{Z}^t).
\]
On the other hand, one can show that $l\in\mathcal{A}_L(S,\mathbb{Z})$ if and only if $l$ has half integral coordinates and for any triangle the sum of the coordinates associated to its edges is an integer. Thus we have inclusions 
\[
\mathcal{A}(\mathbb{Z}^t)\subseteq\mathcal{A}_L(S,\mathbb{Z})\subseteq\mathcal{A}\left(\frac{1}{2}\mathbb{Z}^t\right).
\]

\subsection{Canonical pairings}

\subsubsection{Multiplicative canonical pairings}

As explained in \cite{IHES} and \cite{dual}, the Teichm\"uller and lamination spaces defined above are dual in the sense that there exists a multiplicative canonical pairing of either lamination space with the Teichm\"uller space of the opposite type. More precisely, we have a pair of maps
\begin{align*}
\mathbb{I}&:\mathcal{X}_L(S,\mathbb{Z})\times\mathcal{A}^+_0(S)\rightarrow\mathbb{R}_{>0}, \\
\mathbb{I}&:\mathcal{A}_L^0(S,\mathbb{Z})\times\mathcal{X}^+(S) \rightarrow\mathbb{R}_{>0},
\end{align*}
which we denote by the same symbol $\mathbb{I}$. We usually think of these maps as operations which assign to a lamination $l$ the function on Teichm\"uller space given by $\mathbb{I}(l)(m)=\mathbb{I}(l,m)$.

Before giving the definition of the maps $\mathbb{I}$, let us describe two constructions involving Teichm\"uller spaces and laminations:

\begin{enumerate}
\item Let $m$ be a point of $\mathcal{A}_0^+$ and $l$ a curve on $S$ connecting punctures or boundary segments on $S$. If an endpoint of $l$ lies on a boundary segment, drag this endpoint in the counterclockwise direction until it hits a marked point. Take a geodesic homotopic to the resulting curve, and define $I(l,m)$ to be half the length between the horocycles at the endpoints of $l$.
\item Let $m$ be a point of $\mathcal{X}^+$ and $l$ a collection of curves with rational weights connecting boundary segments of $S$ so that for any boundary segment, the total weight of the curves hitting it vanishes. If an endpoint lies on a boundary segment, drag this endpoint in the counterclockwise direction until it hits a marked point. Finally, take geodesics homotopic to the resulting curves. Choose a horocycle around every marked point. Then $I(l,m)$ is the weighted sum of the signed half lengths of the curve between the horocycles. One can check that this definition is independent of the choice of horocycles.
\end{enumerate}

One now has the following definition of the canonical pairings.

\Needspace*{2\baselineskip}
\begin{definition}
\label{def:canonicalAX} \mbox{}
\begin{enumerate}
\item Let $l\in\mathcal{X}_L(S,\mathbb{Z})$ be a lamination consisting of a single closed curve of weight~$k$, and let $m\in\mathcal{A}_0^+$. Then $\mathbb{I}(l,m)$ is the absolute value of the trace of the $k$th power of the monodromy around $l$.
\item Let $m\in\mathcal{X}^+$ and let $l\in\mathcal{A}_L^0(S,\mathbb{Z})$ be a lamination consisting of a single closed curve of weight $k$ which is not retractable to a hole. Then $\mathbb{I}(l,m)$ is the absolute value of the trace of the $k$th power of the monodromy around $l$.
\item Let $m\in\mathcal{X}^+$ and let $l\in\mathcal{A}_L^0(S,\mathbb{Z})$ be a lamination consisting of a single closed curve of weight $k$ which is retractable to a hole with positive (respectively, negative) orientation. Then $\mathbb{I}(l,m)$ is the absolute value of the smallest (respectively, largest) eigenvalue of the $k$th power of the monodromy around $l$.
\item Let $l\in\mathcal{X}_L(S,\mathbb{Z})$ be a lamination consisting of a single curve of weight $k$ connecting punctures or points on the boundary of $S$, and let $m\in\mathcal{A}_0^+$. Assume that the orientation of any hole at one end of $l$ is induced by the orientation of the surface. Then $\mathbb{I}(l,m)=\exp I(l,m)$.
\item Let $l\in\mathcal{A}_L^0(S,\mathbb{Z})$ be a lamination consisting of a collection of open curves of integral weights on $S$, and let $m\in\mathcal{X}^+$. Then $\mathbb{I}(l,m)=\exp I(l,m)$.
\item Let $l_1$ and~$l_2$ be laminations such that no curve from $l_1$ intersects or is homotopic to a curve from $l_2$. Then $\mathbb{I}(l_1+l_2,m)=\mathbb{I}(l_1,m)\mathbb{I}(l_2,m)$.
\end{enumerate}
\end{definition}

This defines the canonical pairing in the special case where the orientations of the holes in an $\mathcal{X}$-lamination are induced from the orientation of the surface. If there is a hole $p$ in an $\mathcal{X}$-lamination $l$ whose orientation disagrees with the orientation of~$S$, then we can define $\mathbb{I}(l)$ by modifying slightly the above definition. To do this, we first note that there is a natural invariant associated to $p$. Indeed, fix an ideal triangulation $T$ of~$S$ and consider the triangles having $p$ as a vertex. We can label the edges that meet $p$ by $\eta_1,\dots,\eta_N$ in counterclockwise order and we can write $\zeta_i$ for the third edge of the triangle having edges $\eta_i$ and $\eta_{i+1}$, where we count indices modulo~$N$.
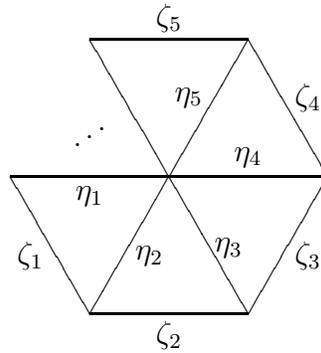
\[
\xy /l2.5pc/:
{\xypolygon6"A"{~:{(2,0):}~>{}}};
{"A1"\PATH~={**@{-}}'"A2"};
{"A2"\PATH~={**@{-}}'"A3"};
{"A3"\PATH~={**@{-}}'"A4"};
{"A4"\PATH~={**@{-}}'"A5"};
{"A5"\PATH~={**@{-}}'"A6"};
{(1,0)\PATH~={**@{-}}'"A1"};
{(1,0)\PATH~={**@{-}}'"A2"};
{(1,0)\PATH~={**@{-}}'"A3"};
{(1,0)\PATH~={**@{-}}'"A4"};
{(1,0)\PATH~={**@{-}}'"A5"};
{(1,0)\PATH~={**@{-}}'"A6"};
(2,-0.5)*{\Ddots};
(2,0.25)*{\eta_1};
(1.25,1)*{\eta_2};
(0.25,0.9)*{\eta_3};
(0,-0.25)*{\eta_4};
(0.75,-1)*{\eta_5};
(2.75,1)*{\zeta_1};
(1,2)*{\zeta_2};
(-0.75,1)*{\zeta_3};
(-0.75,-1)*{\zeta_4};
(1,-2)*{\zeta_5};
\endxy
\]
Given a point $m\in\mathcal{A}_0^+(S)$, we get a number $A_i$ associated to each edge $i$ of~$T$. We define 
\[
\alpha(p)=\sum_{i=1}^N\frac{A_{\zeta_i}}{A_{\eta_i}A_{\eta_{i+1}}}.
\]
One can check that $\alpha(p)$ is independent of the choice of ideal triangulation.

If $r$ is the number of holes in the surface, then there is a natural action of the group ${(\mathbb{Z}/2\mathbb{Z})}^r$ on the Teichm\"uller $\mathcal{A}$-space~\cite{IHES}. The generator of the $i$th factor of $\mathbb{Z}/2\mathbb{Z}$ acts by multiplying $A_j$ by $\alpha(p)$ whenever the edge $j$ meets $p$. There is also an action of ${(\mathbb{Z}/2\mathbb{Z})}^r$ on the space of $\mathcal{X}$-laminations where the generator of the $i$th factor of $\mathbb{Z}/2\mathbb{Z}$ acts by changing the orientation of the $i$th hole. We can therefore extend the canonical mapping $\mathbb{I}$ to all $\mathcal{X}$-laminations by requiring that it be equivariant with respect to these group actions.

\subsubsection{Laurent polynomials from laminations}
\label{subsec:LaurentPolynomialsFromLaminations}

In~\cite{IHES} and~\cite{dual}, Fock and Goncharov describe a method for computing the functions $\mathbb{I}(l)$ in terms of the coordinates on the Teichm\"uller $\mathcal{A}$- or $\mathcal{X}$-space. Here we review the method in the case of the $\mathcal{X}$-space and discuss some of its implications.

To begin, let $T$ be an ideal triangulation of the surface $S$. Draw a small edge transverse to every edge of the triangulation $T$ and connect the endpoints of these edges pairwise within each triangle of $T$. In this way, we associate a graph $\Gamma$ to the triangulation as illustrated below. (The original triangulation is indicated by the dotted lines in this picture.)
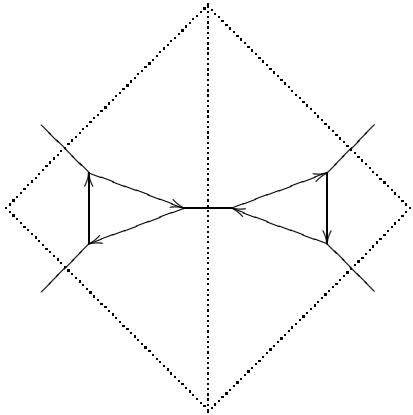
\[
\xy /l1.5pc/:
{\xypolygon4"A"{~:{(3,3):}~>{.}}},
{"A1"\PATH~={**@{.}}'"A3"},
(0.5,0)*{}="X1"; 
(1.5,0)*{}="Y1"; 
(-1.5,0.75)*{}="X2"; 
(3.5,0.75)*{}="Y2";
(-1.5,-0.75)*{}="X3"; 
(3.5,-0.75)*{}="Y3"; 
(-2.5,1.75)*{}="X4";  
(4.5,1.75)*{}="Y4"; 
(-2.5,-1.75)*{}="X5";  
(4.5,-1.75)*{}="Y5"; 
"X1";"Y1" **\dir{-}; 
"X2";"X4" **\dir{-}; 
"Y2";"Y4" **\dir{-}; 
"X3";"X5" **\dir{-}; 
"Y3";"Y5" **\dir{-}; 
{\ar"X2";"X1"},
{\ar"Y1";"Y2"},
{\ar"X1";"X3"},
{\ar"Y3";"Y1"},
{\ar"X3";"X2"},
{\ar"Y2";"Y3"},
\endxy
\]
This construction produces a small triangle inside every triangle of $T$, and we orient the edges of these small triangles in the clockwise direction.

Suppose we are given an oriented closed loop $\gamma$ based at a vertex of the graph. Then $\gamma$ is homotopic to a closed path in $\Gamma$. Let $e_1,\dots,e_n$ be the edges of this path in order, and let us associate a matrix $M(e_i)$ to each edge $e_i$ as follows. If $e_i$ intersects an edge~$k$ of~$T$, then the matrix that we assign to this edge is 
\[
M(e_i)=
\left( \begin{array}{cc}
0 & X_k^{1/2} \\
-X_k^{-1/2} & 0 \end{array} \right).
\]
This matrix satisfies $M(e_i)^2=-1$ in $SL(2,\mathbb{R})$ and hence $M(e_i)=M(e_i)^{-1}$ in $PSL(2,\mathbb{R})$. On the other hand, if $e_i$ does not intersect an edge of $T$ and its orientation agrees with the orientation coming from $\gamma$, then we define 
\[
M(e_i)=
\left( \begin{array}{cc}
1 & 1 \\
-1 & 0 \end{array} \right).
\]
Finally, if $e_i$ does not intersect an edge of $T$ and its orientation disagrees with the orientation of $\gamma$, then $M(e_i)$ is the inverse of this last matrix. We can multiply the matrices $M(e_i)$ defined in this way to get the monodromy 
\[
\rho(\gamma)=M(e_n)\dots M(e_1)\in PSL(2,\mathbb{R}).
\]
One can check that this element $\rho(\gamma)$ is well defined.

There is a similar construction that allows us to express the monodromy of a closed loop in terms of $A$-coordinates. In \cite{IHES}, Fock and Goncharov use these constructions to compute $\mathbb{I}(l)$ when $l$ is an $\mathcal{A}$- or $\mathcal{X}$-lamination and to show that this function can be written as a Laurent polynomial with positive integral coefficients in the variables $X_j^{1/2}$ or $A_j$ ($j\in J$), respectively. Moreover, they show that when $l$ has integral coordinates (that is, when $l$ comes from $\mathcal{A}(\mathbb{Z}^t)$ or $\mathcal{X}(\mathbb{Z}^t)$), the function $\mathbb{I}(l)$ can in fact be written as a Laurent polynomial with positive integral coefficients in the variables $X_j$ or $A_j$ ($j\in J$). In Section~\ref{sec:TheCanonicalPairing}, we will revisit the above construction and use it to prove a similar result for $\mathcal{D}$-laminations.

\subsubsection{The intersection pairing}

In addition to the multiplicative canonical pairings that we defined above, we have the following canonical map, which should be viewed as a degeneration of these pairings.

\begin{definition}
Let $S$ be a punctured surface and choose $l\in\mathcal{A}_L(S,\mathbb{Z})$ and $m\in\mathcal{X}_L(S,\mathbb{Z})$. Assume that $m$ provides the negative orientation for each hole. Then we define $\mathcal{I}(l,m)$ to be half the minimal number of intersections between~$l$ and~$m$. Here we take into account the weights of the curves so that if a curve of weight $k_1$ intersects a curve of weight $k_2$, then this intersection contributes the term $k_1k_2/2$ to $\mathcal{I}(l,m)$.
\end{definition}

This defines $\mathcal{I}(l,m)$ in the special case where $m$ provides the negative orientation for each hole. Note that there are natural actions of the group $(\mathbb{Z}/2\mathbb{Z})^r$ on $\mathcal{A}_L(S,\mathbb{Z})$ and $\mathcal{X}_L(S,\mathbb{Z})$ where $r$ is the number of punctures in $S$. The generator of the $i$th factor of $\mathbb{Z}/2\mathbb{Z}$ acts on $\mathcal{A}_L(S,\mathbb{Z})$ by changing the sign of the weight of curves surrounding the $i$th hole. It acts on $\mathcal{X}_L(S,\mathbb{Z})$ by changing the orientation of the $i$th hole. We extend $\mathcal{I}$ to a map
\[
\mathcal{I}:\mathcal{A}_L(S,\mathbb{Z})\times\mathcal{X}_L(S,\mathbb{Z})\rightarrow\mathbb{Q}
\]
called the \emph{intersection pairing} by requiring that it be equivariant with respect to these group actions.

To understand the relationship between the intersection pairing and the multiplicative canonical pairings defined above, we recall the notion of the tropicalization of a subtraction-free rational function~\cite{dual}. Let $F(u_1,\dots,u_n)$ be such a function. We define its \emph{tropicalization} $F^t(u_1,\dots,u_n)$ by the formula 
\[
F^t(u_1,\dots,u_n)=\lim_{C\rightarrow\infty}\frac{\log F(e^{Cu_1},\dots,e^{Cu_n})}{C}.
\]
It is easy to see that 
\[
\lim_{C\rightarrow\infty}\frac{\log(e^{Cv_1}+\dots+e^{Cv_n})}{C}=\max(v_1,\dots,v_n),
\]
so tropicalization takes the operations~$+$ and~$\cdot$ to~$\max$ and~$+$, respectively.

For any laminations $l\in\mathcal{A}_L^0(S,\mathbb{Z})$ and $m\in\mathcal{X}_L(S,\mathbb{Z})$, one has 
\[
\mathcal{I}(l,m)=(\mathbb{I}(l))^t(m)
\]
and 
\[
\mathcal{I}(l,m)=(\mathbb{I}(m))^t(l)
\]
so that the intersection pairing is the tropicalization of the multiplicative canonical pairings. In Section~\ref{sec:TheCanonicalPairing}, we will give a similar characterization of the tropicalization of the canonical mapping $\mathbb{I}_{\mathcal{D}}$ on $\mathcal{D}$-laminations.

\section{The symplectic double}
\label{sec:TheSymplecticDouble}

\subsection{Teichm\"uller $\mathcal{D}$-space}

We now describe the moduli space from \cite{double}. Let $\Sigma$ be an oriented punctured surface.

\begin{definition}
A \emph{simple lamination} on $\Sigma$ is a finite collection $\gamma=\{\gamma_i\}$ of simple nontrivial disjoint nonhomotopic closed curves on $\Sigma$ which do not retract to the punctures, considered up to homotopy.
\end{definition}

Let $\Sigma$ be an oriented surface equipped with a simple lamination $\gamma=\{\gamma_i\}$. For any subset $\{\gamma_1,\dots,\gamma_k\}$ of the set of loops $\gamma_i$, we let $\Sigma_{p_1,\dots,p_k}$ denote the singular surface obtained by pinching these loops to get the nodes $p_1,\dots,p_k$. This singular surface is equipped with a simple lamination $\gamma_{p_1,\dots,p_k}$ given by the image of $\gamma-\{\gamma_1,\dots,\gamma_k\}$. Note that if we cut the surface $\Sigma_{p_1,\dots,p_k}$ at the nodes $p_1,\dots,p_k$, we get a punctured surface $\Sigma_{p_1,\dots,p_k}'$. When we talk about a point in the Teichm\"uller space of $\Sigma_{p_1,\dots,p_k}$, we really mean a point in the Teichm\"uller space of the surface $\Sigma_{p_1,\dots,p_k}'$. Similarly, when we talk about a horocycle at a node of $\Sigma_{p_1,\dots,p_k}$, we really mean a horocycle at one of the corresponding punctures in $\Sigma_{p_1,\dots,p_k}'$.

\begin{definition}[\cite{double}]
The set $\mathcal{X}_{\Sigma;\gamma;p_1,\dots,p_k}^+$ parametrizes points $\rho\in\mathcal{T}(\Sigma_{p_1,\dots,p_k})$ of the Teichm\"uller space of $\Sigma_{p_1,\dots,p_k}$, together with the following data:
\begin{enumerate}
\item An orientation for each loop of the simple lamination $\gamma_{p_1,\dots,p_k}$.

\item For every node $p_i$, a pair of horocycles $(c_{-,i},c_{+,i})$ centered at $p_i$, located on opposite sides of the node, and defined up to simultaneous shift by any real number.
\end{enumerate}
We write $\mathcal{X}_{\Sigma;\gamma}^+$ for the union of these sets $\mathcal{X}_{\Sigma;\gamma;p_1,\dots,p_k}^+$.
\end{definition}

Now let $S$ be a decorated surface. If $S$ has punctures, then we can compactify~$S$ by gluing a boundary circle to each puncture. This produces a compact oriented surface with boundary which we will also denote by $S$. Let $S^\circ$ be the same surface with the opposite orientation.

The \emph{double} $S_\mathcal{D}$ of $S$ is the punctured surface obtained by gluing $S$ and~$S^\circ$ along corresponding boundary components and deleting the image of each marked point in the resulting surface. The surface $\Sigma=S_\mathcal{D}$ carries a natural simple lamination~$\gamma$ given by the image of the boundary loops of $S$.

\begin{definition}
The \emph{Teichm\"uller $\mathcal{D}$-space} $\mathcal{D}^+(S)$ is the space $\mathcal{X}_{\Sigma;\gamma}^+$ with $\Sigma=S_\mathcal{D}$.
\end{definition}

When there is no possibility of confusion, we will denote the Teichm\"uller $\mathcal{D}$-space by $\mathcal{D}^+$. Note that so far this space is a union of strata with no specified topology. Below we will define coordinates on $\mathcal{D}^+$ giving it a natural topology.

\subsection{Construction of coordinates}

\subsubsection{Preliminaries}

Let $S$ be a surface as above and $S_\mathcal{D}$ its double. The next two results will be used to construct coordinates on $\mathcal{D}^+(S)$.

\begin{lemma}
If $\rho\in\mathcal{T}(S_\mathcal{D})$ is a point in the Teichm\"uller space of $S_\mathcal{D}$, then there exists a hyperbolic structure representing $\rho$ such that $\partial S\subseteq S_\mathcal{D}$ is geodesic.
\end{lemma}

\begin{proof}
A point $\rho$ of the Teichm\"uller space can be viewed as a marked hyperbolic surface, that is, a hyperbolic surface $X$ together with a diffeomorphism $\phi:S_\mathcal{D}\rightarrow X$. The image of $\partial S\subseteq S_\mathcal{D}$ under $\phi$ may not be geodesic in $X$, but we can deform~$\phi$ to a homotopic map $\phi'$ such that $\phi'(\partial S)$ is geodesic. This is the same as saying that the diagram 
\[
\xymatrix{ 
& S_\mathcal{D} \ar[rd]^-{\phi'} \ar[ld]_{\phi} \\
X \ar[rr]_-{1_X} & & X 
}
\]
commutes up to homotopy. Thus $(X,\phi)$ and $(X,\phi')$ represent the same point of Teichm\"uller space. Pulling back the hyperbolic structure of $X$ along $\phi'$, we obtain a hyperbolic structure on $S_\mathcal{D}$ representing the point $\rho$.
\end{proof}

Let $\rho\in\mathcal{T}(S_\mathcal{D})$. By the above lemma, we can view $\rho$ as a hyperbolic structure on~$S_\mathcal{D}$ such that $\partial S\subseteq S_\mathcal{D}$ is geodesic. By cutting along $\partial S$, we recover the surfaces~$S$ and~$S^\circ$ equipped with hyperbolic structures with geodesic boundary. Then the universal cover $\tilde{S}$ of $S$ can be identified with a subset of $\mathbb{H}$ obtained by removing countably many geodesic half disks. Similarly, the universal cover $\tilde{S^\circ}$ of~$S^\circ$ can be identified with a subset of $\mathbb{H}$.

\begin{lemma}
\label{lem:correspondence}
Let $\iota$ be the natural map $S\rightarrow S^\circ$. Then there exists a map $\tilde{\iota}:\tilde{S}\rightarrow\tilde{S^\circ}$, unique up to the action of $\pi_1(S)$ by deck transformations, such that the diagram 
\[
\xymatrix{ 
\tilde{S} \ar[r]^{\tilde{\iota}} \ar[d] & \tilde{S^\circ} \ar[d] \\
S \ar_{\iota}[r] & S^\circ
}
\]
commutes. This map $\tilde{\iota}$ restricts to an isometry on the geodesic boundary of $\tilde{S}$.
\end{lemma}

\begin{proof}
Since $\tilde{S}$ is simply connected, the composition $\tilde{S}\rightarrow S\stackrel{\iota}{\rightarrow} S^\circ$ induces the zero map on fundamental groups. Hence, by the lifting criterion, there exists a map $\tilde{\iota}:\tilde{S}\rightarrow\tilde{S^\circ}$ making the above diagram commutative. If $\tilde{\iota}':\tilde{S}\rightarrow\tilde{S^\circ}$ is any other map with this property, then $\tilde{\iota}$ and $\tilde{\iota}'$ are both lifts of the composition $\tilde{S}\rightarrow S\stackrel{\iota}{\rightarrow} S^\circ$. It follows that they coincide up to a deck transformation.

We know that the map $\tilde{\iota}$ restricts to an isometry on the geodesic boundary of~$\tilde{S}$ because the other three maps of the diagram restrict to isometries on boundary geodesics.
\end{proof}

\subsubsection{Definition of $B_i$ and $X_i$}

To construct coordinates on $\mathcal{D}^+(S)$, fix a point $m\in\mathcal{D}^+(S)$ and let $i\in J$ be an edge of the ideal triangulation $T$. The point $m$ determines a point of the Teichm\"uller space of $\Sigma_{p_1,\dots,p_k}$ where $\Sigma=S_{\mathcal{D}}$, and we can represent this by a hyperbolic structure such that any curve in the image of~$\partial S$ is geodesic. By cutting along the image of~$\partial S$, we recover the surfaces $S$ and $S^\circ$ equipped with hyperbolic structures with geodesic boundary.

Suppose the edge $i$ corresponds to an arc connecting two holes in $S$. Choose a pair of geodesics $g_1$ and~$g_2$ in the upper half plane $\mathbb{H}$ which project to these boundary components. Deform the arc connecting the holes by winding its endpoints around the holes infinitely many times in the direction prescribed by the orientation. This corresponds to deforming the preimage $\tilde{i}$ in $\mathbb{H}$ so that its endpoints coincide with endpoints of~$g_1$ and~$g_2$. If we let $i^\circ$ be the image of the arc $i$ under the tautological map $S\rightarrow S^\circ$, then we can apply the same procedure to $i^\circ$ to get an arc $\tilde{i}^\circ$ in $\mathbb{H}$.

Choose horocycles $c_1$ and~$c_2$ around the endpoints of $\tilde{i}$. The horocycle $c_k$ intersects $g_k$ in a unique point. By Lemma~\ref{lem:correspondence}, there is a corresponding point on $g_k^\circ$ and thus a corresponding horocycle $c_k^\circ$. The map constructed in Lemma~\ref{lem:correspondence} restricts to an isometry on $g_k$, so if we shift $c_k$ by some amount, then the corresponding horocycle $c_k^\circ$ will shift by the same amount.
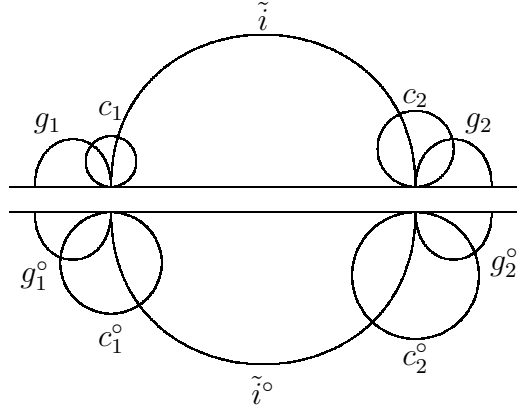
\[
\xy 0;/r.40pc/: 
(-18,-6)*{}="0"; 
(-12,-6)*{}="1"; 
(12,-6)*{}="2"; 
(18,-6)*{}="3"; 
"0";"1" **\crv{(-18,-1) & (-12,-1)}; 
"1";"2" **\crv{(-12,10) & (12,10)}; 
"2";"3" **\crv{(12,-1) & (18,-1)}; 
(-20,-6)*{}="X"; 
(20,-6)*{}="Y"; 
"X";"Y" **\dir{-}; 
(-12,-4)*\xycircle(2,2){-};
(12,-3)*\xycircle(3,3){-};
(0,7.5)*{\tilde{i}}; 
(-12,0)*{c_1}; 
(12,1)*{c_2}; 
(-17,-1)*{g_1}; 
(17,-1)*{g_2}; 
(-20,-8)*{}="Z"; 
(20,-8)*{}="W"; 
"Z";"W" **\dir{-}; 
(-18,-8)*{}="0"; 
(-12,-8)*{}="1"; 
(12,-8)*{}="2"; 
(18,-8)*{}="3"; 
"0";"1" **\crv{(-18,-13) & (-12,-13)}; 
"1";"2" **\crv{(-12,-24) & (12,-24)}; 
"2";"3" **\crv{(12,-13) & (18,-13)}; 
(12,-13)*\xycircle(5,5){-};
(-12,-12)*\xycircle(4,4){-};
(0,-22)*{\tilde{i}^\circ};
(-12,-18)*{c_1^\circ}; 
(12,-19.5)*{c_2^\circ}; 
(-18,-13)*{g_1^\circ}; 
(19,-12)*{g_2^\circ}; 
\endxy
\]

Denote by $A_i$ the exponentiated signed half distance between~$c_1$ and~$c_2$, and denote by $A_i^\circ$ the exponentiated signed half distance between~$c_1^\circ$ and~$c_2^\circ$. We can then define a number associated to the edge $i$ by 
\[
B_i=\frac{A_i^\circ}{A_i}.
\]
This defines $B_i$ when $i$ corresponds to an arc connecting two holes. If one or both of the endpoints of $i$ are punctures, then the data of the point $m\in\mathcal{D}^+(S)$ provides a horocycle~$c_k$ at any endpoint of $i$ which is a puncture, as well as a horocycle $c_k^\circ$ at the corresponding endpoint of $i^\circ$. Thus we can associate numbers $A_i$ and $A_i^\circ$ to these arc as before, and we can define $B_i$ by the above formula. One can show that the $|J|$ numbers obtained in this way are independent of all choices made in the construction.

In addition to the $B_i$, there are $|J|$ numbers $X_i$ associated to a point in the space $\mathcal{D}^+(S)$. Given a point of $\mathcal{D}^+(S)$, these are simply defined as the $X$-coordinates of the point of the Teichm\"uller $\mathcal{X}$-space of the surface $S$ obtained by cutting $S_\mathcal{D}$ along the image of $\partial S$.

\subsubsection{Reconstruction}

We will now show that the $B_i$ and $X_i$ are indeed coordinates on $\mathcal{D}^+(S)$.

\begin{proposition}
The numbers $B_i$ and $X_i$ provide a bijection 
\[
\mathcal{D}^+(S)\rightarrow\mathbb{R}_{>0}^{2|J|}.
\]
\end{proposition}

\begin{proof}
Call this map $\phi$. We will construct a map $\psi:\mathbb{R}_{>0}^{2|J|}\rightarrow\mathcal{D}^+(S)$ and prove that $\phi$ and $\psi$ are inverses.

\Needspace*{2\baselineskip}
\begin{step}[1]
Definition of $\psi:\mathbb{R}_{>0}^{2|J|}\rightarrow\mathcal{D}^+(S)$.
\end{step}

Suppose we are given positive numbers $B_j$ and $X_j$ for every edge $j\in J$. By the reconstruction procedure for the Teichm\"uller $\mathcal{X}$-space, we can glue together ideal triangles in $\mathbb{H}$, using the $X_i$ as gluing parameters, to get a tiling of a region $\tilde{S}\subseteq\mathbb{H}$ with geodesic boundary.

Choose a horocycle around each vertex of this triangulation. (We do not require these horocycles to be covariant under the action of the deck transformation group.) Then for every edge~$i$ of this triangulation, there is a number $A_i$ defined as the exponentiated signed half length between the horocycles at the ends of~$i$. If $B_i$ is the $B$-coordinate corresponding to the lifted edge~$i$, define 
\[
A_i^\circ=B_iA_i.
\]
We will use these numbers $A_i^\circ$ to construct another region $\tilde{S^\circ}$ with geodesic boundary in a copy of~$\mathbb{H}$.

To construct $\tilde{S^\circ}$, first choose an edge $i$ in the triangulation of $\tilde{S}$. Let $i^\circ$ be any geodesic in the second copy of $\mathbb{H}$. There are horocycles around the endpoints of~$i$ with corresponding exponentiated half length $A_i$, and we can choose horocycles around the endpoints of $i^\circ$ so that the resulting exponentiated half length equals~$A_i^\circ$. Consider one of the triangles adjacent to~$i$. There are numbers~$A_j$ and~$A_k$ corresponding to its other sides. By \cite{P}, Chapter~1, Lemma~4.14, there is a unique point in the second copy of $\partial\mathbb{H}$ and a horocycle about this point so that the exponentiated half lengths of the decorated arcs incident to this point are $A_j^\circ$ and $A_k^\circ$, and the cyclic order of the vertices of the resulting triangle agrees with the original cyclic order of the vertices.

Repeating this process for every edge in the triangulation of $\tilde{S}$, we obtain a region $\tilde{S^\circ}\subseteq\mathbb{H}$ with geodesic boundary. To construct this region, we had to choose horocycles around the endpoints of the triangulation of $\tilde{S}$, but the construction is independent of these choices. We also had to choose edges $i$ and $i^\circ$ and horocycles around the endpoints of $i^\circ$. If we had made a different choice we would get a different region obtained from $\tilde{S^\circ}$ by a transformation in $PSL(2,\mathbb{R})$. Thus $\tilde{S^\circ}$ is well defined up to isometry.

Now the region $\tilde{S}$ is obtained from $\mathbb{H}$ by removing infinitely many geodesic half disks. We can extend $\tilde{S}$ to a larger region by gluing a copy of $\tilde{S^\circ}$ in each of these half disks in such a way that the horocycles around identified vertices coincide. The resulting region is again obtained from $\mathbb{H}$ by removing infinitely many geodesic half disks, and we can enlarge it by gluing a copy of $\tilde{S}$ in each of these half disks. Continuing this process ad infinitum, we partition the entire hyperbolic plane into ideal triangles. The illustration below shows an example of how the spaces~$\tilde{S}$ and~$\tilde{S}^\circ$ may be glued together in the disk model of the hyperbolic plane.
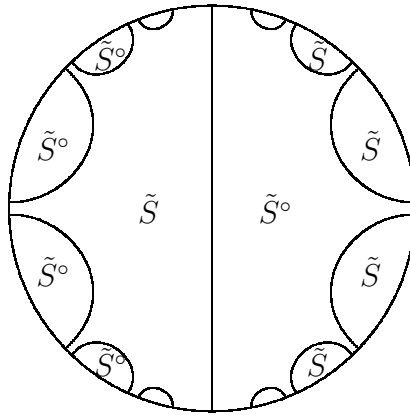
\[
\xy 0;/r.40pc/: 
(0,0)*\xycircle(16,16){-};
(0,16)*{}="X"; 
(0,-16)*{}="Y"; 
(-3.12,15.69)*{}="A1"; 
(-5.8,14.9)*{}="A2"; 
(-6.6,14.5)*{}="A3"; 
(-11,11.5)*{}="A4"; 
(-11.5,11)*{}="A5"; 
(-15.9,0.5)*{}="A6"; 
(-15.9,-0.5)*{}="A7"; 
(-11.5,-11)*{}="A8"; 
(-11,-11.5)*{}="A9"; 
(-6.6,-14.5)*{}="A10"; 
(-5.8,-14.9)*{}="A11"; 
(-3.12,-15.69)*{}="A12"; 
(3.12,15.69)*{}="B1"; 
(5.8,14.9)*{}="B2"; 
(6.6,14.5)*{}="B3"; 
(11,11.5)*{}="B4"; 
(11.5,11)*{}="B5"; 
(15.9,0.5)*{}="B6"; 
(15.9,-0.5)*{}="B7"; 
(11.5,-11)*{}="B8"; 
(11,-11.5)*{}="B9"; 
(6.6,-14.5)*{}="B10"; 
(5.8,-14.9)*{}="B11"; 
(3.12,-15.69)*{}="B12"; 
"X";"Y" **\dir{-}; 
"A1";"A2" **\crv{(-2.9,14.4) & (-5.1,13.4)}; 
"A3";"A4" **\crv{(-5,12) & (-9,9)}; 
"A5";"A6" **\crv{(-6.6,6.5) & (-11,0.4)}; 
"A7";"A8" **\crv{(-11,-0.4) & (-6.6,-6.5)}; 
"A9";"A10" **\crv{(-9,-9) & (-5,-12)}; 
"A11";"A12" **\crv{(-5.1,-13.4) & (-2.9,-14.4)}; 
"B1";"B2" **\crv{(2.9,14.4) & (5.1,13.4)}; 
"B3";"B4" **\crv{(5,12) & (9,9)}; 
"B5";"B6" **\crv{(6.6,6.5) & (11,0.4)}; 
"B7";"B8" **\crv{(11,-0.4) & (6.6,-6.5)}; 
"B9";"B10" **\crv{(9,-9) & (5,-12)}; 
"B11";"B12" **\crv{(5.1,-13.4) & (2.9,-14.4)}; 
(-5,0)*{\tilde{S}}; 
(5,0)*{\tilde{S}^\circ}; 
(-12.5,5)*{\tilde{S}^\circ}; 
(12.5,5)*{\tilde{S}}; 
(-12.5,-5)*{\tilde{S}^\circ}; 
(12.5,-5)*{\tilde{S}}; 
(-8,12.1)*{\tilde{S}^\circ}; 
(8.3,12.1)*{\tilde{S}}; 
(-8,-12.1)*{\tilde{S}^\circ}; 
(8.3,-12.1)*{\tilde{S}}; 
\endxy
\]

Any element of $\pi_1(S_\mathcal{D})$ corresponds to a deck transformation and is thus represented as a unique element of $PSL(2,\mathbb{R})$. In this way we recover a discrete and faithful representation $\rho:\pi_1(S_\mathcal{D})\rightarrow PSL(2,\mathbb{R})$ representing a point in $\mathcal{T}(S_\mathcal{D})$.

\Needspace*{2\baselineskip}
\begin{step}[2]
Proof of $\phi\circ\psi=1_{\mathbb{R}_{>0}^{2|J|}}$.
\end{step}

Let $B_j$ and $X_j$ ($j\in J$) be given. By the reconstruction procedure described above, we construct triangulated regions $\tilde{S}\subseteq\mathbb{H}$ and $\tilde{S}^\circ\subseteq\mathbb{H}$ which we then glue together to get a triangulation of the hyperbolic plane. From this triangulation, we get a point $m\in\mathcal{D}^+(S)$. We want to show that the coordinates of $m$ are the numbers $B_j$ and $X_j$.

This point $m$ determines a hyperbolic metric on $S_{\mathcal{D}}$, and the universal cover of the resulting hyperbolic surface is isometric to the triangulated surface $\tilde{S}_{\mathcal{D}}$ that we got by gluing together copies of $\tilde{S}$ and $\tilde{S}^\circ$. Consider a copy of $\tilde{S}$ in $\tilde{S}_{\mathcal{D}}$. To find the coordinates of $m$, we choose horocycles around the vertices of the triangulation of $\tilde{S}$, and we can take these to be exactly the horocycles used in the construction of $m$ from the~$B_j$ and~$X_j$. Now consider a copy of $\tilde{S}^\circ$ adjacent to $\tilde{S}$ in $\tilde{S}_{\mathcal{D}}$. Choose a basepoint $x$ on the geodesic separating these regions. If $g$ is any boundary geodesic of $\tilde{S}$, then there is a horocycle $h$ around one of the endpoints of $g$, and this horocycle intersects $g$ in a unique point $p$. Draw a curve $\tilde{\alpha}$ in~$\tilde{S}$ from the point~$p$ to the point~$x$.
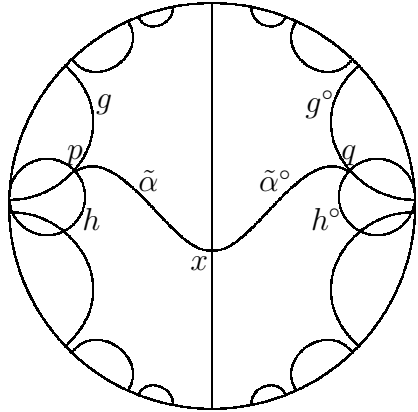
\[
\xy 0;/r.40pc/: 
(0,0)*\xycircle(16,16){-};
(-13,0.7)*\xycircle(3,3){-};
(13,0.7)*\xycircle(3,3){-};
(-10.7,2.75)*{}="p"; 
(0,-3.5)*{}="x"; 
(10.7,2.75)*{}="q"; 
(0,16)*{}="X"; 
(0,-16)*{}="Y"; 
(-3.12,15.69)*{}="A1"; 
(-5.8,14.9)*{}="A2"; 
(-6.6,14.5)*{}="A3"; 
(-11,11.5)*{}="A4"; 
(-11.5,11)*{}="A5"; 
(-15.9,0.5)*{}="A6"; 
(-15.9,-0.5)*{}="A7"; 
(-11.5,-11)*{}="A8"; 
(-11,-11.5)*{}="A9"; 
(-6.6,-14.5)*{}="A10"; 
(-5.8,-14.9)*{}="A11"; 
(-3.12,-15.69)*{}="A12"; 
(3.12,15.69)*{}="B1"; 
(5.8,14.9)*{}="B2"; 
(6.6,14.5)*{}="B3"; 
(11,11.5)*{}="B4"; 
(11.5,11)*{}="B5"; 
(15.9,0.5)*{}="B6"; 
(15.9,-0.5)*{}="B7"; 
(11.5,-11)*{}="B8"; 
(11,-11.5)*{}="B9"; 
(6.6,-14.5)*{}="B10"; 
(5.8,-14.9)*{}="B11"; 
(3.12,-15.69)*{}="B12"; 
"p";"x" **\crv{(-7,5) & (-3,-4)}; 
"x";"q" **\crv{(3,-4) & (7,5)}; 
"X";"Y" **\dir{-}; 
"A1";"A2" **\crv{(-2.9,14.4) & (-5.1,13.4)}; 
"A3";"A4" **\crv{(-5,12) & (-9,9)}; 
"A5";"A6" **\crv{(-6.6,6.5) & (-11,0.4)}; 
"A7";"A8" **\crv{(-11,-0.4) & (-6.6,-6.5)}; 
"A9";"A10" **\crv{(-9,-9) & (-5,-12)}; 
"A11";"A12" **\crv{(-5.1,-13.4) & (-2.9,-14.4)}; 
"B1";"B2" **\crv{(2.9,14.4) & (5.1,13.4)}; 
"B3";"B4" **\crv{(5,12) & (9,9)}; 
"B5";"B6" **\crv{(6.6,6.5) & (11,0.4)}; 
"B7";"B8" **\crv{(11,-0.4) & (6.6,-6.5)}; 
"B9";"B10" **\crv{(9,-9) & (5,-12)}; 
"B11";"B12" **\crv{(5.1,-13.4) & (2.9,-14.4)}; 
(-1,-4.5)*{x}; 
(-5,2)*{\tilde{\alpha}}; 
(5,2)*{\tilde{\alpha}^\circ}; 
(-10.8,4)*{p}; 
(10.8,4)*{q}; 
(-9.5,-1)*{h}; 
(9,-1)*{h^\circ}; 
(-8.5,8)*{g}; 
(8.5,8)*{g^\circ}; 
\endxy
\]
This curve $\tilde{\alpha}$ projects to a curve~$\alpha$ on the surface $S$, and we can apply the map~$\iota$ to get a curve $\alpha^\circ$ on $S^\circ$. Finally, we lift $\alpha^\circ$  to a curve $\tilde{\alpha}^\circ$ in~$S^\circ$ that starts at~$x$. This curve $\tilde{\alpha}^\circ$ ends at some point $q=\tilde{\iota}(p)$, and there is a horocycle $h^\circ$ that intersects~$g^\circ$ at~$q$. In this way, we draw horocycles around all of the vertices of $\tilde{S}^\circ$ and can define the coordinates of $m$.

On the other hand, consider the monodromy along the curve in $S_{\mathcal{D}}$ obtained by concatenating $\alpha$ and $\alpha^\circ$. It maps the region~$\tilde{S}$ isometrically into the region bounded by~$g^\circ$. The horocycle used in the construction of~$m$ from the~$B_j$ and~$X_j$ is obtained by applying this transformation to $h$. Thus the horocycles used to construct $m$ agree with the ones obtained using the map $\tilde{\iota}$.

It follows that the $B$-coordinates of $m$ are simply the numbers $B_j$ that we started with. It is easy to see that the $X$-coordinates of $m$ are the numbers $X_j$. This completes Step~2 of the proof.

\Needspace*{2\baselineskip}
\begin{step}[3]
Proof of $\psi\circ\phi=1_{\mathcal{D}^+}$.
\end{step}

Let $m\in\mathcal{D}^+(S)$ be given. To calculate the coordinates of~$m$, let us pass to the universal cover~$\tilde{S}_{\mathcal{D}}$ of~$S_{\mathcal{D}}$ and choose a copy of $\tilde{S}$ in~$\tilde{S}_{\mathcal{D}}$. Choose horocycles around the vertices of the triangulation of this copy of $\tilde{S}$. Using the construction described above involving lifts of the map $\iota:S\rightarrow S^\circ$, we can get horocycles around all vertices of the triangulation of $S_{\mathcal{D}}$. We can use these horocycles to calculate the numbers $A_j$ and~$A_j^\circ$, and thus the coordinates $B_j$ and~$X_j$.

We can now use the $X_i$ to reconstruct the region~$\tilde{S}$. We must choose horocycles around the vertices of the triangulation of~$\tilde{S}$, and we can assume that these are exactly the ones used above to define the~$B_j$ and~$X_j$. Then we can use the $B_j$ to reconstruct an adjacent region $\tilde{S}^\circ$ with exactly the horocycles used above. By gluing together copies of $\tilde{S}$ and $\tilde{S}^\circ$, we recover $\tilde{S}_{\mathcal{D}}$ together with all of the horocycles used to compute the coordinates. Quotienting this universal cover by the action of $\pi_1(S_{\mathcal{D}})$, we recover the point $m\in\mathcal{D}^+(S)$. This completes Step~3 of the proof.
\end{proof}

\subsubsection{Transformation rule}

We have defined the $B_i$ and $X_i$ coordinates in terms of a fixed ideal triangulation of $S$. The next result says how these coordinates vary when we change the triangulation.

\begin{proposition}
A regular flip at an edge $k$ of the triangulation changes the coordinates~$X_i$ and~$B_i$ to new coordinates~$X_i'$ and~$B_i'$ given by the formulas 
\begin{align*}
X_i' &=
\begin{cases}
X_k^{-1} & \mbox{if } i=k \\
X_i{(1+X_k^{-\sgn(\varepsilon_{ik})})}^{-\varepsilon_{ik}} & \mbox{if } i\neq k,
\end{cases} \\
B_i' &=
\begin{cases}
\frac{X_k\prod_{j|\varepsilon_{kj>0}}B_j^{\varepsilon_{kj}} + \prod_{j|\varepsilon_{kj<0}}B_j^{-\varepsilon_{kj}}}{(1+X_k)B_k} & \mbox{if } i=k \\
B_i & \mbox{if } i\neq k.
\end{cases}
\end{align*}
\end{proposition}

\begin{proof}
By definition, the coordinates $X_i$ on $\mathcal{D}^+(S)$ coincide with the coordinates of the corresponding point of the Teichm\"uller $\mathcal{X}$-space of $S$. Therefore the transformation rule for the $X_i$ is just the usual transformation rule for the coordinates on the $\mathcal{X}$-space.

To prove the second rule, deform the arcs corresponding to the edges of the ideal triangulation by winding their endpoints around the curves of the simple lamination infinitely many times. Lift these deformed curves to the universal cover of $S$, and consider the ideal quadrilateral formed by the two triangles adjacent to the geodesic arc $\tilde{k}$ corresponding to~$k$.

Number the vertices of this quadrilateral in counterclockwise order so that the edge $\tilde{k}$ joins vertices~1 and~3. Choose a horocycle around each vertex of this ideal quadrilateral and let $A_{ij}$ denote the exponentiated signed half length between the horocycles at $i$ and $j$. There is a corresponding ideal quadrilateral in the universal cover of $S^\circ$ and a number $A_{ij}^\circ$ corresponding to the edge connecting vertices~$i$ and~$j$ of this ideal quadrilateral.

If we flip at the edge $k$, then the ideal quadrilateral in the universal cover of~$S$ is replaced by the same ideal quadrilateral triangulated by the arc from~2 to~4, and the number associated to this new arc is 
\[
A_{24}=\frac{A_{12}A_{34}+A_{14}A_{23}}{A_{13}}.
\]
Similarly, the number $A_{13}^\circ$ transforms to 
\[
A_{24}^\circ=\frac{A_{12}^\circ A_{34}^\circ+A_{14}^\circ A_{23}^\circ}{A_{13}^\circ}.
\]
Applying these rules to the quotient $B_{13}=A_{13}^\circ/A_{13}$, we obtain 
\begin{align*}
B_{24} &= \frac{A_{12}^\circ A_{34}^\circ+A_{14}^\circ A_{23}^\circ}{A_{12}A_{34}+A_{14}A_{23}}\frac{A_{13}}{A_{13}^\circ} \\
&= \frac{\frac{A_{12}^\circ A_{34}^\circ}{A_{14}A_{23}}+\frac{A_{14}^\circ A_{23}^\circ}{A_{14}A_{23}}}{1+\frac{A_{12}A_{34}}{A_{14}A_{23}}}\frac{A_{13}}{A_{13}^\circ} \\
&= \frac{X_{13}B_{12}B_{34}+B_{14}B_{23}}{(1+X_{13})B_{13}}
\end{align*}
where we have used the relation $X_{13}=\frac{A_{12}A_{34}}{A_{14}A_{23}}$. This proves the transformation rule for the~$B_i$.
\end{proof}

\subsection{Definition of the symplectic double}

We will now use the transformation rules stated above to define the symplectic double. For any seed $\mathbf{i}=(I,J,\varepsilon_{ij},d_i)$ with $I=J$, we define a split algebraic torus $\mathcal{D}_\mathbf{i} = (\mathbb{G}_m)^{2|J|}$ with natural coordinates $\{B_i, X_i\}_{i\in J}$. A seed mutation induces a map on tori defined by 
\[
\mu_k^*X_i'=
\begin{cases}
X_k^{-1} & \mbox{if } i=k \\
X_i{(1+X_k^{-\sgn(\varepsilon_{ik})})}^{-\varepsilon_{ik}} & \mbox{if } i\neq k
\end{cases}
\]
and 
\[
\mu_k^*B_i' =
\begin{cases}
\frac{X_k\prod_{j|\varepsilon_{kj>0}}B_j^{\varepsilon_{kj}} + \prod_{j|\varepsilon_{kj<0}}B_j^{-\varepsilon_{kj}}}{(1+X_k)B_k} & \mbox{if } i=k \\
B_i & \mbox{if } i\neq k
\end{cases}
\]
where $X_i'$ and $B_i'$ are the coordinates on $\mathcal{D}_{\mathbf{i}'}$.

\begin{definition}
The \emph{symplectic double} $\mathcal{D}=\mathcal{D}_{|\mathbf{i}|}$ is a scheme over $\mathbb{Z}$ obtained by gluing the $\mathcal{D}$-tori for all seeds mutation equivalent to the seed $\mathbf{i}$ using the above birational maps.
\end{definition}

Our goal in the next section is to describe a space of laminations which is related to the symplectic double in the same way that the spaces of $\mathcal{A}$- and $\mathcal{X}$-laminations are related to the cluster $\mathcal{A}$- and $\mathcal{X}$-varieties. These new laminations will be called \emph{$\mathcal{D}$-laminations}.

The space $\mathcal{D}_L(S,\mathbb{Q})$ of rational $\mathcal{D}$-laminations should be a tropical version of the symplectic double in the sense that it is identified with $\mathcal{D}(\mathbb{Q}^t)$. More concretely, this means that there are coordinates~$\{b_i,x_i\}_{i\in J}$ associated to the edges of an ideal triangulation of $S$, and if we flip the triangulation at an edge $k$, we get new coordinates 
\[
x_i' =
\begin{cases}
-x_k & \mbox{if } i=k \\ 
x_i+\varepsilon_{ki}\max\left(0,\sgn(\varepsilon_{ki})x_k\right) & \mbox{if } i\neq k
\end{cases}
\]
and 
\[
b_i' =
\begin{cases}
\max\biggr(x_k+\sum_{j|\varepsilon_{kj}>0}\varepsilon_{kj}b_j,-\sum_{j|\varepsilon_{kj}<0}\varepsilon_{kj}b_j\biggr)-\max(0,x_k)-b_k & \mbox{if } i=k \\
b_i & \mbox{if } i\neq k.
\end{cases} 
\]

\section{$\mathcal{D}$-laminations}
\label{sec:D-laminations}

\subsection{The notion of a $\mathcal{D}$-lamination}

Let $S$ be a compact oriented surface with finitely many marked points on its boundary, and let $S^\circ$ be the same surface with the opposite orientation. Recall that the \emph{double} $S_\mathcal{D}$ of $S$ is defined as the surface obtained by gluing $S$ and $S^\circ$ along corresponding boundary components and deleting the image of each marked point in the resulting surface.

By a (closed) \emph{curve} on $S_{\mathcal{D}}$, we mean an embedding of a circle into~$S_{\mathcal{D}}$. When we talk about homotopies, we mean homotopies within the class of such curves. A curve is called \emph{contractible} if it can be retracted to a point within this class of curves. It is \emph{special} if it is retractible to a puncture.

As before, a \emph{simple lamination} on such a surface is a finite collection $\gamma=\{\gamma_i\}$ of simple noncontractible nonspecial disjoint nonhomotopic closed curves considered up to homotopy. The surface~$S_\mathcal{D}$ comes equipped with a simple lamination~$\gamma$ given by the image of the boundary loops of~$S$ in~$S_\mathcal{D}$.

\begin{definition}
A \emph{rational $\mathcal{D}$-lamination} on $S_\mathcal{D}$ is the homotopy class of a collection of finitely many nonintersecting, noncontractible, and non-special closed curves with positive rational weights and a choice of orientation for each component of $\gamma$ which meets or is homotopic to a curve. A lamination containing homotopic curves of weights $a$ and $b$ is equivalent to the lamination with one curve removed and the weight $a+b$ on the other.
\end{definition}

The set of all rational $\mathcal{D}$-laminations on $S_\mathcal{D}$ will be denoted $\mathcal{D}_L(S,\mathbb{Q})$. We will write $\mathcal{D}_L(S,\mathbb{Z})$ for the set of all $\mathcal{D}$-laminations on $S_\mathcal{D}$ that can be represented by a collection of curves with integral weights. When there is no possibility of confusion, we will simply write~$\mathcal{D}_L$.

\subsection{Construction of coordinates}

\subsubsection{Preliminaries}

Let $l\in\mathcal{D}_L(S,\mathbb{Q})$. Then $l$ can be represented by a collection of nonintersecting simple closed curves on $S_\mathcal{D}$ such that all curves have the same rational weight. If there exists a curve homotopic to a component $\gamma_i$ of the simple lamination, then the orientation of $\gamma_i$ agrees with one of the surfaces $S$ or $S^\circ$, and we assume that the curve lies entirely on this surface. By cutting along $\partial S$, we recover the surfaces~$S$ and $S^\circ$ with curves drawn on them.

Let $\tilde{S}$ and $\tilde{S^\circ}$ denote the universal covers of $S$ and $S^\circ$, respectively. It is useful when drawing pictures to choose hyperbolic structures on $S$ and $S^\circ$. We can choose these hyperbolic structures so that if $\gamma$ is a component of $\partial S$ or $\partial S^\circ$ that does not meet a curve of the lamination, then the monodromy around $\gamma$ is parabolic. Then the universal covers can be obtained from $\mathbb{H}$ by removing countably many geodesic half disks. Below we will assume that these hyperbolic structures have been specified. None of our constructions will depend on the choice of hyperbolic structures.

\begin{lemma}
\label{lem:correspondence2}
There exists a map $\tilde{\iota}:\tilde{S}\rightarrow\tilde{S^\circ}$, unique up to the action of $\pi_1(S)$ by deck transformations, such that the diagram 
\[
\xymatrix{ 
\tilde{S} \ar[r]^{\tilde{\iota}} \ar[d] & \tilde{S^\circ} \ar[d] \\
S \ar_{\iota}[r] & S^\circ
}
\]
commutes. Points on $\partial\tilde{S}$ that project to points on the curves of the lamination are mapped bijectively to points on $\partial\tilde{S^\circ}$ that project to points on the curves, and this bijection preserves the natural order of these points.
\end{lemma}

\begin{proof}
The lifting criterion ensures that there is a map $\tilde{\iota}:\tilde{S}\rightarrow\tilde{S^\circ}$ lifting the composition $\tilde{S}\rightarrow S\rightarrow S^\circ$. It is easy to check that this map has the required properties.
\end{proof}

\subsubsection{Definition of $b_i$ and $x_i$}

To construct coordinates on $\mathcal{D}_L(S,\mathbb{Q})$, fix a point $l\in\mathcal{D}_L(S,\mathbb{Q})$ and let $i\in J$ be an edge of the ideal triangulation $T$. Realize the lamination $l$ by a collection of nonintersecting simple closed curves on $S_\mathcal{D}$ such that all curves have the same rational weight. By cutting along the image of $\partial S$, we recover the surfaces~$S$ and~$S^\circ$ with curves drawn on them.

Suppose the edge $i$ corresponds to an arc connecting two holes $\gamma_1$ and~$\gamma_2$ in~$S$ which meet the curves of $l$. Choose a pair of geodesics $g_1$ and $g_2$ in $\partial\tilde{S}$ which project to these boundary components. Deform the curve connecting the holes by winding its endpoint around the holes infinitely many times in the direction prescribed by the orientation. This corresponds to deforming the preimage $\tilde{i}$ in~$\tilde{S}$ so that its endpoints coincide with endpoints of $g_1$ and $g_2$. If we let $i^\circ$ denote the image of $i$ under the tautological map $S\rightarrow S^\circ$, then we can apply the same procedure to $i^\circ$ to get an arc $\tilde{i^\circ}$ in $\tilde{S^\circ}$.
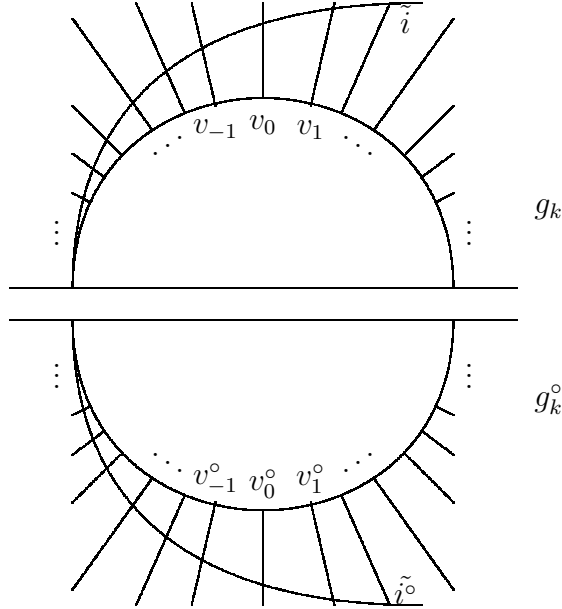
\[
\xy 0;/r.50pc/: 
(-12,-6)*{}="1"; 
(12,-6)*{}="2"; 
"1";"2" **\crv{(-12,10) & (12,10)}; 
(-16,-6)*{}="X"; 
(16,-6)*{}="Y"; 
"X";"Y" **\dir{-}; 
(18,-1)*{g_k}; 
(-12,0)*{}="P1";
(-11,-0.5)*{}="Q1"; 
(-12,2.5)*{}="P2";
(-10,1)*{}="Q2"; 
(-12,5.5)*{}="P3";
(-9,2.5)*{}="Q3"; 
(-12,11)*{}="P4";
(-7,4)*{}="Q4";
(-8,12)*{}="P5";
(-5,5.2)*{}="Q5"; 
(-4.5,12)*{}="P6";
(-3,5.5)*{}="Q6"; 
(0,12)*{}="P7";
(0,6)*{}="Q7"; 
(4.5,12)*{}="P8";
(3,5.5)*{}="Q8"; 
(8,12)*{}="P9";
(5,5.2)*{}="Q9"; 
(12,11)*{}="P10";
(7,4)*{}="Q10"; 
(12,5.5)*{}="P11";
(9,2.5)*{}="Q11"; 
(12,2.5)*{}="P12";
(10,1)*{}="Q12"; 
(12,0)*{}="P13";
(11,-0.5)*{}="Q13"; 
"P1";"Q1" **\crv{(-12,0) & (-11,-0.5)}; 
"P2";"Q2" **\crv{(-12,2.5) & (-10,1)}; 
"P3";"Q3" **\crv{(-12,5.5) & (-9,2.5)}; 
"P4";"Q4" **\crv{(-12,11) & (-7,4)}; 
"P5";"Q5" **\crv{(-8,12) & (-5,5.2)}; 
"P6";"Q6" **\crv{(-4.5,12) & (-3,5.5)}; 
"P7";"Q7" **\crv{(0,12) & (0,6)}; 
"P8";"Q8" **\crv{(4.5,12) & (3,5.5)}; 
"P9";"Q9" **\crv{(8,12) & (5,5.2)}; 
"P10";"Q10" **\crv{(12,11) & (7,4)}; 
"P11";"Q11" **\crv{(12,5.5) & (9,2.5)}; 
"P12";"Q12" **\crv{(12,2.5) & (10,1)}; 
"P13";"Q13" **\crv{(12,0) & (11,-0.5)}; 
(-3,4)*{v_{-1}};
(0,4.3)*{v_0};
(3,4)*{v_1};
(-6,3)*{\Ddots};
(6,3.5)*{\ddots};
(13,-2)*{\vdots};
(-13,-2)*{\vdots};
(10,12)*{}="B";
"1";"B" **\crv{(-12,9) & (0,12)}; 
(9,11)*{\tilde{i}};
(-16,-8)*{}="Z"; 
(16,-8)*{}="W"; 
"Z";"W" **\dir{-}; 
(-12,-8)*{}="1"; 
(12,-8)*{}="2"; 
"1";"2" **\crv{(-12,-24) & (12,-24)}; 
(18,-13)*{g_k^\circ};
(-12,-14)*{}="P1";
(-11,-13.5)*{}="Q1"; 
(-12,-16.5)*{}="P2";
(-10,-15)*{}="Q2"; 
(-12,-19.5)*{}="P3";
(-9,-16.5)*{}="Q3"; 
(-12,-25)*{}="P4";
(-7,-18)*{}="Q4"; 
(-8,-26)*{}="P5";
(-5,-19.2)*{}="Q5"; 
(-4.5,-26)*{}="P6";
(-3,-19.5)*{}="Q6"; 
(0,-26)*{}="P7";
(0,-20)*{}="Q7"; 
(4.5,-26)*{}="P8";
(3,-19.5)*{}="Q8"; 
(8,-26)*{}="P9";
(5,-19.2)*{}="Q9"; 
(12,-25)*{}="P10";
(7,-18)*{}="Q10"; 
(12,-19.5)*{}="P11";
(9,-16.5)*{}="Q11"; 
(12,-16.5)*{}="P12";
(10,-15)*{}="Q12"; 
(12,-14)*{}="P13";
(11,-13.5)*{}="Q13"; 
"P1";"Q1" **\crv{(-12,-14) & (-11,-13.5)}; 
"P2";"Q2" **\crv{(-12,-16.5) & (-10,-15)}; 
"P3";"Q3" **\crv{(-12,-19.5) & (-9,-16.5)}; 
"P4";"Q4" **\crv{(-12,-25) & (-12,-25)}; 
"P5";"Q5" **\crv{(-8,-26) & (-5,-19.2)}; 
"P6";"Q6" **\crv{(-4.5,-26) & (-3,-19.5)}; 
"P7";"Q7" **\crv{(0,-26) & (0,-20)}; 
"P8";"Q8" **\crv{(4.5,-26) & (3,-19.5)}; 
"P9";"Q9" **\crv{(8,-26) & (5,-19.2)}; 
"P10";"Q10" **\crv{(12,-25) & (7,-18)}; 
"P11";"Q11" **\crv{(12,-19.5) & (9,-16.5)}; 
"P12";"Q12" **\crv{(12,-16.5) & (10,-15)}; 
"P13";"Q13" **\crv{(12,-14) & (11,-13.5)}; 
(6,-17)*{\Ddots};
(-6,-16.5)*{\ddots};
(13,-11)*{\vdots};
(-13,-11)*{\vdots};
(-3,-18)*{v_{-1}^\circ};
(0,-18.3)*{v_0^\circ};
(3,-18)*{v_1^\circ};
(10,-26)*{}="B";
"1";"B" **\crv{(-12,-23) & (0,-26)}; 
(9,-25)*{\tilde{i^\circ}};
\endxy
\]

Observe that the curves of the lamination that end on~$\gamma_k$ can be lifted to infinitely many curves in~$\tilde{S}$ that end on~$g_k$. Similarly, if $g_k^\circ$ is the geodesic in~$\tilde{S^\circ}$ that projects to $\gamma_k$, then the curves in~$S^\circ$ that end on~$\gamma_k$ can be lifted to infinitely many curves in~$\tilde{S^\circ}$ that end on~$g_k^\circ$. Label the endpoints of curves on $g_k$ by the symbols~$v_\alpha$ ($\alpha\in\mathbb{Z}$), and label the endpoints of curves on $g_k^\circ$ by~$v_\alpha^\circ$ ($\alpha\in\mathbb{Z}$). By Lemma~\ref{lem:correspondence2}, there is a map $\tilde{\iota}:\tilde{S}\rightarrow\tilde{S^\circ}$ that projects to the natural map $\iota:S\rightarrow S^\circ$ and is unique up to the action of $\pi_1(S)$ by deck transformations. This provides a bijection 
\[
f:\{\text{vertices }v_\alpha\}\rightarrow\{\text{vertices }v_\alpha^\circ\}
\]
which preserves the order of the vertices. Choose a vertex $v_{\alpha(k)}$ on $g_k$ and let $v_{\beta(k)}^\circ$ be the corresponding vertex given by $v_{\beta(k)}^\circ=f(v_{\alpha(k)})$. We can choose $v_{\alpha(k)}$ in such a way that the curve ending at $v_{\alpha(k)}$ intersects $\tilde{i}$ and the curve ending at $v_{\beta(k)}^\circ$ intersects $\tilde{i^\circ}$. Notice that if we choose a different vertex $v_{\alpha(k)}$, then the vertex $v_{\beta(k)}^\circ$ will change by a corresponding amount.

By construction, the curve that ends at $v_{\alpha(k)}$ intersects $\tilde{i}$ at some point $p(k)$. Denote by~$a_i$ half the number of intersections between the lifted curves of $l$ and the lifted edge $\tilde{i}$ between the points~$p(1)$ and~$p(2)$. Similarly, the curve that ends at~$v_{\beta(k)}^\circ$ intersects $\tilde{i^\circ}$ at some point $p^\circ(k)$. Let $a_i^\circ$ be half the number of intersections between the lifted curves of $l$ and the lifted edge $\tilde{i^\circ}$ between the points~$p^\circ(1)$ and~$p^\circ(2)$. We can then define a coordinate associated to the edge $i$ by 
\[
b_i=a_i^\circ-a_i.
\]
This defines $b_i$ when $i$ corresponds to an arc connecting two holes. If one or both of the endpoints of $i$ are punctures, then $i$ intersects only finitely many curves near these punctures, and $i^\circ$ intersects only finitely many curves near the corresponding punctures in~$S^\circ$. Thus we can associate the half intersection numbers $a_i$ and $a_i^\circ$ to these arcs as before, and we can define $b_i$ by the above formula. One can show that the $|J|$ numbers obtained in this way are independent of all choices made in the construction.

In addition to the $b_i$, there are $|J|$ numbers $x_i$ associated to a point in the space $\mathcal{D}_L(S,\mathbb{Q})$. Given a point of $\mathcal{D}_L(S,\mathbb{Q})$, these are simply defined as the $X$-coordinates of the $\mathcal{X}$-lamination on the surface $S$ obtained by cutting $S_\mathcal{D}$ along the image of $\partial S$.

\subsubsection{Reconstruction}

We will now show that the $b_i$ and $x_i$ are indeed coordinates on $\mathcal{D}^+(S)$.

\begin{proposition}
The numbers $b_i$ and $x_i$ provide a bijection 
\[
\mathcal{D}_L(S,\mathbb{Q})\rightarrow\mathbb{Q}^{2|J|}.
\]
\end{proposition}

\begin{proof}
Call this map $\phi$. We will construct a map $\psi:\mathbb{Q}^{2|J|}\rightarrow\mathcal{D}_L(S,\mathbb{Q})$ and prove that $\phi$ and $\psi$ are inverses.

\Needspace*{2\baselineskip}
\begin{step}[1]
Definition of $\psi:\mathbb{Q}^{2|J|}\rightarrow\mathcal{D}_L(S,\mathbb{Q})$.
\end{step}

Suppose we are given rational numbers $b_j$ and $x_j$ for every $j\in J$. By the reconstruction procedure for $\mathcal{X}$-laminations, we can use the $x_j$ to glue together ideal triangles to get a region $\tilde{S}\subseteq\mathbb{H}$ with geodesic boundary together with infinitely many curves.

Let $t_0$ be a triangle in $\tilde{S}$. There are infinitely many curves connecting each pair of adjacent sides of this triangle. For each vertex $p$ of $t_0$, choose a curve connecting the two sides of this triangle that meet at $p$. Then for every edge $i$, there is a number $a_i$ defined as half the total weight of the curves that intersect $i$ between the distinguished curves. Define 
\[
a_i^\circ=b_i+a_i.
\]
We will use these numbers $a_i^\circ$ to construct another region $\tilde{S^\circ}$ with geodesic boundary in a copy of $\mathbb{H}$.

To construct $\tilde{S^\circ}$, let $u_0$ be an ideal triangle with infinitely many curves connecting each pair of adjacent sides as before. If we choose the distinguished curves on $t_0$ sufficiently close to the vertices, then we can choose a triple of distinguished curves near the vertices of $u_0$ so that $a_i^\circ$ is half the total weight of the curves that intersect an edge of $u_0$ between distinguished curves.

Now suppose $t$ is a triangle adjacent to $t_0$ in the triangulation of $\tilde{S}$. As before we can choose a distinguished curve near each vertex of $t$ to get a triple of numbers~$a_i$. Then we can draw an ideal triangle $u$ with infinitely many curves connecting each pair of adjacent sides, and there are distinguished curves near each vertex of $u$ realizing the numbers~$a_i^\circ$. We can choose the distinguished curves on $t_0$ and $t$ so that they coincide at the common edge $t_0\cap t$, and then we can glue $u_0$ and $u$ so that their distinguished curves coincide. Continuing this process inductively, we obtain the desired space $\tilde{S^\circ}$.

Now the region $\tilde{S}$ is obtained from $\mathbb{H}$ by removing infinitely many geodesic half disks. We can extend $\tilde{S}$ to a larger region by gluing a copy of $\tilde{S^\circ}$ in each of these half disks in such a way that corresponding curves are identified. The resulting region is again obtained from $\mathbb{H}$ by removing infinitely many geodesic half disks, and we can enlarge it by gluing a copy of $\tilde{S}$ in each of these half disks. Continuing this process ad infinitum, we partition the hyperbolic plane into ideal triangles together with infinitely many curves. Quotienting this space by the group of deck transformations, we recover the surface $S_{\mathcal{D}}$ with a collection of curves. There may be infinitely many curves homotopic to some component $\gamma_i$ of $\gamma$. In this case we delete a maximal collection of such curves between the distinguished curves. If there are any remaining curves homotopic to $\gamma_i$, then they all lie on one of the surfaces $S$ or $S^\circ$, and we choose the orientation of $\gamma_i$ to agree with the orientation of this surface.

\Needspace*{2\baselineskip}
\begin{step}[2]
Proof of $\phi\circ\psi=1_{\mathbb{Q}^{2|J|}}$.
\end{step}

Let $b_j$ and $x_j$ ($j\in J$) be given. By the reconstruction procedure described above, we construct triangulated regions $\tilde{S}\subseteq\mathbb{H}$ and $\tilde{S}^\circ\subseteq{H}$ which we then glue together to get a triangulation of the hyperbolic plane together with a collection of curves. From these data, we get a $\mathcal{D}$-lamination $l$ on $S_{\mathcal{D}}$. We want to show that the coordinates of $l$ are the numbers~$b_j$ and~$x_j$.

We can lift $l$ to its universal cover, which is exactly the triangulated surface~$\tilde{S}_{\mathcal{D}}$ that we got by gluing together copies of $\tilde{S}$ and $\tilde{S}^\circ$. Consider a copy of $\tilde{S}$ in $\tilde{S}_{\mathcal{D}}$, and let $\tilde{i}$ be an edge of the triangulation of $\tilde{S}$. To find the coordinate of $l$ corresponding to the edge $\tilde{i}$, we must choose near each endpoint of $\tilde{i}$ a distinguished curve in $\tilde{S}$ which intersects $\tilde{i}$. Now consider a copy of $\tilde{S}^\circ$ adjacent to $\tilde{S}$ in $\tilde{S}_{\mathcal{D}}$. Choose a basepoint $x$ on the geodesic separating these regions. Suppose the edge~$\tilde{i}$ is asymptotic to an endpoint of the boundary geodesic $g$ of $\tilde{S}$. Then there is a distinguished curve $c$ that intersects $\tilde{i}$ and $g$. Let $p$ be the point of intersection with $g$. Draw a curve $\tilde{\alpha}$ in~$\tilde{S}$ from the point~$p$ to the point~$x$.
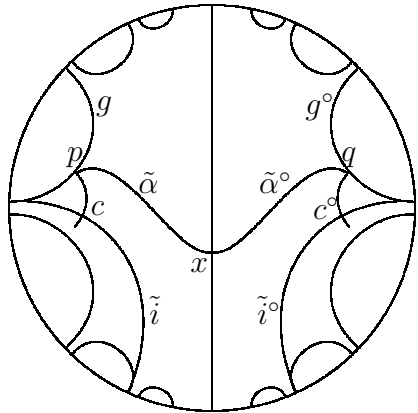
\[
\xy 0;/r.40pc/: 
(0,0)*\xycircle(16,16){-};
(-10.7,2.75)*{}="p"; 
(-10.8,-1.5)*{}="p1"; 
(0,-3.5)*{}="x"; 
(10.7,2.75)*{}="q"; 
(10.8,-1.5)*{}="q1"; 
(0,16)*{}="X"; 
(0,-16)*{}="Y"; 
(-3.12,15.69)*{}="A1"; 
(-5.8,14.9)*{}="A2"; 
(-6.6,14.5)*{}="A3"; 
(-11,11.5)*{}="A4"; 
(-11.5,11)*{}="A5"; 
(-15.9,0.5)*{}="A6"; 
(-15.9,-0.5)*{}="A7"; 
(-11.5,-11)*{}="A8"; 
(-11,-11.5)*{}="A9"; 
(-6.6,-14.5)*{}="A10"; 
(-5.8,-14.9)*{}="A11"; 
(-3.12,-15.69)*{}="A12"; 
(3.12,15.69)*{}="B1"; 
(5.8,14.9)*{}="B2"; 
(6.6,14.5)*{}="B3"; 
(11,11.5)*{}="B4"; 
(11.5,11)*{}="B5"; 
(15.9,0.5)*{}="B6"; 
(15.9,-0.5)*{}="B7"; 
(11.5,-11)*{}="B8"; 
(11,-11.5)*{}="B9"; 
(6.6,-14.5)*{}="B10"; 
(5.8,-14.9)*{}="B11"; 
(3.12,-15.69)*{}="B12"; 
"p";"x" **\crv{(-7,5) & (-3,-4)}; 
"x";"q" **\crv{(3,-4) & (7,5)}; 
"X";"Y" **\dir{-}; 
"A1";"A2" **\crv{(-2.9,14.4) & (-5.1,13.4)}; 
"A3";"A4" **\crv{(-5,12) & (-9,9)}; 
"A5";"A6" **\crv{(-6.6,6.5) & (-11,0.4)}; 
"A7";"A8" **\crv{(-11,-0.4) & (-6.6,-6.5)}; 
"A9";"A10" **\crv{(-9,-9) & (-5,-12)}; 
"A11";"A12" **\crv{(-5.1,-13.4) & (-2.9,-14.4)}; 
"B1";"B2" **\crv{(2.9,14.4) & (5.1,13.4)}; 
"B3";"B4" **\crv{(5,12) & (9,9)}; 
"B5";"B6" **\crv{(6.6,6.5) & (11,0.4)}; 
"B7";"B8" **\crv{(11,-0.4) & (6.6,-6.5)}; 
"B9";"B10" **\crv{(9,-9) & (5,-12)}; 
"B11";"B12" **\crv{(5.1,-13.4) & (2.9,-14.4)}; 
"A6";"A10" **\crv{(-8,1) & (-3,-6)}; 
"B6";"B10" **\crv{(8,1) & (3,-6)}; 
"p";"p1" **\crv{(-9.6,1.5) & (-9.6,0)}; 
"q";"q1" **\crv{(9.6,1.5) & (9.6,0)}; 
(-1,-4.5)*{x}; 
(-5,2)*{\tilde{\alpha}}; 
(5,2)*{\tilde{\alpha}^\circ}; 
(-10.8,4)*{p}; 
(10.8,4)*{q}; 
(-9,0)*{c}; 
(9,0)*{c^\circ}; 
(-4.5,-8)*{\tilde{i}}; 
(4.5,-8)*{\tilde{i}^\circ}; 
(-8.5,8)*{g}; 
(8.5,8)*{g^\circ}; 
\endxy
\]
This curve $\tilde{\alpha}$ projects to a curve~$\alpha$ on the surface $S$, and we can apply the map~$\iota$ to get a curve $\alpha^\circ$ on $S^\circ$. Finally, we lift $\alpha^\circ$  to a curve $\tilde{\alpha}^\circ$ in~$S^\circ$ that starts at $x$. This curve $\tilde{\alpha}^\circ$ ends at some point $q=\tilde{\iota}(p)$, and there is a unique lifted curve $c^\circ$ of the lamination that passes through $q$. In this way, we get distinguished curves near the endpoints of the edge $\tilde{i}^\circ$. These can be used to define the coordinates of~$l$.

On the other hand, consider the monodromy along the curve in $S_{\mathcal{D}}$ obtained by concatenating $\alpha$ and $\alpha^\circ$. It maps the region~$\tilde{S}$ isometrically into the region bounded by~$g^\circ$. In the construction of $l$ from the coordinates~$b_j$ and~$x_j$, there is a correspondence between curves that end on $g$ and curves that end on $g^\circ$, and this correspondence is obtained by applying this transformation. Thus the distinguished curves used to construct $l$ agree with the ones obtained using the map $\tilde{\iota}$.

It follows that the $b$-coordinates of $l$ are simply the numbers $b_j$ that we started with. It is easy to see that the $x$-coordinates of $l$ are the numbers $x_j$. This completes Step~2 of the proof.

\Needspace*{3\baselineskip}
\begin{step}[3]
Proof of $\psi\circ\phi=1_{\mathcal{D}_L}$.
\end{step}

Let $l\in\mathcal{D}_L(S,\mathbb{Q})$ be given. To calculate the coordinates of~$l$, let us pass to the universal cover~$\tilde{S}_{\mathcal{D}}$ of~$S_{\mathcal{D}}$ and choose a copy of $\tilde{S}$ in~$\tilde{S}_{\mathcal{D}}$. We associate numbers $x_i$ to the edges of the triangulation in the usual way. If $\tilde{i}$ is an edge of the triangulation of~$\tilde{S}$, then we can choose a distinguished curve near each endpoint of~$\tilde{i}$. Using the construction described above involving lifts of the map $\iota:S\rightarrow S^\circ$, we can get a pair of distinguished curves near the endpoints of a corresponding edge $\tilde{i}^\circ$ in~$\tilde{S}^\circ$. We can use these curves to calculate the numbers~$a_i$ and~$a_i^\circ$, and thus the coordinate $b_i$.

We can now use the $x_i$ to reconstruct the region~$\tilde{S}$. We must choose a distinguished curve near each endpoint of the edge~$\tilde{i}$, and we can assume that these are exactly the ones used above to define~$b_i$. We can use the curves to begin constructing the space $\tilde{S}^\circ$ with a collection of curves. By gluing together copies of $\tilde{S}$ and $\tilde{S}^\circ$, we recover $\tilde{S}_{\mathcal{D}}$ with a collection of curves. Quotienting this space by the action of $\pi_1(S_{\mathcal{D}})$, we recover the point $l\in\mathcal{D}_L(S,\mathbb{Q})$. This completes Step~3 of the proof.
\end{proof}

\subsubsection{Transformation rule}

We have defined the $b_i$ and $x_i$ coordinates in terms of a fixed ideal triangulation of~$S$. The next result says how these coordinates vary when we change the triangulation.

\begin{proposition}
\label{prop:flip}
A regular flip at an edge $k$ of the triangulation changes the coordinates $x_i$ and~$b_i$ to new coordinates $x_i'$ and~$b_i'$ given by the formulas 
\begin{align*}
x_i' &=
\begin{cases}
-x_k & \mbox{if } i=k \\ 
x_i+\varepsilon_{ki}\max\left(0,\sgn(\varepsilon_{ki})x_k\right) & \mbox{if } i\neq k,
\end{cases} \\
b_i' &=
\begin{cases}
\max\biggr(x_k+\sum_{j|\varepsilon_{kj}>0}\varepsilon_{kj}b_j,-\sum_{j|\varepsilon_{kj}<0}\varepsilon_{kj}b_j\biggr)-\max(0,x_k)-b_k & \mbox{if } i=k \\
b_i & \mbox{if } i\neq k.
\end{cases} 
\end{align*}
\end{proposition}

\begin{proof}
By definition, the coordinates $x_i$ on $\mathcal{D}_L(S,\mathbb{Q})$ coincide with the coordinates of the $\mathcal{X}$-lamination on the surface $S$. Therefore the transformation rule for the $x_i$ is just the usual transformation rule for the coordinates of an $\mathcal{X}$-lamination.

To prove the second rule, deform the curves corresponding to edges of the ideal triangulation by winding their endpoints around the holes infinitely many times. Lift these deformed curves to the universal cover of $S$, and consider the ideal quadrilateral formed by the two triangles adjacent to the geodesic arc $\tilde{k}$ corresponding to $k$.
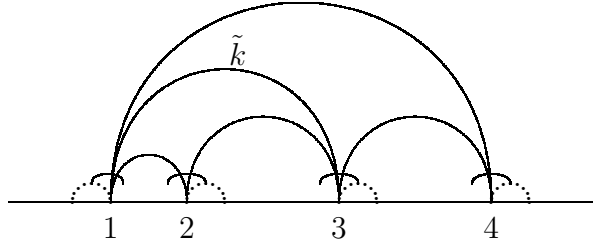
\[
\xy 0;/r.40pc/: 
(4,6)*{\tilde{k}}; 
(-6,-8)*{1}; 
(0,-8)*{2}; 
(12,-8)*{3}; 
(24,-8)*{4}; 
(-9,-6)*{}="A"; 
(3,-6)*{}="B"; 
(15,-6)*{}="C"; 
(27,-6)*{}="D"; 
(-6,-6)*{}="1"; 
(0,-6)*{}="2"; 
(12,-6)*{}="3"; 
(24,-6)*{}="4"; 
"A";"1" **\crv{~*=<2pt>{.}(-9,-4) & (-6,-4)}; 
"2";"B" **\crv{~*=<2pt>{.}(0,-4) & (3,-4)}; 
"3";"C" **\crv{~*=<2pt>{.}(12,-4) & (15,-4)}; 
"4";"D" **\crv{~*=<2pt>{.}(24,-4) & (27,-4)}; 
"1";"2" **\crv{(-6,-1) & (0,-1)}; 
"2";"3" **\crv{(0,3) & (12,3)}; 
"3";"4" **\crv{(12,3) & (24,3)}; 
"1";"4" **\crv{(-6,15) & (24,15)}; 
"1";"3" **\crv{(-6,8) & (12,8)}; 
(-7.5,-4.5)*{}="P"; 
(-5,-4.5)*{}="Q"; 
(-1.5,-4.5)*{}="R"; 
(1.5,-4.5)*{}="S"; 
(10.5,-4.5)*{}="T"; 
(13.5,-4.5)*{}="U"; 
(22.5,-4.5)*{}="V"; 
(25.5,-4.5)*{}="W"; 
"P";"Q" **\crv{(-7.5,-3.5) & (-5,-3.5)}; 
"R";"S" **\crv{(-1.5,-3.5) & (1.5,-3.5)}; 
"T";"U" **\crv{(10.5,-3.5) & (13.5,-3.5)}; 
"V";"W" **\crv{(22.5,-3.5) & (25.5,-3.5)}; 
(-14,-6)*{}="X"; 
(32,-6)*{}="Y"; 
"X";"Y" **\dir{-}; 
\endxy
\]

Number the vertices of this quadrilateral in counterclockwise order so that the edge $\tilde{k}$ joins vertices~1 and~3. In the construction of the $b_i$, we considered near each endpoint $e$ of a lifted edge $\tilde{i}$ a curve that intersects $\tilde{i}$ and projects down to a curve of the lamination. We can choose these curves to intersect both edges of the quadrilateral that meet at $e$.

Our construction will then associate a number $a_{ij}$ to the geodesic connecting~$i$ and~$j$. There is a corresponding ideal quadrilateral in the universal cover of $S^\circ$ and a number $a_{ij}^\circ$ corresponding to the edge connecting vertices~$i$ and~$j$ of this ideal quadrilateral.

If we flip at the edge~$k$, then the ideal quadrilateral in the universal cover of~$S$ is replaced by the same ideal quadrilateral triangulated by the arc from~2 to~4, and the number associated to this new arc is 
\[
a_{24}=\max(a_{12}+a_{34},a_{14}+a_{23})-a_{13}.
\]
Similarly, the number $a_{13}^\circ$ transforms to 
\[
a_{24}^\circ=\max(a_{12}^\circ+a_{34}^\circ,a_{14}^\circ+a_{23}^\circ)-a_{13}^\circ.
\]
Applying these rules to the difference $b_{13}=a_{13}^\circ-a_{13}$, we obtain 
\begin{align*}
b_{24} &=\max(a_{12}^\circ+a_{34}^\circ,a_{14}^\circ+a_{23}^\circ) - \max(a_{12}+a_{34},a_{14}+a_{23})-(a_{13}^\circ-a_{13}) \\
&= \max(a_{12}^\circ+a_{34}^\circ-a_{14}-a_{23},a_{14}^\circ+a_{23}^\circ-a_{14}-a_{23}) \\
&\qquad - \max(0,a_{12}+a_{34}-a_{14}-a_{23})-(a_{13}^\circ-a_{13}) \\
&= \max(x_{13}+b_{12}+b_{34},b_{14}+b_{23}) - \max(0,x_{13}) - b_{13}
\end{align*}
where we have used the relation $x_{13}=a_{12}+a_{34}-a_{14}-a_{23}$. This proves the transformation rule for the $b_i$.
\end{proof}

\subsection{Real $\mathcal{D}$-laminations}

Since the transformation rules in Proposition \ref{prop:flip} are continuous with respect to the standard topology on $\mathbb{Q}^{2|J|}$, the coordinates define a natural topology on $\mathcal{D}_L(S,\mathbb{Q})$. We define the space of \emph{real $\mathcal{D}$-laminations} as the metric space completion of $\mathcal{D}_L(S,\mathbb{Q})$.

This space of real $\mathcal{D}$-laminations is identified with the space $\mathcal{D}(\mathbb{R}^t)$ of $\mathbb{R}^t$-points of the symplectic double associated to the surface~$S$. There is a natural action of the group $\mathbb{R}_{>0}$ on this space where an element $\lambda\in\mathbb{R}_{>0}$ acts by multiplying the coordinates in any coordinate system by~$\lambda$. The \emph{spherical tropical space} $\mathcal{SD}(\mathbb{R}^t)$ is the quotient 
\[
\mathcal{SD}(\mathbb{R}^t)=\left(\mathcal{D}(\mathbb{R}^t)-0\right)/\mathbb{R}_{>0}.
\]
By Proposition~2.2 of~\cite{infinity}, we know that $\mathcal{SD}(\mathbb{R}^t)$ can be viewed as a boundary of the space~$\mathcal{D}^+(S)$ of positive real points of the symplectic double associated to the surface~$S$.

\section{Cluster algebras and $F$-polynomials}
\label{sec:ClusterAlgebrasAndFPolynomials}

\subsection{Cluster algebras with coefficients}

In the following definition, $\mathbb{P}$ denotes an arbitrary semifield. One can show that the group ring $\mathbb{ZP}$ is an integral domain, and hence we can form its fraction field $\mathbb{QP}$. We will write~$\mathcal{F}$ for a field isomorphic to the field of rational functions in~$n$ independent variables with coefficients in~$\mathbb{QP}$.

\begin{definition}
A \emph{labeled seed} $(\mathbf{x},\mathbf{y},B)$ consists of a skew-symmetrizable $n\times n$ integer  matrix~$B=(b_{ij})$, an $n$-tuple $\mathbf{y}=(y_1,\dots,y_n)$ of elements of $\mathbb{P}$, and an $n$-tuple $\mathbf{x}=(x_1,\dots,x_n)$ of elements of $\mathcal{F}$ such that the $x_i$ are algebraically independent over $\mathbb{QP}$ and $\mathcal{F}=\mathbb{QP}(x_1,\dots,x_n)$.
\end{definition}

\begin{definition}
Let $(\mathbf{x},\mathbf{y},B)$ be a labeled seed, and let $k\in\{1,\dots,n\}$. Then we define a new seed $(\mathbf{x}',\mathbf{y}',B')$, called the seed obtained by \emph{mutation} in the direction~$k$ as follows:
\begin{enumerate}
\item The entries of $B'=(b_{ij}')$ are given by 
\[
b_{ij}'=
\begin{cases}
-b_{ij} & \mbox{if } k\in\{i,j\} \\
b_{ij}+\frac{|b_{ik}|b_{kj}+b_{ik}|b_{kj}|}{2} & \mbox{if } k\not\in\{i,j\}.
\end{cases}
\]
\item The elements of the $n$-tuple $\mathbf{y}'=(y_1',\dots,y_n')$ are given by 
\begin{align*}
y_j'=
\begin{cases}
y_k^{-1} & \mbox{if } j=k \\
y_jy_k^{[b_{kj}]_+}(y_k\oplus1)^{-b_{kj}} & \mbox{if } j\neq k
\end{cases}
\end{align*}
where we are using the notation $[b]_+=\max(b,0)$.
\item The elements of the $n$-tuple $\mathbf{x}'=(x_1',\dots,x_n')$ are given by 
\begin{align*}
x_j' =
\begin{cases}
\frac{y_k\prod_{i|b_{ik>0}}x_i^{b_{ik}} + \prod_{i|b_{ik<0}}x_i^{-b_{ik}}}{(y_k\oplus1)x_k} & \mbox{if } j=k \\
x_j & \mbox{if } j\neq k.
\end{cases}
\end{align*}
\end{enumerate}
\end{definition}

\begin{definition}
We denote by $\mathbb{T}_n$ an $n$-regular tree with edges labeled by the numbers~$1,\dots,n$ in such a way that the $n$ edges emanating from any vertex have distinct labels. A \emph{cluster pattern} is an assignment of a labeled seed $\Sigma_t=(\mathbf{x}_t,\mathbf{y}_t,B_t)$ to every vertex $t\in\mathbb{T}_n$ so that if $t$ and $t'$ are vertices connected by an edge labeled~$k$, then $\Sigma_{t'}$ is obtained from $\Sigma_t$ by a mutation in the direction $k$. We will use the following notation for the elements of $\Sigma_t$:
\[
\mathbf{x}_t=(x_{1;t},\dots,x_{n;t}), \quad \mathbf{y}_t=(y_{1;t},\dots,y_{n;t}), \quad B_t=(b_{ij}^t).
\]
\end{definition}

\begin{definition}
Given a cluster pattern $t\mapsto(\mathbf{x}_t,\mathbf{y}_t,B_t)$, we form the set of all cluster variables in all seeds of the cluster pattern:
\[
\mathcal{S}=\{x_{l;t}:t\in\mathbb{T}_n,1\leq l\leq n\}.
\]
Then the \emph{cluster algebra} with coefficients in $\mathbb{P}$ is the $\mathbb{ZP}$-subalgebra of $\mathcal{F}$ generated by elements in this set $\mathcal{S}$.
\end{definition}

\subsection{Principal coefficients}

In \cite{FZIV}, Fomin and Zelevinsky define a \emph{cluster algebra with principal coefficients} at a vertex $t_0\in\mathbb{T}_n$ to be a cluster algebra with $\mathbb{P}=\mathrm{Trop}(y_1,\dots,y_n)$ and $\mathbf{y}_{t_0}=(y_1,\dots,y_n)$. Let~$\mathcal{A}$ be such a cluster algebra, and let $\Sigma_{t_0}=(\mathbf{x}_{t_0},\mathbf{y}_{t_0},B_{t_0})$ be its initial seed with 
\[
\mathbf{x}_{t_0}=(x_1,\dots,x_n), \quad \mathbf{y}_{t_0}=(y_1,\dots,y_n), \quad B_{t_0}=(b_{ij}^0).
\]
By iterating the exchange relations, we can express any cluster variable $x_{l;t}$ as a subtraction-free rational function of the variables $x_1,\dots,x_n,y_1,\dots,y_n$. We will denote this subtraction-free rational function by 
\[
X_{l;t}\in\mathbb{Q}_\mathrm{sf}(x_1,\dots,x_n,y_1,\dots,y_n).
\]
We will denote by $F_{l;t}\in\mathbb{Q}_\mathrm{sf}(y_1,\dots,y_n)$ the subtraction-free rational function obtained from~$X_{l;t}$ by specializing all the $x_i$ to 1. Thus 
\[
F_{l;t}(y_1,\dots,y_n)=X_{l;t}(1,\dots,1,y_1,\dots,y_n).
\]
By the Laurent phenomenon theorem of Fomin and Zelevinsky, $X_{l;t}$ is a Laurent polynomial in $x_1,\dots,x_n$ whose coefficients are integral polynomials in $y_1,\dots,y_n$, and $F_{l;t}$ is an integral polynomial in $y_1,\dots,y_n$. We will refer to $X_{l;t}$ and $F_{l;t}$ as the $X$- and \emph{$F$-polynomials}.

In addition to the $X$- and $F$-polynomials, Fomin and Zelevinsky define a $\mathbb{Z}^n$-grading on the ring $\mathbb{Z}[x_1^{\pm1},\dots,x_n^{\pm1},y_1,\dots,y_n]$ by 
\[
\deg(x_i) = \mathbf{e}_i, \quad \deg(y_i) = -\mathbf{b}_j^0
\]
where $\mathbf{e}_i$ is the $i$th standard basis vector in $\mathbb{Z}^n$ and $\mathbf{b}_j^0=\sum_i b_{ij}^0\mathbf{e}_i$ is the $j$th column of $B_{t_0}$. By a result of \cite{FZIV}, each $X$-polynomial is homogeneous with respect to this $\mathbb{Z}^n$-grading. The degree 
\[
\mathbf{g}_{l;t}=
\left( \begin{array}{ccc}
g_1 \\
\vdots \\
g_n \end{array} \right)
=\deg(X_{l;t})\in\mathbb{Z}^n
\]
is called the \emph{$\mathbf{g}$-vector} of the cluster variable $x_{l;t}$.

The notions of $F$-polynomials and $\mathbf{g}$-vectors are important because they allow us to express an arbitrary cluster variable in terms of the variables $x_1,\dots,x_n$ and $y_1,\dots,y_n$ of the initial seed $\Sigma_{t_0}$. To see this, we need one more piece of notation. It is a fact that any subtraction-free rational identity that holds in the semifield $\mathbb{Q}_\mathrm{sf}(u_1,\dots,u_n)$ will remain valid when we replace the $u_i$ by elements of an arbitrary semifield~$\mathbb{P}$. Thus if $f$ is a subtraction-free rational expression in $u_1,\dots,u_n$, there is a well defined element $f|_\mathbb{P}(y_1,\dots,y_n)$ of $\mathbb{P}$ obtained by evaluating $f$ at $y_1,\dots,y_n\in\mathbb{P}$.

\begin{proposition}[\cite{FZIV}, Corollary 6.3]
\label{prop:FZ63}
Let $\mathcal{A}$ be a cluster algebra over an arbitrary semifield $\mathbb{P}$ of coefficients. Then a cluster variable $x_{l;t}$ can be expressed in terms of the cluster variables at an initial seed as 
\[
x_{l;t}=\frac{F_{l;t}|_\mathcal{F}(\widehat{y}_1,\dots,\widehat{y}_n)}{F_{l;t}|_{\mathbb{P}}(y_1,\dots,y_n)}x_1^{g_1}\dots x_n^{g_n}
\]
where 
\[
\widehat{y}_j=y_j\prod_i x_i^{b_{ij}^0}.
\]
\end{proposition}

\subsection{Relation to the symplectic double}

Let us now specialize to the case where $\mathbb{P}=\mathbb{Q}_{\mathrm{sf}}(X_1,\dots,X_n)$ is the semifield of subtraction-free rational functions in $X_1,\dots,X_n$.

\begin{proposition}
If $\mathbb{P}=\mathbb{Q}_\mathrm{sf}(X_1,\dots,X_n)$, then the ambient field $\mathcal{F}$ is given by $\mathcal{F}=\mathbb{Q}(X_1,\dots,X_n,B_1,\dots,B_n)$ for algebraically independent variables $B_1,\dots,B_n$.
\end{proposition}

\begin{proof}
If $\mathbb{P}=\mathbb{Q}_\mathrm{sf}(X_1,\dots,X_n)$, then we have $\mathbb{QP}=\mathbb{Q}(X_1,\dots,X_n)$ and hence 
\begin{align*}
\mathcal{F} &= \mathbb{Q}(X_1,\dots,X_n)(B_1,\dots,B_n) \\
&= \mathbb{Q}(X_1,\dots,X_n,B_1,\dots,B_n)
\end{align*}
as desired.
\end{proof}

\begin{proposition}
If $\mathbb{P}=\mathbb{Q}_\mathrm{sf}(X_1,\dots,X_n)$, then a mutation in the direction $k$ transforms the $n$-tuples $(X_1,\dots,X_n)$ and $(B_1,\dots,B_n)$ into $n$-tuples $(X_1',\dots,X_n')$ and $(B_1',\dots,B_n')$ given by 
\begin{align*}
X_j'=
\begin{cases}
X_k^{-1} & \mbox{if } j=k \\
X_j{(1+X_k^{-\sgn(\varepsilon_{jk})})}^{-\varepsilon_{jk}} & \mbox{if } j\neq k
\end{cases}
\end{align*}
and 
\begin{align*}
B_j' =
\begin{cases}
\frac{X_k\prod_{i|\varepsilon_{ki>0}}B_i^{\varepsilon_{ki}} + \prod_{i|\varepsilon_{ki<0}}B_i^{-\varepsilon_{ki}}}{B_k(1+X_k)} & \mbox{if } i=k \\
B_j & \mbox{if } i\neq k
\end{cases}
\end{align*}
where $\varepsilon_{ij}=b_{ji}$.
\end{proposition}

\begin{proof}
The second formula obviously follows from the general mutation formula. To prove the first formula, observe that 
\begin{align*}
X_jX_k^{[b_{kj}]_+}(X_k\oplus1)^{-b_{kj}} &= X_jX_k^{[b_{kj}]_+}(X_k+1)^{-b_{kj}} \\
&= X_jX_k^{-(-b_{kj})}(X_k+1)^{-b_{kj}} \\
&= X_j{(X_k^{-1}X_k+X_k^{-1})}^{-b_{kj}} \\
&= X_j{(1+X_k^{-\sgn(b_{kj})})}^{-b_{kj}}
\end{align*}
for $b_{kj}>0$ and 
\begin{align*}
X_jX_k^{[b_{kj}]_+}(X_k\oplus1)^{-b_{kj}} &= X_j(1+X_k)^{-b_{kj}} \\
&= X_j{(1+X_k^{-\sgn(b_{kj})})}^{-b_{kj}}
\end{align*}
for $b_{kj}\leq0$.
\end{proof}

Thus we recover the mutation formulas for the symplectic double in the special case where $\mathbb{P}=\mathbb{Q}_\mathrm{sf}(X_1,\dots,X_n)$. In this case, we will denote a cluster algebra with coefficients in $\mathbb{P}$ by the symbol $\mathcal{D}$ and call it a \emph{cluster $\mathcal{D}$-algebra}.

\subsection{Cluster algebras from surfaces}

\subsubsection{Construction of cluster algebras}

In~\cite{FST} (see also~\cite{FT}), Fomin, Shapiro, and Thurston discuss the relationship between cluster algebras and the combinatorics of decorated surfaces. Here we will review their work and apply it to the cluster $\mathcal{D}$-algebra that we defined above.

The idea of~\cite{FST} is to associate to a decorated surface $S$ a corresponding cluster algebra. This cluster algebra is defined in such a way that each seed corresponds to a ``tagged triangulation'' of the surface $S$. An ordinary ideal triangulation is a special case of a tagged triangulation provided there are no self-folded triangles. Fomin, Shapiro, and Thurston assume that the surface $S$ is not a sphere with one, two, or three punctures, a monogon with zero or one puncture, or a bigon or triangle without punctures. According to Lemma~2.13 of~\cite{FST}, such a surface always admits an ideal triangulation $T$ with no self-folded triangles.

If $T$ is an ideal triangulation of $S$ with no self-folded triangles, then we get an exchange matrix $b_{ij}=\varepsilon_{ji}$ ($i,j\in J$), indexed by the internal edges of $T$. To each internal edge $i$ of $T$, we associate variables $x_i$ and~$y_i$. This defines a labeled seed, and hence a cluster algebra. This is the cluster algebra that Fomin, Shapiro, and Thurston associate to the surface~$S$.

If $c$ is any arc on $S$ which is an internal edge for some ideal triangulation and does not cut out a once punctured monogon, then there is a cluster variable $x_c$ in this cluster algebra corresponding to~$c$. In particular, this means that for any such arc $c$ on $S$ there is an associated $F$-polynomial 
\[
F_c(y_1,\dots,y_n)
\]
and a $\mathbf{g}$-vector $\mathbf{g}_c$.

We can apply the results of~\cite{FST} in the special case where the semifield of coefficients is a semifield of subtraction-free rational functions $\mathbb{P}=\mathbb{Q}_{\mathrm{sf}}(X_1,\dots,X_n)$. In this way, we get a cluster $\mathcal{D}$-algebra associated to a surface $S$. Applying Proposition~\ref{prop:FZ63} to this cluster algebra, we obtain the following result.

\begin{proposition}
\label{prop:specialization}
Let $\mathcal{D}$ be the cluster $\mathcal{D}$-algebra associated to a surface $S$, and let $B_c$ be a cluster variable of $\mathcal{D}$ corresponding to an arc~$c$ on~$S$. Then $B_c$ can be expressed in terms of the initial variables $B_1,\dots,B_n$ and $X_1,\dots,X_n$ by the formula 
\[
B_c=\frac{F_c(\widehat{X}_1,\dots,\widehat{X}_n)}{F_c(X_1,\dots,X_n)}B_1^{g_1}\dots B_n^{g_n}
\]
where 
\[
\widehat{X}_i=X_i\prod_j B_j^{\varepsilon_{ij}}.
\]
\end{proposition}

In Section~\ref{sec:TheCanonicalPairing}, we will discuss an extension of this result which will allow us to completely understand the canonical pairing between $\mathcal{D}$-laminations and points of the Teichm\"uller $\mathcal{D}$-space. In order to prove this result, we will need to understand the $\mathbf{g}$-vectors associated to arcs in $S$.

\subsubsection{Calculation of $\mathbf{g}$-vectors}

As part of their work on the positivity conjecture for cluster algebras from surfaces, Musiker, Schiffler, and Williams~\cite{MSW1} gave a formula for computing the $\mathbf{g}$-vector associated to an arc. Their construction associates, to any arc $c$, a graph $\bar{G}_{T,c}$ in the plane with labeled edges. This graph is obtained by gluing together ``tiles'' of the form
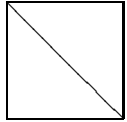
\[
\xy /l1.3pc/:
{\xypolygon4"A"{~:{(2,0):}}};
{"A2"\PATH~={**@{-}}'"A4"};
\endxy
\]

Indeed, suppose $c$ is an arc on a triangulated unpunctured surface. (We refer the reader to \cite{MSW1} for the case of a surface with punctures, which is similar.) Assume that this arc is not an edge of the triangulation. The illustration below shows an example of such a curve on a disk with ten marked points.
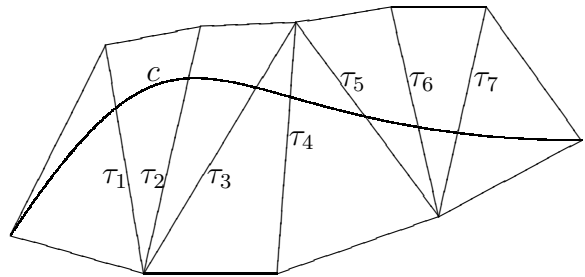
\[
\xy /l6pc/:
(3,0.5)*{}="1";
(2.5,-0.5)*{}="2";
(2,-0.6)*{}="3";
(1.5,-0.62)*{}="4";
(1,-0.7)*{}="5";
(0.5,-0.7)*{}="6";
(0,0)*{}="7";
(0.75,0.4)*{}="8";
(1.6,0.7)*{}="9";
(2.3,0.7)*{}="10";
{"1"\PATH~={**@{-}}'"2"},
{"2"\PATH~={**@{-}}'"3"},
{"3"\PATH~={**@{-}}'"4"},
{"4"\PATH~={**@{-}}'"5"},
{"5"\PATH~={**@{-}}'"6"},
{"6"\PATH~={**@{-}}'"7"},
{"7"\PATH~={**@{-}}'"8"},
{"8"\PATH~={**@{-}}'"9"},
{"9"\PATH~={**@{-}}'"10"},
{"10"\PATH~={**@{-}}'"1"},
{"10"\PATH~={**@{-}}'"2"},
{"10"\PATH~={**@{-}}'"3"},
{"10"\PATH~={**@{-}}'"4"},
{"9"\PATH~={**@{-}}'"4"},
{"8"\PATH~={**@{-}}'"4"},
{"8"\PATH~={**@{-}}'"5"},
{"8"\PATH~={**@{-}}'"6"},
"1";"7" **\crv{(2,-1) & (2,0)},
(2.25,-0.35)*{c};
(2.45,0.2)*{\tau_1};
(2.25,0.2)*{\tau_2};
(1.9,0.2)*{\tau_3};
(1.47,0)*{\tau_4};
(1.2,-0.31)*{\tau_5};
(0.84,-0.31)*{\tau_6};
(0.5,-0.31)*{\tau_7};
\endxy
\]

Choose an orientation for $c$, and label the arcs that $c$ crosses in order by $\tau_{i_1},\dots,\tau_{i_d}$. For any index~$j$, let $\Delta_{j-1}$ and $\Delta_j$ be the two triangles on either side of $\tau_{i_j}$. Then we can associate a tile $G_j$ as above to each $\tau_{i_j}$. It consists of two triangles with edges labeled as in~$\Delta_{j-1}$ and~$\Delta_j$ and glued together along the edge labeled $\tau_{i_j}$ so that the orientations of these triangles both agree or both disagree with those of~$\Delta_{j-1}$ and~$\Delta_j$. Note that there are two possible planar embeddings of the graph $G_j$.

The two arcs $\tau_{i_j}$ and $\tau_{i_{j+1}}$ are edges of the triangle $\Delta_j$. We will write $\tau_{[c_j]}$ for the third arc in this triangle. Then we can recursively glue together the tiles in order from~1 to~$d$ so that $G_{j+1}$ and $G_j$ are glued along the edges labeled $\tau_{[c_j]}$ and if the orientation of the triangles of~$G_j$ agrees with the orientation of~$\Delta_{j-1}$ and~$\Delta_j$ then the orientation of $G_{j+1}$ disagrees with the orientation of $\Delta_{j}$ and $\Delta_{j+1}$, and vice versa. We denote the resulting graph by $\bar{G}_{T,c}$.

For example, the graph $\bar{G}_{T,c}$ corresponding to the above example is
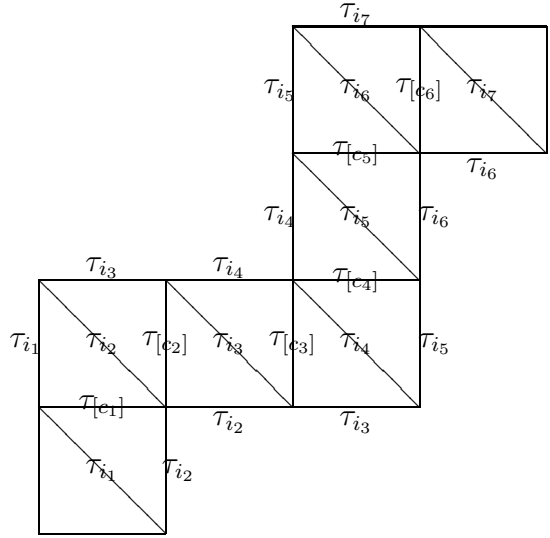
\[
\xy /l8pc/:
(1,-1)*{}="00";
(1.5,-1)*{}="10";
(2,-1)*{}="20";
(2.5,-1)*{}="30";
(3,-1)*{}="40";
(1,-0.5)*{}="01";
(1.5,-0.5)*{}="11";
(2,-0.5)*{}="21";
(2.5,-0.5)*{}="31";
(3,-0.5)*{}="41";
(1,0)*{}="02";
(1.5,0)*{}="12";
(2,0)*{}="22";
(2.5,0)*{}="32";
(3,0)*{}="42";
(1,0.5)*{}="03";
(1.5,0.5)*{}="13";
(2,0.5)*{}="23";
(2.5,0.5)*{}="33";
(3,0.5)*{}="43";
(2.5,1)*{}="34";
(3,1)*{}="44";
{"00"\PATH~={**@{-}}'"10"},
{"10"\PATH~={**@{-}}'"20"},
{"01"\PATH~={**@{-}}'"11"},
{"11"\PATH~={**@{-}}'"21"},
{"12"\PATH~={**@{-}}'"22"},
{"22"\PATH~={**@{-}}'"32"},
{"32"\PATH~={**@{-}}'"42"},
{"13"\PATH~={**@{-}}'"23"},
{"23"\PATH~={**@{-}}'"33"},
{"33"\PATH~={**@{-}}'"43"},
{"34"\PATH~={**@{-}}'"44"},
{"00"\PATH~={**@{-}}'"01"},
{"10"\PATH~={**@{-}}'"11"},
{"11"\PATH~={**@{-}}'"12"},
{"12"\PATH~={**@{-}}'"13"},
{"20"\PATH~={**@{-}}'"21"},
{"21"\PATH~={**@{-}}'"22"},
{"22"\PATH~={**@{-}}'"23"},
{"32"\PATH~={**@{-}}'"33"},
{"33"\PATH~={**@{-}}'"34"},
{"42"\PATH~={**@{-}}'"43"},
{"43"\PATH~={**@{-}}'"44"},
{"10"\PATH~={**@{-}}'"01"},
{"20"\PATH~={**@{-}}'"11"},
{"21"\PATH~={**@{-}}'"12"},
{"22"\PATH~={**@{-}}'"13"},
{"32"\PATH~={**@{-}}'"23"},
{"42"\PATH~={**@{-}}'"33"},
{"43"\PATH~={**@{-}}'"34"},
(3.05,0.25)*{\tau_{i_1}};
(2.75,-0.05)*{\tau_{i_3}};
(2.25,-0.05)*{\tau_{i_4}};
(2.05,-0.25)*{\tau_{i_4}};
(2.05,-0.75)*{\tau_{i_5}};
(1.75,-1.05)*{\tau_{i_7}};
(1.25,-0.44)*{\tau_{i_6}};
(1.44,-0.25)*{\tau_{i_6}};
(1.44,0.25)*{\tau_{i_5}};
(1.75,0.56)*{\tau_{i_3}};
(2.25,0.56)*{\tau_{i_2}};
(2.45,0.75)*{\tau_{i_2}};
(2.75,0.75)*{\tau_{i_1}};
(2.75,0.25)*{\tau_{i_2}};
(2.25,0.25)*{\tau_{i_3}};
(1.75,0.25)*{\tau_{i_4}};
(1.75,-0.25)*{\tau_{i_5}};
(1.75,-0.75)*{\tau_{i_6}};
(1.25,-0.75)*{\tau_{i_7}};
(2.75,0.5)*{\tau_{[c_1]}};
(2.5,0.25)*{\tau_{[c_2]}};
(2,0.25)*{\tau_{[c_3]}};
(1.75,0)*{\tau_{[c_4]}};
(1.75,-0.5)*{\tau_{[c_5]}};
(1.5,-0.75)*{\tau_{[c_6]}};
\endxy
\]

Write $G_{T,c}$ for the graph obtained from $\bar{G}_{T,c}$ by removing the diagonal in every tile. Recall that for any graph $G$, a \emph{perfect matching} of $G$ is a collection~$P$ of edges such that every vertex of $G$ is incident to exactly one edge in $P$. It is easy to show that the graph~$G_{T,c}$ constructed in~\cite{MSW1} has exactly two perfect matchings consisting only of boundary edges. These perfect matchings are called the \emph{minimal matching} and \emph{maximal matching} and are denoted $P_{-}=P_{-}(G_{T,c})$ and $P_{+}=P_{+}(G_{T,c})$, respectively. In the above example, the maximal matching $P_{+}$ is the matching that contains the horizontal edge at the bottom of the graph~$\bar{G}_{T,c}$.

If the edges of a perfect matching $P$ are labeled $\tau_{j_1},\dots,\tau_{j_r}$, then we define the \emph{weight} $x(P)$ of $P$ as the product 
\[
x(P)=\prod_{s=1}^r x_{\tau_{j_s}}
\]
of the cluster variables associated to $\tau_{j_1},\dots,\tau_{j_r}$. Similarly, if $\tau_{i_1},\dots,\tau_{i_d}$ is the sequence of arcs in $T$ that $c$ crosses, then we define the \emph{crossing monomial} $\cross(T,c)$ of $c$ with respect to $T$ as the product 
\[
\cross(T,c)=\prod_{s=1}^d x_{\tau_{i_s}}.
\]
Note that the arcs $\tau_{i_1},\dots,\tau_{i_d}$ in this definition also appear as the labels on the diagonal edges in the graph $\bar{G}_{T,c}$.

\begin{proposition}[\cite{MSW1}]
Let $c$ be an arc on a decorated surface $S$. Then the $\mathbf{g}$-vector associated to $c$ is given by the formula 
\[
\mathbf{g}_{c}=\deg\left(\frac{x(P_{-})}{\cross(T,c)}\right).
\]
\end{proposition}

For an alternative approach to computing $\mathbf{g}$-vectors, see~\cite{R}, Proposition~5.2.

\section{The canonical pairing}
\label{sec:TheCanonicalPairing}

\subsection{The multiplicative canonical pairing}

A simple closed curve $l$ on $S_\mathcal{D}$ will be called an \emph{intersecting curve} if every curve in its homotopy class intersects the image of $\partial S$ in $S_\mathcal{D}$.

Suppose we are given an intersecting curve $l$ on $S_\mathcal{D}$. Deform this curve so that it intersects the image of~$\partial S$ in the minimal number of points. Then, starting from any component of this image which we may call $\gamma_1$, there is a segment $c_1$ of the curve $l$ which lies entirely in~$S$ and connects the component~$\gamma_1$ to another component which we may call $\gamma_2$. Starting from this component, there is a segment~$c_2$ of~$l$ which lies entirely in~$S^\circ$ and connects~$\gamma_2$ to a component~$\gamma_3$. Continue labeling in this way until the curve closes.

Now suppose we are given a point $m\in\mathcal{D}^+(S)$. This determines a point in the Teichm\"uller space of some surface $\Sigma_{p_1,\dots,p_k}$ where $\Sigma=S_\mathcal{D}$, and we can represent this by a hyperbolic structure such that any curve in the image of~$\partial S\subseteq S_\mathcal{D}$ is geodesic. Cut the surface $\Sigma_{p_1,\dots,p_k}$ along $\partial S$. Then each $c_i$ is a curve that connects boundary components or punctures of~$S$ or~$S^\circ$. We can wind the ends of $c_i$ around the holes infinitely many times in the direction prescribed by the orientations from $m$.

Consider a curve $c_i$ for $i$ odd. Lifting this curve to the upper half plane, we obtain a geodesic connecting two points on the boundary of $\mathbb{H}$. Choose a horocycle around each of these boundary points, and define $A_{c_i}$ as the exponentiated signed half length of the portion of the lifted curve between these horocycles. Next consider $c_i$ for $i$ even. Lifting to the upper half plane, we again get a geodesic connecting two points on the boundary of $\mathbb{H}$, and the horocycles already chosen determine a pair of horocycles around these points. We define~$A_{c_i}^\circ$ as the exponentiated signed half length of the portion of the lifted curve between these horocycles.

\Needspace*{5\baselineskip}
\begin{definition} \mbox{}
Fix a point $m\in\mathcal{D}^+(S)$.
\begin{enumerate}
\item Let $l$ be an intersecting curve of weight $k$ on $S_\mathcal{D}$, and assume that the orientation of each component of $\gamma$ agrees with the orientation of~$S$. Then 
\[
\mathbb{I}_\mathcal{D}(l,m)=\left(\frac{\prod_{\text{$i$ even}}A_{c_i}^\circ}{\prod_{\text{$i$ odd}}A_{c_i}}\right)^k.
\]

\item Let $l$ be a curve of weight $k$ on $S_\mathcal{D}$ which is not an intersecting curve and is not homotopic to a loop in the simple lamination. If $l$ lies in $S^\circ$, then $\mathbb{I}_\mathcal{D}(l,m)$ is defined as the absolute value of the trace of the $k$th power of the of the monodromy around $l$. If $l$ lies in $S$ then $\mathbb{I}_\mathcal{D}(l,m)$ is the reciprocal of this quantity.

\item Let $l$ be a curve of weight $k$ on $S_\mathcal{D}$ which is homotopic to a loop $\gamma_i$ in the simple lamination. If $l$ and $m$ provide the same orientation for $\gamma_i$, then $\mathbb{I}_\mathcal{D}(l,m)$ is defined as the absolute value of the $k$th power of the largest eigenvalue of the monodromy around~$l$. If $l$ and $m$ provide different orientations for $\gamma_i$, then $\mathbb{I}_\mathcal{D}(l,m)$ is the reciprocal of this quantity.

\item Let $l_1$ and $l_2$ be laminations on $S_\mathcal{D}$ such that no curve from $l_1$ intersects or is homotopic to a curve from $l_2$. Then $\mathbb{I}_\mathcal{D}(l_1+l_2,m)=\mathbb{I}_\mathcal{D}(l_1,m)\mathbb{I}_\mathcal{D}(l_2,m)$.
\end{enumerate}
\end{definition}

This defines the canonical map in the special case where the orientation of any~$\gamma_i$ that meets a curve agrees with the orientation of the surface~$S$. If the orientation of $\gamma_i$ disagrees with the orientation of~$S$, then we can define $\mathbb{I}_{\mathcal{D}}(l)$ by modifying slightly the above definition. To do this, suppose that $l$ is a point in $\mathcal{D}_L(S,\mathbb{Z})$ for which the function $\mathbb{I}_{\mathcal{D}}(l)$ has been defined. Suppose the orientation that $l$ provides for $\gamma_i$ agrees with the orientation of~$S$, and let $l'$ be the lamination obtained from $l$ by reversing the orientation of~$\gamma_i$. We will define $\mathbb{I}_{\mathcal{D}}(l')(m)$ by considering two possibilities for $m\in\mathcal{D}^+(S)$. First, let us assume that $m$ is chosen so that the monodromy around~$\gamma_i$ is hyperbolic. In this case, we define $m'$ to be the point of $\mathcal{D}^+(S)$ obtained from~$m$ be reversing the orientation of~$\gamma_i$, and we put 
\[
\mathbb{I}_{\mathcal{D}}(l')(m)=\mathbb{I}_{\mathcal{D}}(l)(m').
\]
For the second possibility, assume $m$ is chosen so that the monodromy around~$\gamma_i$ is parabolic. In this case, we note that there is a natural invariant associated to $\gamma_i$. Indeed, fix an ideal triangulation $T$ of~$S$, and lift this to an ideal triangulation~$\tilde{T}$ of the universal cover~$\tilde{S}$. Consider the triangles $t_1,\dots,t_N$ of $T$ that meet the puncture in~$S$ corresponding to $\gamma_i$. We can find an ideal polygon in $\tilde{S}$, formed by $N$ triangles of $\tilde{T}$, which projects onto the union of the triangles $t_1,\dots,t_N$. We will label the edges of this ideal polygon as in the diagram below.
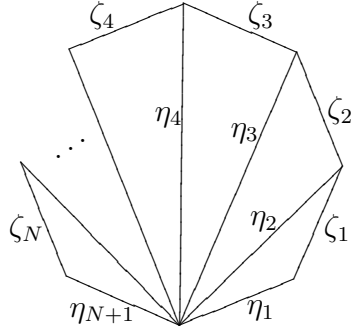
\[
\xy /l2pc/:
{\xypolygon8"A"{~:{(1,2.33):}~>{}}};
{"A1"\PATH~={**@{-}}'"A2"};
{"A2"\PATH~={**@{-}}'"A3"};
{"A3"\PATH~={**@{-}}'"A4"};
{"A4"\PATH~={**@{-}}'"A5"};
{"A5"\PATH~={**@{-}}'"A6"};
{"A7"\PATH~={**@{-}}'"A8"};
{"A8"\PATH~={**@{-}}'"A1"};
{"A1"\PATH~={**@{-}}'"A3"};
{"A1"\PATH~={**@{-}}'"A4"};
{"A1"\PATH~={**@{-}}'"A5"};
{"A1"\PATH~={**@{-}}'"A6"};
{"A1"\PATH~={**@{-}}'"A7"};
(2.75,-0.25)*{\Ddots};
(-0.25,2.3)*{\eta_1};
(-0.3,0.9)*{\eta_2};
(0,-0.5)*{\eta_3};
(1.2,-0.75)*{\eta_4};
(2.25,2.35)*{\eta_{N+1}};
(-1.45,1)*{\zeta_1};
(-1.5,-0.85)*{\zeta_2};
(-0.25,-2.35)*{\zeta_3};
(2.25,-2.35)*{\zeta_4};
(3.45,1)*{\zeta_N};
\endxy
\]
Choose a horocycle at each vertex of this polygon. There is a similar polygon in the universal cover $\tilde{S}^\circ$ of $S^\circ$, and the horocycles already chosen determine a corresponding collection of horocycles associated to the vertices of this polygon. We denote the corresponding edges of the triangulation of this new polygon by the same symbols $\eta_1,\dots,\eta_{N+1}$ and $\zeta_1,\dots,\zeta_N$, and we write $A_j^\circ$ for the invariant associated to the edge $j$. We can then form the expression 
\[
\kappa_i=\frac{\alpha^\circ}{\alpha}
\]
where 
\[
\alpha^\circ=\sum_{j=1}^N\frac{A_{\zeta_j}^\circ}{A_{\eta_j}^\circ A_{\eta_{j+1}}^\circ}
\]
and 
\[
\alpha=\sum_{j=1}^N\frac{A_{\zeta_j}}{A_{\eta_j} A_{\eta_{j+1}}}.
\]
Then we define $\mathbb{I}_{\mathcal{D}}(l')$ as $\mathbb{I}_{\mathcal{D}}(l)$ multiplied by a factor of $\kappa_i$ for each curve of weight~1 of~$l$ that meets $\gamma_i$.

We will see below that the functions $\mathbb{I}_{\mathcal{D}}(l)$ defined in this way are independent of all the choices made in the construction and are given by algebraic expressions in the $B_j$ and~$X_j$ for any choice of triangulation.

\subsection{Expression in terms of $F$-polynomials}

\subsubsection{Intersecting curves}

We will prove our formula for $\mathbb{I}_{\mathcal{D}}(l)$ in several steps. We begin by examining the special case where $l$ is an intersecting curve. In this case, we will find that the formula involves the $F$-polynomials of Fomin and Zelevinsky.

Fix an ideal triangulation $T$ of $S$, and a point $m\in\mathcal{D}^+(S)$. Then the universal cover of~$S$ can be identified with a subset of the hyperbolic plane $\mathbb{H}$, and the triangulation $T$ can be lifted to a triangulation $\tilde{T}$ of the universal cover. Using the natural map $S^\circ\rightarrow S$, we can draw all of the curves $c_i$ on the surface $S$. We can then lift each curve to a geodesic $\tilde{c}_i$ in the universal cover in such a way that $\tilde{c}_i$ and $\tilde{c}_{i+1}$ share a common endpoint. Let $P$ be a triangulated ideal polygon, formed from triangles in $\tilde{T}$, which includes all of the triangles that the curves~$\tilde{c}_i$ pass through. Let $T(P)$ be the triangulation of $P$ provided by the triangulation~$\tilde{T}$. Choose a horocycle around each vertex of $T(P)$. Then we can define $A_i$ as the exponentiated signed half length of an edge $i$ of an ideal triangle between the chosen horocycles.

In exactly the same way, the universal cover of $S^\circ$ can be identified with a subset of the hyperbolic plane, and the triangulation $T$ provides a triangulation $\tilde{T}^\circ$ of this universal cover. The polygon $P$ gives rise to a polygon $P^\circ$ in this universal cover. The latter polygon has a triangulation $T(P^\circ)$ provided by the triangulation~$\tilde{T}^\circ$. The horocycles that we already chose around the vertices of $T(P)$ provide horocycles around the vertices of $T(P^\circ)$. We can define~$A_i^\circ$ as the exponentiated signed half length of any edge $i$ of an ideal triangle between these horocycles.

It is convenient at this point to adopt a kind of multi-index notation. If $\mathbf{v}=(v_i)$ is a vector indexed by the edges of the triangulation $T(P)$, then we will write 
\[
A^{\mathbf{v}}=\prod_i A_i^{v_i}.
\]
Similarly, if $\mathbf{v}=(v_i)$ is indexed by the edges of $T(P^\circ)$, then we will write 
\[
(A^\circ)^{\mathbf{v}}=\prod_i (A_i^\circ)^{v_i}.
\]
Since there is a natural bijection between the edges of $T(P^\circ)$ and the edges of~$T(P)$, the indexing sets for these vectors can be identified.

\begin{lemma}
\label{lem:Adecomposition}
For each $i$, there is an $F$-polynomial $F_{c_i}$ and an integral vector $\mathbf{g}_{c_i}$, indexed by the edges of the triangulation $T(P)$, such that 
\[
A_{c_i} = F_{c_i}(X_1,\dots,X_n)A^{\mathbf{g}_{c_i}}
\]
and 
\[
A_{c_i}^\circ = F_{c_i}(\widehat{X}_1,\dots,\widehat{X}_n)(A^\circ)^{\mathbf{g}_{c_i}}.
\]
\end{lemma}

\begin{proof}
Let $i_1,\dots,i_m$ be the edges of $T(P)$. Associated to the polygon $P$, there is a cluster algebra generated by the variables $A_i$ over the trivial semifield $\mathbb{P}=\{1\}$. Applying Proposition~\ref{prop:FZ63} to this cluster algebra, we see that
\[
A_{c_i} = F_{c_i}\biggr(\prod_{j}A_j^{\varepsilon_{i_1j}},\dots,\prod_{j}A_j^{\varepsilon_{i_mj}}\biggr)A_{i_1}^{g_{i_1}}\dots A_{i_m}^{g_{i_m}}
\]
for some integers $g_{i_1},\dots,g_{i_m}$. Now for any internal edge $k$ of $T(P)$, the product $\prod_{j}A_j^{\varepsilon_{kj}}$ equals the coordinate $X_k$ associated to the corresponding edge $k$ of $T$. Moreover, by the matrix formula of~\cite{MW}, this $F_{c_i}$ is a polynomial only in the variables associated to edges that $\tilde{c}_i$ crosses, which are all internal edges. Therefore we can write it as $F_{c_i}(X_1,\dots,X_n)$, a~polynomial in the $X$-coordinates. This proves the first equation. The proof of the second equation is similar. In this case, one uses the fact that $\widehat{X}_k=\prod_{j}(A_j^\circ)^{\varepsilon_{kj}}$.
\end{proof}

Recall that the variables $X_k$ are defined as cross ratios $X_k=\frac{A_iA_m}{A_jA_l}$ where $i$, $j$, $l$, $m$ are the edges of the quadrilateral with diagonal $k$. Here we will consider additional variables associated to the edges $i$ of $\tilde{T}$. Consider a triangle $\Delta$ in $\tilde{T}$ that includes $i$ as one of its edges, and label the other edges of this triangle as follows:
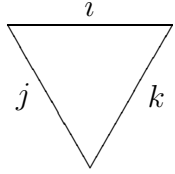
\[
\xy /l1.5pc/:
{\xypolygon3"A"{~:{(2,0):}}};
(2.4,0.5)*{j};
(-0.4,0.5)*{k}; 
(1,-1.4)*{i};
\endxy
\]
Then we define the variable associated to $i$ by $W_{i,\Delta}=\frac{A_iA_j}{A_k}$. Note that this depends on the chosen triangle as well as the edge $i$.

Fix an edge $i$ of the triangulation $T(P)$. For any vertex $v$ of $i$, there is a collection of edges in $\tilde{T}$ that start at $v$ and lie in the counterclockwise direction from $i$. Consider a curve~$i'$ that goes diagonally across $i$, intersecting finitely many of these edges transversely before terminating on one of them. An example is illustrated below.
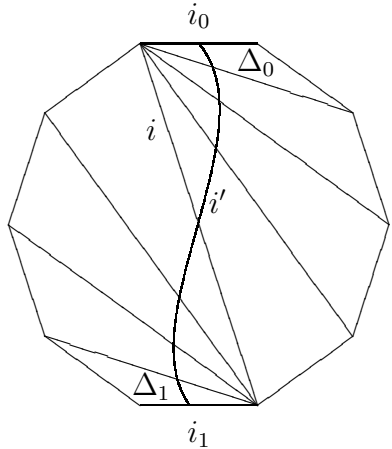
\[
\xy /l3pc/:
{\xypolygon10"A"{~:{(2,0):}}};
{"A9"\PATH~={**@{-}}'"A4"},
{"A9"\PATH~={**@{-}}'"A5"},
{"A9"\PATH~={**@{-}}'"A6"},
{"A9"\PATH~={**@{-}}'"A7"},
{"A10"\PATH~={**@{-}}'"A4"},
{"A1"\PATH~={**@{-}}'"A4"},
{"A2"\PATH~={**@{-}}'"A4"},
(1,-1.9)*{}="1";
(1.1,1.9)*{}="2";
"1";"2" **\crv{(0.25,-1) & (1.75,1)};
(1.5,-1)*{i};
(0.8,-0.25)*{i'};
(1,-2.2)*{i_0};
(1,2.2)*{i_1};
(0.4,-1.7)*{\Delta_0};
(1.5,1.7)*{\Delta_1};
\endxy
\]

Given such a curve $i'$, let $E_i$ be the set of all edges in $T(P)$ that $i'$ crosses. Then we can form the product 
\[
P_i=W_{i_0,\Delta_0}\cdot\prod_{j\in E_i} X_j\cdot W_{i_1,\Delta_1}
\]
where $i_0$ and $i_1$ are the edges on which $i'$ terminates. One can check that this expression equals $A_i^2$. We will use this fact to prove the following result.

\begin{lemma}
\label{lem:factorization}
Let $\mathbf{s}=\sum_{i\text{ even}}\mathbf{g}_{c_i}-\sum_{i\text{ odd}}\mathbf{g}_{c_i}$. Then there exists a half integral vector $\mathbf{h}=(h_i)_{i\in J}$, indexed by the internal edges of the triangulation $T$, such that 
\[
A^{\mathbf{s}}=X^{\mathbf{h}}=\prod_{i\in J}X_i^{h_i}.
\]
\end{lemma}

\begin{proof}
Consider an arc $\tilde{c}_i$ in $T(P)$. If this arc coincides with an edge $i$ of the triangulation, then the associated $\mathbf{g}$-vector equals the standard basis vector $\mathbf{e}_i$. Otherwise $\tilde{c}_i$ intersects one or more edges of $T(P)$, and we can compute the corresponding $\mathbf{g}$-vector $\mathbf{g}_{c_i}$ using the results of \cite{MSW1} that we reviewed above. In this case, $\mathbf{g}_{c_i}$ is given by the formula 
\[
\mathbf{g}_{c_i}=\deg\left(\frac{x(P_{-})}{\cross(T,c_i)}\right).
\]
By definition of a perfect matching, we know that any endpoint of a diagonal in the graph~$\bar{G}_{T,c_i}$ meets exactly one edge of the minimal perfect matching $P_{-}$. It follows that the vector $\mathbf{g}_{c_i}$ is an alternating sum of standard basis vectors corresponding to the edges of a path in the graph $\bar{G}_{T,c_i}$.

Consider the path formed by the $\tilde{c}_i$ in $P$. An example of such a path is illustrated below.
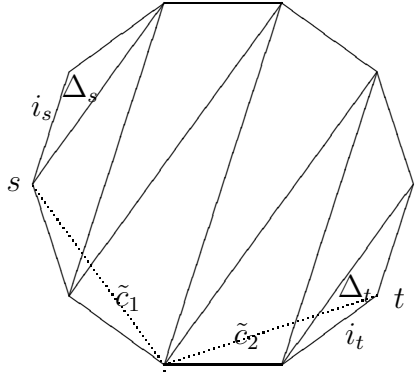
\[
\xy /l3pc/:
{\xypolygon10"A"{~:{(2,0):}}};
{"A7"\PATH~={**@{-}}'"A3"},
{"A7"\PATH~={**@{-}}'"A4"},
{"A6"\PATH~={**@{-}}'"A4"},
{"A8"\PATH~={**@{-}}'"A3"},
{"A8"\PATH~={**@{-}}'"A2"},
{"A9"\PATH~={**@{-}}'"A2"},
{"A9"\PATH~={**@{-}}'"A1"},
{"A1"\PATH~={**@{.}}'"A3"},
{"A3"\PATH~={**@{.}}'"A5"},
(3.2,0)*{s};
(-0.85,1.2)*{t};
(2,1.2)*{\tilde{c}_1};
(0.75,1.6)*{\tilde{c}_2};
(-0.39,1.1)*{\Delta_t};
(2.5,-1)*{\Delta_s};
(-0.39,1.6)*{i_t};
(2.9,-0.8)*{i_s};
\endxy
\]
Let $s$ and $t$ be the endpoints of the path formed by the $\tilde{c}_i$. Consider the closed path on the surface $S$ obtained by drawing all the $c_i$ on $S$ using the natural map $S^\circ\rightarrow S$. The horocycle at $t$ is obtained from the one at $s$ by applying the monodromy around this path. Let $\Delta_t$ be the triangle in $P$ that contains $t$ and is that last triangle that the lifted arcs $c_i$ pass through. Let $\Delta_s$ be the preimage of this triangle under the monodromy.

By the above discussion, the vector $\mathbf{s}$ is an alternating sum of standard basis vectors associated to the edges of a path in $T(P)$. To each edge $i$ on this path, we associate a curve~$i'$ as above so that $i'$ and $j'$ terminate on a common edge whenever $i$ and $j$ terminate on a common vertex. We can choose these curves so that the first and last ones terminate at corresponding edges~$i_s$ of~$\Delta_s$ and~$i_t$ of~$\Delta_t$. One can show in this case that $W_{i_s,\Delta_s}=W_{i_t,\Delta_t}$. It follows from the above discussion that 
\[
\prod_iP_i^{s_i}=\biggr(\prod_iA_i^{s_i}\biggr)^2
\]
where $\mathbf{s}=(s_i)$. Since $\mathbf{s}$ is an alternating sum, all $W$-factors cancel on the left hand side of the equation. Therefore the left hand side is a product of the $X_i$, and the lemma follows by taking square roots on both sides.
\end{proof}

We can now prove our formula for $\mathbb{I}_{\mathcal{D}}(l)$ in the case where $l$ is an intersecting curve.

\begin{proposition}
\label{prop:intersecting}
Let $l$ be a $\mathcal{D}$-lamination represented by a single intersecting curve of weight~1, and suppose that the orientation of each component of~$\gamma$ agrees with the orientation of~$S$. Then 
\[
\mathbb{I}_\mathcal{D}(l)=\frac{\prod_{\text{$i$ even}}F_{c_i}(\widehat{X}_1,\dots,\widehat{X}_n)}{\prod_{\text{$i$ odd}}F_{c_i}(X_1,\dots,X_n)}B_1^{g_{l,1}}\dots B_n^{g_{l,n}}X_1^{h_{l,1}}\dots X_n^{h_{l,n}}
\]
where the $g_{l,i}$ are integers and the $h_{l,i}$ are half integers.
\end{proposition}

\begin{proof}
By Lemma~\ref{lem:Adecomposition}, we know that
\[
A_{c_i} = F_{c_i}(X_1,\dots,X_n)A^{\mathbf{g}_{c_i}}
\]
and
\[
A_{c_i}^\circ = F_{c_i}(\widehat{X}_1,\dots,\widehat{X}_n)(A^\circ)^{\mathbf{g}_{c_i}}.
\]
Inserting these expressions into the formula in the definition of $\mathbb{I}_{\mathcal{D}}$, we obtain 
\[
\mathbb{I}_\mathcal{D}(l) = \frac{\prod_{\text{$i$ even}}F_{c_i}(\widehat{X}_1,\dots,\widehat{X}_n)}{\prod_{\text{$i$ odd}}F_{c_i}(X_1,\dots,X_n)}\frac{(A^\circ)^{\mathbf{g}_{\text{even}}}}{A^{\mathbf{g}_{\text{odd}}}}
\]
where we have defined $\mathbf{g}_{\text{even}}=\sum_{i\text{ even}}\mathbf{g}_{c_i}$ and $\mathbf{g}_{\text{odd}}=\sum_{i\text{ odd}}\mathbf{g}_{c_i}$. Substituting $A_i^\circ=B_iA_i$ into this expression, we obtain 
\[
\mathbb{I}_\mathcal{D}(l) =\frac{\prod_{\text{$i$ even}}F_{c_i}(\widehat{X}_1,\dots,\widehat{X}_n)}{\prod_{\text{$i$ odd}}F_{c_i}(X_1,\dots,X_n)}B^{\mathbf{g}_{\text{even}}}A^{\mathbf{g}_{\text{even}}-\mathbf{g}_{\text{odd}}}.
\]
Finally, by Lemma~\ref{lem:factorization}, we can write the last factor as
\[
A^{\mathbf{g}_{\text{even}}-\mathbf{g}_{\text{odd}}}=X^{\mathbf{h}}
\]
for some half integral vector $\mathbf{h}$. This completes the proof.
\end{proof}

Notice that the proof of Lemma~\ref{lem:factorization} actually gives an explicit description of the product $X_1^{h_{l,1}}\dots X_n^{h_{l,n}}$ appearing in this theorem. In proving the lemma, we have essentially described a cycle $\eta_l$ with $\mathbb{Z}/2\mathbb{Z}$-coefficients on~$S$ such that this monomial equals $X_{i_1}^{\pm1/2}\dots X_{i_s}^{\pm1/2}$ where $i_1,\dots,i_s$ are the edges of $T$ that $\eta_l$ intersects. This fact will be important below when we discuss rational functions obtained from laminations.

\subsubsection{Other curves}

We will now derive formulas for $\mathbb{I}_{\mathcal{D}}(c)$ in special cases where $c$ is a closed curve on $S_{\mathcal{D}}$ that is not an intersecting curve. The following simple lemma will be used repeatedly for calculations.

\Needspace*{2\baselineskip}
\begin{lemma}
\label{lem:leftrightturns} \mbox{}
Let $S$ be a decorated surface, $T$ an ideal triangulation of $S$, and $m\in\mathcal{X}^+(S)$.
\begin{enumerate}
\item If $c$ is a loop on $S$, then the monodromy around $c$ is represented by a product of the matrices 
\[
\left( \begin{array}{cc}
X_i^{1/2} & 0 \\
X_i^{-1/2} & X_i^{-1/2} \end{array} \right)
\quad
\text{and}
\quad
\left( \begin{array}{cc}
X_i^{1/2} & X_i^{1/2} \\
0 & X_i^{-1/2} \end{array} \right),
\]
one for each edge $i$ of $T$ that the loop intersects.

\item If $c$ is a loop surrounding a hole and the orientation of this hole agrees with the orientation of $S$, then the smallest eigenvalue of the monodromy around $c$ is $X_{i_1}^{1/2}\dots X_{i_s}^{1/2}$ where $i_1,\dots,i_s$ are the edges of the triangulation that spiral into this hole. If the orientation of the hole disagrees with the orientation of $S$, then $X_{i_1}^{-1/2}\dots X_{i_s}^{-1/2}$ is the smallest eigenvalue of the monodromy.
\end{enumerate}
\end{lemma}

\begin{proof}
1. This monodromy can be computed using the method described in Subsection~\ref{subsec:LaurentPolynomialsFromLaminations}. We first retract the curve $c$ to the graph described there. If the resulting path turns to the left before crossing the edge $i$, then the monodromy includes a factor 
\[
\left( \begin{array}{cc}
0 & X_i^{1/2} \\
-X_i^{-1/2} & 0 \end{array} \right)
\left( \begin{array}{cc}
1 & 1 \\
-1 & 0 \end{array} \right) =
\left( \begin{array}{cc}
X_i^{1/2} & 0 \\
X_i^{-1/2} & X_i^{-1/2} \end{array} \right) 
\in PSL(2,\mathbb{R}).
\]
If it turns to the right before crossing $i$, then the monodromy includes a factor 
\[
\left( \begin{array}{cc}
0 & X_i^{1/2} \\
-X_i^{-1/2} & 0 \end{array} \right)
{\left( \begin{array}{cc}
1 & 1 \\
-1 & 0 \end{array} \right)}^{-1} =
\left( \begin{array}{cc}
X_i^{1/2} & X_i^{1/2} \\
0 & X_i^{-1/2} \end{array} \right) 
\in PSL(2,\mathbb{R}).
\]
The entire path can be decomposed into consecutive left and right turns.

2. By part~1, the monodromy around~$c$ is given by 
\[
\prod_{k=1}^s\left( \begin{array}{cc}
X_{i_k}^{1/2} & 0 \\
X_{i_k}^{-1/2} & X_{i_k}^{-1/2} \end{array} \right)
=\left( \begin{array}{cc}
X_{i_1}^{1/2}\dots X_{i_s}^{1/2} & 0 \\
C & X_{i_1}^{-1/2}\dots X_{i_s}^{-1/2} \end{array} \right)
\]
where $C$ is a polynomial in the variables $X_{i_k}^{\pm1/2}$. The eigenvalues of this matrix are the diagonal elements $X_{i_1}^{1/2}\dots X_{i_s}^{1/2}$ and $X_{i_1}^{-1/2}\dots X_{i_s}^{-1/2}$. If the orientation of the hole agrees with the orientation of~$S$, then $X_{i_1}\dots X_{i_s}<1$, and hence the first of these products is smaller. If the orientation of the hole disagrees with the orientation of~$S$, then $X_{i_1}\dots X_{i_s}>1$, and hence the second product is smaller.
\end{proof}

\begin{proposition}
\label{prop:Scircloop}
Let $c$ be a $\mathcal{D}$-lamination consisting of a single loop of weight~$k$ that lies entirely in~$S^\circ$ and is not retractible to the simple lamination. Then 
\[
\mathbb{I}_{\mathcal{D}}(c)=F_c(\widehat{X}_1,\dots,\widehat{X}_n)B_1^{g_{c,1}}\dots B_n^{g_{c,n}}X_1^{h_{c,1}}\dots X_n^{h_{c,n}}
\]
where $F_c$ is a polynomial, the $g_{c,i}$ are integers, and the $h_{c,i}$ are half integers.
\end{proposition}

\begin{proof}
Let $c$ be a $\mathcal{D}$-lamination satisfying the hypotheses of the proposition. In this case, $\mathbb{I}_{\mathcal{D}}(c)$ is defined as the absolute value of the trace of the $k$th power of the monodromy around~$c$. By Lemma~\ref{lem:leftrightturns}, this monodromy is a product of matrices 
\[
\left( \begin{array}{cc}
\widehat{X}_i^{1/2} & 0 \\
\widehat{X}_i^{-1/2} & \widehat{X}_i^{-1/2} \end{array} \right)
\quad
\text{and}
\quad
\left( \begin{array}{cc}
\widehat{X}_i^{1/2} & \widehat{X}_i^{1/2} \\
0 & \widehat{X}_i^{-1/2} \end{array} \right),
\]
one for each edge $i$ that $c$ intersects. These matrices factor as 
\[
\left( \begin{array}{cc}
\widehat{X}_i & 0 \\
1 & 1 \end{array} \right)
\cdot \widehat{X}_i^{-1/2}
\quad
\text{and}
\quad
\left( \begin{array}{cc}
\widehat{X}_i & \widehat{X}_i \\
0 & 1 \end{array} \right)
\cdot \widehat{X}_i^{-1/2},
\]
so the monodromy factors as $M\cdot\widehat{X}_{i_1}^{-1/2}\dots\widehat{X}_{i_s}^{-1/2}$ where $i_1,\dots,i_s$ are the edges that $c$ intersects and $M$ is a matrix with polynomial entries. Write $F_c(\widehat{X}_1,\dots,\widehat{X}_n)$ for the trace of the matrix $M^k$. Then we have 
\[
\mathbb{I}_{\mathcal{D}}(c)=F_c(\widehat{X}_1,\dots,\widehat{X}_n)\widehat{X}_{i_1}^{-k/2}\dots\widehat{X}_{i_s}^{-k/2}.
\]
Consider the product $\prod_l\prod_j B_j^{\varepsilon_{i_lj}}$. Let~$i$ be any edge of the triangulation. Using the definition of the exchange matrix given in Section~\ref{sec:BackgroundOnClusterVarieties}, it is easy to show that the total degree of $B_i$ in this product is even. Hence
\begin{align*}
\widehat{X}_{i_1}^{-k/2}\dots\widehat{X}_{i_s}^{-k/2} &= \prod_l {\biggr(X_{i_l}\prod_jB_j^{\varepsilon_{i_lj}}\biggr)}^{-k/2} \\
&= {\biggr(\prod_{j,l} B_j^{\varepsilon_{i_lj}}\biggr)}^{-k/2}\biggr(\prod_lX_{i_l}^{-k/2}\biggr) \\
&= B_1^{g_{c,1}}\dots B_n^{g_{c,n}}X_1^{h_{c,1}}\dots X_n^{h_{c,n}}
\end{align*}
for some integers $g_{c,i}$ and half integers $h_{c,i}$. This completes the proof.
\end{proof}

\begin{proposition}
\label{prop:Scirchole}
Let $c$ be a $\mathcal{D}$-lamination consisting of a single loop of weight~$k$ on~$S_{\mathcal{D}}$. Assume that this loop is homotopic to a loop in the simple lamination and that $c$ and $m$ provide the same orientation for this loop. Then 
\[
\mathbb{I}_{\mathcal{D}}(c)=B_1^{g_{c,1}}\dots B_n^{g_{c,n}}X_1^{h_{c,1}}\dots X_n^{h_{c,n}}
\]
for integers $g_{c,i}$ and half integers $h_{c,i}$.
\end{proposition}

\begin{proof}
Let $c$ be a $\mathcal{D}$-lamination satisfying the hypotheses of the proposition. In this case, $\mathbb{I}_{\mathcal{D}}(c)$ is defined as the absolute value of the $k$th power of the largest eigenvalue of the monodromy around~$c$. Let $i_1,\dots,i_s$ be the edges of the triangulation that spiral into this loop in the simple lamination. If the orientation provided by $c$ agrees with the orientation of~$S^\circ$, then the orientation provided by $m$ agrees with that of~$S^\circ$, and Lemma~\ref{lem:leftrightturns} implies 
\[
\mathbb{I}_{\mathcal{D}}(c)=\widehat{X}_{i_1}^{-1/2}\dots\widehat{X}_{i_s}^{-1/2}.
\]
On the other hand, if the orientation provided by $c$ agrees with the orientation of~$S$, then the orientation provided by $m$ disagrees with that of~$S^\circ$, and Lemma~\ref{lem:leftrightturns} implies 
\[
\mathbb{I}_{\mathcal{D}}(c)=\widehat{X}_{i_1}^{1/2}\dots\widehat{X}_{i_s}^{1/2}.
\]
The same argument that we used in the proof of Proposition~\ref{prop:Scircloop} can be used to show that the $k$th powers of these expressions have the form $B_1^{g_{c,1}}\dots B_n^{g_{c,n}}X_1^{h_{c,1}}\dots X_n^{h_{c,n}}$ for some integers~$g_{c,i}$ and half integers~$h_{c,i}$.
\end{proof}

\begin{proposition}
\label{prop:Sloop}
Let $c$ be a $\mathcal{D}$-lamination consisting of a single loop of weight~$k$ that lies entirely in~$S$ and is not retractible to the simple lamination. Then 
\[
\mathbb{I}_{\mathcal{D}}(c)=\frac{1}{F_c(X_1,\dots,X_n)}X_1^{h_{c,1}}\dots X_n^{h_{c,n}}
\]
where $F_c$ is a polynomial and the $h_{c,i}$ are half integers.
\end{proposition}

\begin{proof}
Let $c$ be a $\mathcal{D}$-lamination satisfying the hypotheses of the proposition. In this case, $\mathbb{I}_{\mathcal{D}}(c)$ is defined as the reciprocal of the absolute value of the trace of the $k$th power of the monodromy around~$c$. By Lemma~\ref{lem:leftrightturns}, this monodromy is an element of $PSL(2,\mathbb{R})$ represented by a product of the matrices 
\[
\left( \begin{array}{cc}
X_i^{1/2} & 0 \\
X_i^{-1/2} & X_i^{-1/2} \end{array} \right)
\quad
\text{and}
\quad
\left( \begin{array}{cc}
X_i^{1/2} & X_i^{1/2} \\
0 & X_i^{-1/2} \end{array} \right),
\]
one for each edge $i$ that $c$ intersects. These matrices factor as 
\[
\left( \begin{array}{cc}
X_i & 0 \\
1 & 1 \end{array} \right)
\cdot X_i^{-1/2}
\quad
\text{and}
\quad
\left( \begin{array}{cc}
X_i & X_i \\
0 & 1 \end{array} \right)
\cdot X_i^{-1/2},
\]
so the monodromy factors as $M\cdot X_{i_1}^{-1/2}\dots X_{i_s}^{-1/2}$ where $i_1,\dots,i_s$ are the edges that $c$ intersects and $M$ is a matrix with polynomial entries. Write $F_c(X_1,\dots,X_n)$ for the trace of the matrix $M^k$. Then we have 
\begin{align*}
\mathbb{I}_{\mathcal{D}}(c) &= \frac{1}{F_c(X_1,\dots,X_n)}X_{i_1}^{k/2}\dots X_{i_s}^{k/2} \\
&= \frac{1}{F_c(X_1,\dots,X_n)}X_1^{h_{c,1}}\dots X_n^{h_{c,n}}
\end{align*}
for half integers $h_{c,i}$.
\end{proof}

\begin{proposition}
\label{prop:Shole}
Let $c$ be a $\mathcal{D}$-lamination consisting of a single loop of weight~$k$ on~$S_{\mathcal{D}}$. Assume that this loop is homotopic to a loop in the simple lamination and that $c$ and $m$ provide different orientations for this loop. Then 
\[
\mathbb{I}_{\mathcal{D}}(c)=X_1^{h_{c,1}}\dots X_n^{h_{c,n}}
\]
for half integers $h_{c,i}$.
\end{proposition}

\begin{proof}
Let $c$ be a $\mathcal{D}$-lamination satisfying the hypotheses of the proposition.  In this case, $\mathbb{I}_{\mathcal{D}}(c)$ is defined as the inverse absolute value of the $k$th power of an eigenvalue of the monodromy around~$c$. Let $i_1,\dots,i_s$ be the edges of the triangulation that spiral into this loop in the simple lamination. If the orientation provided by $c$ agrees with that of $S$, then the orientation provided by $m$ disagrees with~$S$, and Lemma~\ref{lem:leftrightturns} implies 
\[
\mathbb{I}_{\mathcal{D}}(c)=X_{i_1}^{-1/2}\dots X_{i_s}^{-1/2}.
\]
On the other hand, if the orientation provided by $c$ agrees with the orientation of~$S^\circ$, then the orientation provided by $m$ agrees with~$S$, and Lemma~\ref{lem:leftrightturns} implies 
\[
\mathbb{I}_{\mathcal{D}}(c)=X_{i_1}^{1/2}\dots X_{i_s}^{1/2}.
\]
Thus $\mathbb{I}_{\mathcal{D}}(c)$ has the desired form.
\end{proof}

Propositions~\ref{prop:Scircloop}, \ref{prop:Scirchole}, \ref{prop:Sloop}, and~\ref{prop:Shole} provide expressions for $\mathbb{I}_{\mathcal{D}}(c)$ in the cases where the curve $c$ is not an intersecting curve. In each case, the expression includes a product of the form $X_1^{h_{c,1}}\dots X_n^{h_{c,n}}$ where each $h_{c,i}$ is a half integer. In fact, the proofs of these propositions provide a more explicit description of this product: It is obtained by multiplying one factor $X_i^{\pm k/2}$ each time the curve $c$ crosses an edge~$i$. This fact will be important below when we discuss rational functions obtained from laminations.

\subsubsection{The general case}

We will now combine the above results to get general expressions for the functions $\mathbb{I}_{\mathcal{D}}(l)$. To state our result, we will need the following notation. Any $\mathcal{D}$-lamination $l\in\mathcal{D}_L(S,\mathbb{Z})$ can be represented by a collection of curves of weight~1. If this collection contains homotopic curves of weights $a$ and $b$ which are not intersecting curves, let us replace these by a single curve of weight $a+b$. In this way, we obtain a new collection of curves representing $l$. If we now cut the surface along the image of $\partial S$, we obtain a collection $\mathcal{C}^\circ$ of curves on $S^\circ$ and a collection $\mathcal{C}$ of curves on $S$.

\begin{theorem}
\label{thm:main}
Let $l\in\mathcal{D}_L(S,\mathbb{Z})$ and suppose that the orientation of the curve $\gamma_i$ agrees with $S$ whenever $\gamma_i$ meets a curve of $l$. Then 
\[
\mathbb{I}_\mathcal{D}(l)=\frac{\prod_{c\in\mathcal{C}^\circ}F_c(\widehat{X}_1,\dots,\widehat{X}_n)}{\prod_{c\in\mathcal{C}}F_c(X_1,\dots,X_n)}B_1^{g_{l,1}}\dots B_n^{g_{l,n}}X_1^{h_{l,1}}\dots X_n^{h_{l,n}}
\]
where the $F_c$ are polynomials, the $g_{l,i}$ are integers and the $h_{l,i}$ are half integers.
\end{theorem}

\begin{proof}
If $c$ is a closed curve which is homotopic to a loop in the simple lamination, we set $F_c=1$. Then the above formula for $\mathbb{I}_{\mathcal{D}}(l)$ is an immediate consequence of Propositions~\ref{prop:Scircloop}, \ref{prop:Scirchole}, \ref{prop:Sloop}, \ref{prop:Shole}, and~\ref{prop:intersecting} and the multiplicativity of $\mathbb{I}_{\mathcal{D}}$.
\end{proof}

The following result is closely related to Lemma~10.14 in~\cite{FT}.

\begin{proposition}
\label{prop:kapparational}
Let $l\in\mathcal{D}_L(S,\mathbb{Z})$ be any lamination, and let $l'$ be the lamination obtained from $l$ by altering the orientations of those curves $\gamma_i$ that meet a curve of~$l$ so that each orientation agrees with that of~$S$. Then we have $\mathbb{I}_{\mathcal{D}}(l)=\mathbb{I}_{\mathcal{D}}(l')\circ\tau$ where $\tau:\mathcal{D}^+(S)\rightarrow\mathcal{D}^+(S)$ is induced by a birational automorphism of the symplectic double.
\end{proposition}

\begin{proof}
Let $m\in\mathcal{D}^+(S)$. We define $m'$ as the point of $\mathcal{D}^+(S)$ such that 
\begin{enumerate}
\item $m'$ equips $S_{\mathcal{D}}$ with the same hyperbolic structure as~$m$.
\item Suppose the curve $\gamma_s$ is not shrunk to a node in the hyperbolic structure defined by~$m$. If the orientation that $l$ provides for~$\gamma_s$ agrees with the orientation of~$S$, then $m$ and $m'$ provide the same orientation for~$\gamma_s$. If the orientation that $l$ provides for~$\gamma_s$ disagrees with that of~$S$, then $m$ and $m'$ provide opposite orientations~$\gamma_s$.
\item Suppose the curve $\gamma_s$ is shrunk to a node in the hyperbolic structure defined by~$m$. If the orientation that $l$ provides for~$\gamma_s$ agrees with the orientation of~$S$, then $m'$ assigns to this node the same horocycles as~$m$. If the orientation that $l$ provides for~$\gamma_s$ disagrees with that of~$S$, then horocycles the $m'$ provides on $S$ and $S^\circ$ are rescaled using the invariants $\alpha$ and $\alpha^\circ$, respectively.
\end{enumerate}
If we now define $\tau:\mathcal{D}^+(S)\rightarrow\mathcal{D}^+(S)$ by $m\mapsto m'$, then it is immediate from the definitions that $\mathbb{I}_{\mathcal{D}}(l)=\mathbb{I}_{\mathcal{D}}(l')\circ\tau$.

To complete the proof, we must show that for any $j\in J$, the pullbacks $\tau^*X_j$ and $\tau^*B_j$ are given by rational expressions in the coordinates. By Lemma~12.3 of~\cite{IHES}, we know that $\tau^*X_j$ is given by a rational expression in the $X$-coordinates. We will adapt the methods of the proof of this result to show that $\tau^*B_j$ is rational as well. In particular, we will use the correspondence described in~\cite{IHES} between horocycles and vectors in~$\mathbb{R}^2$. To describe this correspondence, we identify $\partial\mathbb{H}=\mathbb{R}\cup\{\infty\}$ with $\mathbb{RP}^1$ by identifying $x\in\mathbb{R}$ with the line through $\left(\begin{array}{c} x \\ 1 \end{array}\right)$ and identifying $\infty$ with the line through $\left(\begin{array}{c} 1 \\ 0 \end{array}\right)$. Then any point $x\in\partial\mathbb{H}$ corresponds to a line in~$\mathbb{R}^2$, and a choice of horocycle at~$x$ is the same thing as a choice of vector in this line, up to a sign. If $x_1$ and~$x_2$ are points on~$\partial\mathbb{H}$ and $v_1=\left(\begin{array}{c} v_1^1 \\ v_1^2 \end{array}\right)$ and~$v_2=\left(\begin{array}{c} v_2^1 \\ v_2^2 \end{array}\right)$ are vectors corresponding to horocycles at~$x_1$ and~$x_2$, respectively, then the exponentiated signed half distance between these horocycles is 
\[
\left|
\det\left( \begin{array}{cc}
v_1^1 & v_2^1 \\
v_1^2 & v_2^2 \end{array} \right)\right|.
\]
(See \cite{dual} and \cite{IHES}.)

Let $m\in\mathcal{D}^+(S)$ and suppose the curve $\gamma_s$ is not shrunk to a node in the hyperbolic structure defined by~$m$. We will study how the $B$-coordinates change when we reverse the orientation of~$\gamma_s$. Fix an ideal triangulation $T$ of~$S$, and let $j\in J$ be an edge of~$T$ such that exactly one end of~$j$ terminates at the puncture corresponding to~$\gamma_s$. Lift $T$ to a triangulation~$\tilde{T}$ of the universal cover of~$S$ in~$\mathbb{H}$. Let $t$ be a triangle of $\tilde{T}$ such that one of the edges of~$t$ projects to~$j$. We can assume that the vertices of~$t$ are $-1$,~$0$, and~$\infty$ with the lift of~$j$ connecting~$0$ and~$\infty$. These vertices correspond under the identification $\partial\mathbb{H}\cong\mathbb{RP}^1$ to the lines in~$\mathbb{R}^2$ containing~$\left(\begin{array}{c} -1 \\ 1 \end{array}\right)$, $\left(\begin{array}{c} 0 \\ 1 \end{array}\right)$, and~$\left(\begin{array}{c} 1 \\ 0 \end{array}\right)$, respectively. Let us choose horocycles around the endpoints of~$t$ and write $A_e$ for the exponentiated signed half distance between the horocycles at the ends of an edge~$e$. One can check that the horocycles around the vertices of $t$ correspond to the vectors indicated in the diagram below.
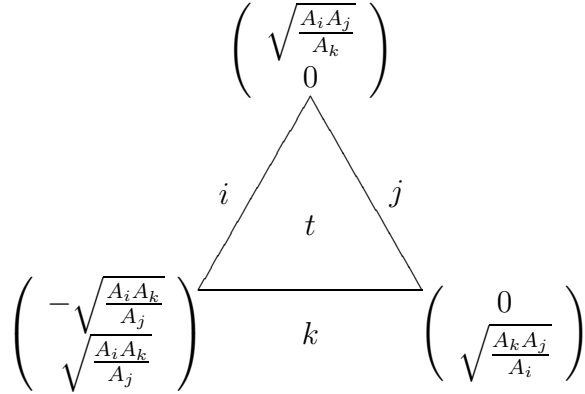
\[
\xy /l1.7pc/:
{\xypolygon3"A"{~:{(-2.4,0):}}},
(1,-3.25)*-{\left(\begin{array}{c} \sqrt{\frac{A_i A_j}{A_k}} \\ 0 \end{array}\right)};
(-2.6,2)*-{\left(\begin{array}{c} 0 \\ \sqrt{\frac{A_k A_j}{A_i}} \end{array}\right)};
(4.8,2)*-{\left(\begin{array}{c} -\sqrt{\frac{A_i A_k}{A_j}} \\ \sqrt{\frac{A_i A_k}{A_j}} \end{array}\right)};
(2.6,-0.6)*{i};
(-0.6,-0.6)*{j};
(1,2)*{k};
(1,0)*{t};
\endxy
\]
There is a corresponding triangle $t^\circ$ in the universal cover of~$S^\circ$ in~$\mathbb{H}$, and we can assume that its vertices are~$-1$, $0$, and~$\infty$. The horocycles already chosen determine a unique triple of horocycles around the vertices of~$t^\circ$, and we write $A_e^\circ$ for the exponentiated signed half distance between the horocycles at the ends of an edge $e$ of~$t^\circ$. 

Let us assume that the curve $\gamma_s$ corresponds to the vertex of~$t$ at $0\in\partial\mathbb{H}$. Let $\eta_1,\dots,\eta_N$ be the edges of~$T$ that spiral into~$\gamma_s$ in clockwise order with $\eta_1=j$, and write $\zeta_r$ for the third edge of the triangle having edges $\eta_r$ and~$\eta_{r+1}$. It is shown in the proof of Lemma~12.3 of~\cite{IHES} that if we reverse the orientation of~$\gamma_s$, then the vertex of $t$ moves from $0\in\partial\mathbb{H}$ to~$\beta=(X_{\eta_1}\dots X_{\eta_N}-1)/C\in\partial\mathbb{H}$ where $C=1+X_{\eta_N}+X_{\eta_N}X_{\eta_{N-1}}+\dots+X_{\eta_N}X_{\eta_{N-1}}\dots X_{\eta_2}$. Under our identification $\partial\mathbb{H}\cong\mathbb{RP}^1$, this corresponds to the line in~$\mathbb{R}^2$ containing $\left(\begin{array}{c} \beta \\ 1 \end{array}\right)$. We have 
\[
\left|
\det\left( \begin{array}{cc}
\sqrt{\frac{A_i}{A_k A_j}} & 0 \\
\sqrt{\frac{A_i}{A_k A_j}}/\beta & \sqrt{\frac{A_k A_j}{A_i}} \end{array} \right)\right|=1,
\]
and therefore the vector $v'=\left(\begin{array}{c} \sqrt{\frac{A_i}{A_k A_j}} \\ \sqrt{\frac{A_i}{A_k A_j}}/\beta \end{array}\right)$ corresponds to the unique horocycle centered at $\beta\in\partial\mathbb{H}$ which is tangent to the horocycle determined by the vector $v=\left(\begin{array}{c} 0 \\ \sqrt{\frac{A_k A_j}{A_i}} \end{array}\right)$. By a similar argument, we calculate that the corresponding horocycle in the triangle $t^\circ$ is associated with the vector $\left(\begin{array}{c} \sqrt{\frac{A_i^\circ}{A_k^\circ A_j^\circ}} \\ \sqrt{\frac{A_i^\circ}{A_k^\circ A_j^\circ}}/\beta^\circ \end{array}\right)$ where we have defined the quantity $\beta^\circ=(\widehat{X}_{\eta_1}\dots\widehat{X}_{\eta_N}-1)/C^\circ$ and $C^\circ=1+\widehat{X}_{\eta_N}+\widehat{X}_{\eta_N}\widehat{X}_{\eta_{N-1}} + \dots+\widehat{X}_{\eta_N}\widehat{X}_{\eta_{N-1}}\dots \widehat{X}_{\eta_2}$.

Let us write $A_j'$ for the exponentiated signed half distance that we get by replacing~$v$ by~$v'$, and let us write $(A_j^\circ)'$ for the similar quantity associated with the triangle $t^\circ$. Thus we have 
\[
A_j'=
\left|
\det\left( \begin{array}{cc}
\sqrt{\frac{A_i}{A_k A_j}} & \sqrt{\frac{A_i A_j}{A_k}} \\
\sqrt{\frac{A_i}{A_k A_j}}/\beta & 0 \end{array} \right)\right|=\frac{A_i}{A_k A_j}\frac{1}{\beta} A_j
\]
and similarly 
\[
(A_j^\circ)'=\frac{A_i^\circ}{A_k^\circ A_j^\circ}\frac{1}{\beta^\circ}A_j^\circ.
\]
It follows that when we reverse the orientation of $\gamma_s$, the coordinate $B_j=A_j^\circ/A_j$ is replaced by 
\[
\frac{(A_j^\circ)'}{A_j'}=\frac{B_i}{B_kB_j}\frac{\beta}{\beta^\circ}B_j,
\]
which is a rational function of the $B$- and $X$-coordinates. If both ends of $j$ terminate at the puncture corresponding to $\gamma_s$, then we can apply the above argument successively to the two ends of the lifted edge, and we get a similar rational expression with additional factors. This proves that the pullback $\tau^*B_j$ is given by a rational expression when restricted to points of~$\mathcal{D}^+(S)$ where the curve $\gamma_s$ is not shrunk to a node.

It remains to show that $\tau^*B_j$ is given by the same rational expression when applied to a point~$m\in\mathcal{D}^+(S)$ for which the curve $\gamma_s$ is shrunk to a node. Note that we have $X_{\eta_1}\dots X_{\eta_N}=\widehat{X}_{\eta_1}\dots\widehat{X}_{\eta_N}$, and therefore this rational expression is 
\[
\frac{B_i}{B_kB_j}\frac{\beta}{\beta^\circ}B_j = \frac{B_i}{B_kB_j}\frac{1+\widehat{X}_{\eta_N}+\widehat{X}_{\eta_N}\widehat{X}_{\eta_{N-1}} + \dots+\widehat{X}_{\eta_N}\widehat{X}_{\eta_{N-1}}\dots \widehat{X}_{\eta_2}}{1+X_{\eta_N}+X_{\eta_N}X_{\eta_{N-1}}+\dots+X_{\eta_N}X_{\eta_{N-1}}\dots X_{\eta_2}}B_j.
\]
Using the definition of the $X$-coordinates as cross ratios, one can check that 
\[
\frac{A_i^\circ}{A_k^\circ A_j^\circ}(1+\widehat{X}_{\eta_N}+\widehat{X}_{\eta_N}\widehat{X}_{\eta_{N-1}} + \dots+\widehat{X}_{\eta_N}\widehat{X}_{\eta_{N-1}}\dots \widehat{X}_{\eta_2})=\sum_{t=1}^N\frac{A_{\zeta_t}^\circ}{A_{\eta_t}^\circ A_{\eta_{t+1}}^\circ}
\]
and likewise 
\[
\frac{A_i}{A_kA_j}(1+X_{\eta_N}+X_{\eta_N}X_{\eta_{N-1}}+\dots+X_{\eta_N}X_{\eta_{N-1}}\dots X_{\eta_2})=\sum_{t=1}^N\frac{A_{\zeta_t}}{A_{\eta_t} A_{\eta_{t+1}}}.
\]
Thus, at the point $m$, the above rational function reduces to 
\[
\frac{B_i}{B_kB_j}\frac{\beta}{\beta^\circ}B_j = \frac{\alpha^\circ}{\alpha} B_j = \tau^*B_j
\]
as desired.
\end{proof}

It follows from Theorem~\ref{thm:main} and Proposition~\ref{prop:kapparational} that $\mathbb{I}_{\mathcal{D}}(l)$ is given by an algebraic expression in the $B$- and $X$-coordinates for all $l\in\mathcal{D}_L(S,\mathbb{Z})$. It is given by a rational expression if and only if the exponents $h_{l,i}$ in Theorem~\ref{thm:main} are integers.

\subsubsection{Expansion of cluster variables}

Theorem~\ref{thm:main} is an extension of Proposition~\ref{prop:specialization}, which gives a formula expressing the $B$-variables of the cluster $\mathcal{D}$-algebra in terms of the $B$- and $X$-variables of an initial seed.

To see this, suppose $T$ is an ideal triangulation of $S$ with no self-folded triangles. This triangulation determines an initial seed for the cluster algebra $\mathcal{D}$. Let $c$ be an arc of some other triangulation with no self-folded triangles. By gluing this arc to its image under the map $S\rightarrow S^\circ$, we get an intersecting curve $l$ on $S_\mathcal{D}$. Choose the orientations of the curves~$\gamma_i$ to agree with the orientation of~$S$. Then the canonical function $\mathbb{I}_{\mathcal{D}}(l)$ associated to this intersecting curve is simply the $B$-coordinate~$B_c$ corresponding to the arc~$c$. Suppose there is a sequence of regular flips taking the triangulation $T$ to the triangulation containing~$c$. By~Proposition~\ref{prop:specialization}, we can write 
\[
B_c=\frac{F_c(\widehat{X}_1,\dots,\widehat{X}_n)}{F_c(X_1,\dots,X_n)}B_1^{g_1}\dots B_n^{g_n}
\]
where the $B_j$ and $X_j$ are the coordinates associated to the internal edges of~$T$. The $F$-polynomials in this formula can be computed using the matrix formulas of Musiker and Williams~\cite{MW}. Since $T$ has no self-folded triangles, one gets the same polynomials by lifting the arc~$c$ and triangulation~$T$ to the universal cover and applying the results of Musiker and Williams there. This implies that the $F$-polynomials appearing in this formula are the same ones we get from Theorem~\ref{thm:main}.

\subsubsection{Generalized $F$-polynomials}

Notice that if $c$ is any arc connecting loops in the simple lamination on $S_{\mathcal{D}}$, then the polynomial $F_c$ appearing in Theorem~\ref{thm:main} is an $F$-polynomial in the sense of Fomin and Zelevinsky. On the other hand, if $c$ is a closed loop, then $F_c$ does not arise in this way. This suggests that the polynomials appearing in this theorem should be seen as generalizing the $F$-polynomials. Such generalized $F$-polynomials have appeared previously in the work of Musiker, Schiffler, and Williams~\cite{MW,MSW2} where they are defined in terms of perfect matchings of a graph.

Indeed, suppose $T$ is an ideal triangulation of $S$ with no self-folded triangles. It follows from Lemma~2.13 of~\cite{FST} that such a triangulation always exists for the surfaces that Musiker and Williams consider in~\cite{MW}. Let $c$ be a loop on $S_{\mathcal{D}}$ that lies entirely in $S$ and is not homotopic to a loop in the simple lamination. The monodromy around such a loop was computed in the proof of Proposition~\ref{prop:Sloop} as a product of matrices. By Proposition~4.12 of~\cite{MW}, the absolute value of the trace of this monodromy is given by their expression $\bar{\chi}_{c,T}$ with the substitutions $x_i=1$ and~$y_i=X_i$. Using the notation from~\cite{MW}, we have 
\begin{align*}
F_c(X_1,\dots,X_n) &= \mathbb{I}_{\mathcal{D}}(c)^{-1}X_{i_1}^{1/2}\dots X_{i_s}^{1/2} \\
&= \bar{\chi}_{c,T}X_{i_1}^{1/2}\dots X_{i_s}^{1/2} \\
&= \widehat{\chi}_{c,T}
\end{align*}
with $x_i=1$ and~$y_i=X_i$. For a triangulation without self-folded triangles, Musiker and Williams show that this last expression equals the generalized $F$-polynomial that they associate to the curve $c$.

If the curve $c$ lies on $S^\circ$ instead of $S$, then the same argument with the variables $X_i$ replaced by $\widehat{X}_i$ shows that our expressions $F_c(\widehat{X}_1,\dots,\widehat{X}_n)$ agree with the generalized $F$-polynomials introduced by Musiker, Schiffler, and Williams.

\subsection{Rational functions from laminations}

Theorem~\ref{thm:main} and Proposition~\ref{prop:kapparational} allow us to completely characterize the function $\mathbb{I}_{\mathcal{D}}(l)$ associated to a $\mathcal{D}$-lamination $l\in\mathcal{D}_L(S,\mathbb{Z})$. However, the exponents $h_i$ are only half integral, so an element of $\mathcal{D}_L(S,\mathbb{Z})$ does not give rise to a rational function in the~$B_i$ and~$X_i$ in general. In analogy with the classical case of $\mathcal{A}$- and $\mathcal{X}$-laminations, we expect that $\mathbb{I}_{\mathcal{D}}(l)$ is a rational function of the variables $B_i$ and $X_i$ if and only if $l$ has integral coordinates. This expectation turns out to be correct, and we will prove it in several steps. In addition, we will give a homological condition on a $\mathcal{D}$-lamination~$l$ which is equivalent to the integrality of the coordinates of $l$.

To state the relevant homological condition, we first use the natural map $S^\circ\rightarrow S$ to draw all of the curves in $\mathcal{C}$ or $\mathcal{C}^\circ$ on the surface $S$. For example, the intersecting curve illustrated in the introduction gives the following picture.
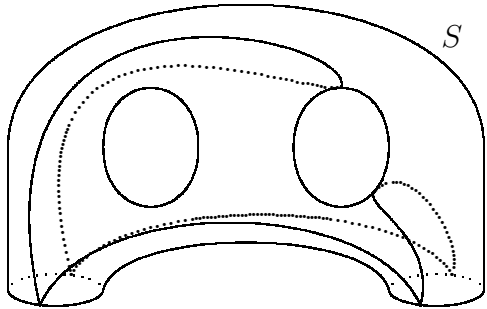
\[
\xy 0;/r.50pc/: 
(-6,-4.5)*\ellipse(3,1){.}; 
(-6,-4.5)*\ellipse(3,1)__,=:a(-180){-}; 
(6,-4.5)*\ellipse(3,1){.}; 
(6,-4.5)*\ellipse(3,1)__,=:a(-180){-}; 
(-9,0)*{}="1";
(-3,0)*{}="2";
(3,0)*{}="3";
(9,0)*{}="4";
(-15,0)*{}="A2";
(15,0)*{}="B2";
"1";"2" **\crv{(-9,-5) & (-3,-5)};
"1";"2" **\crv{(-9,5) & (-3,5)};
"3";"4" **\crv{(3,-5) & (9,-5)};
"3";"4" **\crv{(3,5) & (9,5)};
"A2";"B2" **\crv{(-15,12) & (15,12)};
(-9,-9)*{}="A";
(9,-9)*{}="B";
"A";"B" **\crv{(-8,-5) & (8,-5)}; 
(-15,-9)*{}="A1";
(15,-9)*{}="B1";
"B2";"B1" **\dir{-}; 
"A2";"A1" **\dir{-};
(-11,-8)*{}="V1";
(13,-8)*{}="V2";
(-13,-10)*{}="V3";
(11,-10)*{}="V4";
(6,3.8)*{}="V5";
(8,-3)*{}="V6";
"V1";"V2" **\crv{~*=<2pt>{.} (-8,-3) & (10,-3)};
"V3";"V4" **\crv{(-10,-3) & (8,-3)};
"V3";"V5" **\crv{(-18,13) & (7,7)};
"V1";"V5" **\crv{~*=<2pt>{.} (-16,12) & (4,3)};
"V2";"V6" **\crv{~*=<2pt>{.} (14,-6) & (10,0)};
"V4";"V6" **\crv{(12,-6) & (8,-4)};
(13,7)*{S}; 
\endxy
\]

In this way, we get a collection of curves on $S$, and any arc in this collection that connects two holes can be viewed as a singular 1-simplex with $\mathbb{Z}/2\mathbb{Z}$-coefficients. Similarly, if $c$ is a closed loop in this collection, then we can view this loop as a 1-cycle with $\mathbb{Z}/2\mathbb{Z}$-coefficients. This cycle consists of a single 0-simplex in the interior of each triangle that $c$ intersects in a fixed ideal triangulation~$T$ together with a 1-simplex connecting each pair of consecutive 0-simplices. Let $\sigma_l$ be sum of all the 1-simplices obtained in this way from arcs and closed loops. Then $\sigma_l$ is a cycle representing an element $[\sigma_l]$ of the singular homology $H_1(S,\mathbb{Z}/2\mathbb{Z})$ of $S$ with coefficients in~$\mathbb{Z}/2\mathbb{Z}$. With this notation, one has the following result.

\begin{theorem}
\label{thm:integralDlam}
Let $l\in\mathcal{D}_L(S,\mathbb{Z})$ be a $\mathcal{D}$-lamination. The following are equivalent:
\begin{enumerate}
\item $b_j$,~$x_j\in\mathbb{Z}$ for all~$j\in J$.
\item $[\sigma_l]=0\in H_1(S,\mathbb{Z}/2\mathbb{Z})$.
\item $\mathbb{I}_{\mathcal{D}}(l)$ is a rational function in the variables $B_j$ and $X_j$.
\end{enumerate}
\end{theorem}

\begin{proof}
To simplify notation, we will denote $\sigma_l$ simply by $\sigma$. We assume that the ideal triangulation $T$ has been drawn on~$S$ in such a way that the edges spiral into the holes of~$S$ in the direction specified by~$l$.

\Needspace*{2\baselineskip}
\begin{step}[1]
$1\implies2$
\end{step}

Suppose $l\in\mathcal{D}_L(S,\mathbb{Z})$ is a lamination with integral coordinates. The 1-simplices that make up~$\sigma$ may intersect the triangulation in infinitely many points. We will begin by constructing a homologous cycle~$\sigma'$ that intersects each edge of the triangulation in a finite, and in fact an \emph{even}, number of points. We will then show that $[\sigma']=0\in H_1(S,\mathbb{Z}/2\mathbb{Z})$.

To define this cycle $\sigma'$, consider any hole in~$S$ that meets one of the simplices of~$\sigma$. Modify the cycle $\sigma$ near this hole as illustrated below so that the simplices meet at an edge of the triangulation spiraling into the hole rather than at a point of $\partial S$. Doing this near every hole gives the desired cycle $\sigma'$.
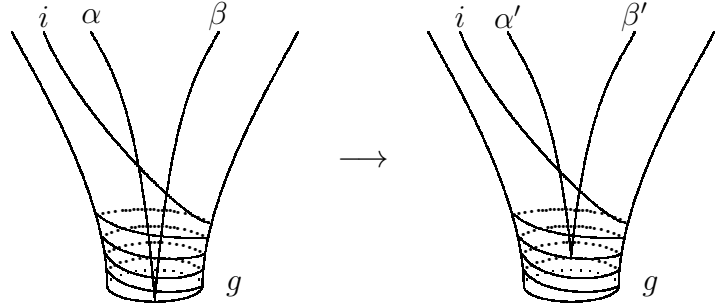
\[
\xy 0;/r.50pc/: 
(0,-4)*\ellipse(3,1){.}; 
(0,-4)*\ellipse(3,1)__,=:a(180){-}; 
(-9,8)*{}="1"; 
(-3,-8)*{}="2"; 
"1";"2" **\crv{(-8,6) & (-3,-2)}; 
(9,8)*{}="3"; 
(3,-8)*{}="4"; 
"3";"4" **\crv{(8,6) & (3,-2)}; 
(-4,8)*{}="L"; 
(0,-9)*{}="M";
(4,8)*{}="N"; 
"L";"M" **\crv{(-3,6) & (-1,4)}; 
"N";"M" **\crv{(3,6) & (1,4)}; 
(-3,-7)*{}="A1"; 
(3,-8)*{}="B1";
(-3.1,-6)*{}="A2"; 
(3,-7)*{}="B2";
(-3.3,-4.9)*{}="A3"; 
(3.1,-6)*{}="B3";
(-3.6,-3.5)*{}="A4"; 
(3.3,-5)*{}="B4";
(-7,8)*{}="A5"; 
(3.4,-4)*{}="B5";
"A1";"B1" **\crv{(-2,-8.6) & (2,-8.6)}; 
"A1";"B2" **\crv{~*=<2pt>{.} (-2,-6) & (2,-6)};
"A2";"B2" **\crv{(-2,-7.8) & (2,-7.8)}; 
"A2";"B3" **\crv{~*=<2pt>{.} (-2,-5) & (2,-5)};
"A3";"B3" **\crv{(-2,-6.5) & (2,-6.5)}; 
"A3";"B4" **\crv{~*=<2pt>{.} (-2,-4) & (2,-4)};
"A4";"B4" **\crv{(-2,-5) & (2,-5)}; 
"A4";"B5" **\crv{~*=<2pt>{.} (-2,-3) & (2,-3)};
"A5";"B5" **\crv{(-5,3) & (2,-4)}; 
(-7,9)*{i}; 
(-4,9)*{\alpha}; 
(4,9)*{\beta}; 
(5,-8)*{g}; 
\endxy
\quad
\longrightarrow
\quad
\xy 0;/r.50pc/: 
(0,-4)*\ellipse(3,1){.}; 
(0,-4)*\ellipse(3,1)__,=:a(180){-}; 
(-9,8)*{}="1"; 
(-3,-8)*{}="2"; 
"1";"2" **\crv{(-8,6) & (-3,-2)}; 
(9,8)*{}="3"; 
(3,-8)*{}="4"; 
"3";"4" **\crv{(8,6) & (3,-2)}; 
(-4,8)*{}="L"; 
(0,-6.1)*{}="M";
(4,8)*{}="N"; 
"L";"M" **\crv{(-3,6) & (-1,4)}; 
"N";"M" **\crv{(3,6) & (1,4)}; 
(-3,-7)*{}="A1"; 
(3,-8)*{}="B1";
(-3.1,-6)*{}="A2"; 
(3,-7)*{}="B2";
(-3.3,-4.9)*{}="A3"; 
(3.1,-6)*{}="B3";
(-3.6,-3.5)*{}="A4"; 
(3.3,-5)*{}="B4";
(-7,8)*{}="A5"; 
(3.4,-4)*{}="B5";
"A1";"B1" **\crv{(-2,-8.6) & (2,-8.6)}; 
"A1";"B2" **\crv{~*=<2pt>{.} (-2,-6) & (2,-6)};
"A2";"B2" **\crv{(-2,-7.8) & (2,-7.8)}; 
"A2";"B3" **\crv{~*=<2pt>{.} (-2,-5) & (2,-5)};
"A3";"B3" **\crv{(-2,-6.5) & (2,-6.5)}; 
"A3";"B4" **\crv{~*=<2pt>{.} (-2,-4) & (2,-4)};
"A4";"B4" **\crv{(-2,-5) & (2,-5)}; 
"A4";"B5" **\crv{~*=<2pt>{.} (-2,-3) & (2,-3)};
"A5";"B5" **\crv{(-5,3) & (2,-4)}; 
(-7,9)*{i}; 
(-4,9)*{\alpha'}; 
(4,9)*{\beta'}; 
(5,-8)*{g}; 
\endxy
\]
It is easy to see that $\sigma'$ is homologous to~$\sigma$. We claim that it intersects each edge of the triangulation of~$S$ in an even number of points. (Here the number of intersections is counted with multiplicity so that if two 1-simplices intersect an edge $i$ at the same point, then we say there are two intersections.)

Indeed, consider the preimages of $\sigma$ and $\sigma'$ in the universal cover $\tilde{S}$ of~$S$.
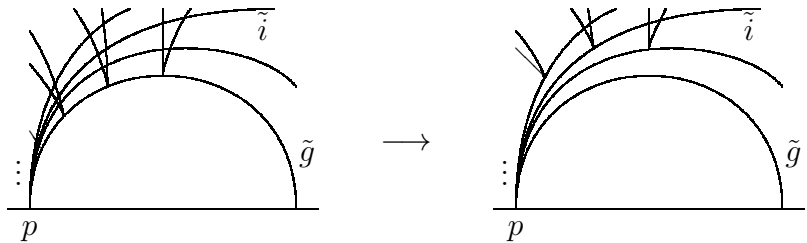
\[
\xy 0;/r.35pc/: 
(-12,-6)*{}="1"; 
(12,-6)*{}="2"; 
"1";"2" **\crv{(-12,10) & (12,10)}; 
(-14,-6)*{}="X"; 
(14,-6)*{}="Y"; 
"X";"Y" **\dir{-}; 
(-12,-8)*{p}; 
(13,-1)*{\tilde{g}}; 
(-12,0)*{}="P1";
(-12,1)*{}="Q1";
(-11,-0.5)*{}="R1"; 
(-12,7)*{}="P2";
(-12,10)*{}="Q2";
(-9,2.5)*{}="R2"; 
(-8,12)*{}="P3";
(-5.5,12)*{}="Q3";
(-5,5.2)*{}="R3"; 
(0,12)*{}="P4";
(2.5,12)*{}="Q4";
(0,6)*{}="R4"; 
"P1";"R1" **\dir{-}; 
"Q1";"R1" **\dir{-}; 
"P2";"R2" **\crv{(-11,6) & (-9,3)}; 
"Q2";"R2" **\crv{(-10,7) & (-9,3)}; 
"P3";"R3" **\crv{(-7,10.5) & (-6,9)}; 
"Q3";"R3" **\crv{(-5,8) & (-5,6)}; 
"P4";"R4" **\crv{(0,12) & (0,6)}; 
"Q4";"R4" **\crv{(1.5,11) & (0,7)}; 
(-13,-2)*{\vdots};
(12,5)*{}="A";
(10,12)*{}="B";
(-3,12)*{}="C";
"1";"A" **\crv{(-12,12) & (8,10)}; 
"1";"B" **\crv{(-12,9) & (0,12)}; 
"1";"C" **\crv{(-12,9) & (-3,12)}; 
(9,10.5)*{\tilde{i}};
(-16,-8)*{}="Z"; 
(16,-8)*{}="W"; 
\endxy
\quad
\longrightarrow
\quad
\xy 0;/r.35pc/: 
(-12,-6)*{}="1"; 
(12,-6)*{}="2"; 
"1";"2" **\crv{(-12,10) & (12,10)}; 
(-14,-6)*{}="X"; 
(14,-6)*{}="Y"; 
"X";"Y" **\dir{-}; 
(-12,-8)*{p}; 
(13,-1)*{\tilde{g}}; 
(-12,8.5)*{}="P2";
(-12,10)*{}="Q2";
(-9.5,6)*{}="R2"; 
(-7,12)*{}="P3";
(-5.5,12)*{}="Q3";
(-5,8.5)*{}="R3"; 
(0,12)*{}="P4";
(1.5,12)*{}="Q4";
(0,8.5)*{}="R4"; 
"P2";"R2" **\dir{-}; 
"Q2";"R2" **\crv{(-11,9) & (-9.5,6.5)}; 
"P3";"R3" **\crv{(-5.5,10) & (-5,8.5)}; 
"Q3";"R3" **\crv{(-5.5,12) & (-5,8.5)}; 
"P4";"R4" **\crv{(0,12) & (0,8.5)}; 
"Q4";"R4" **\crv{(0.5,10) & (0,9)}; 
(-13,-2)*{\vdots};
(12,5)*{}="A";
(10,12)*{}="B";
(-3,12)*{}="C";
"1";"A" **\crv{(-12,12) & (8,10)}; 
"1";"B" **\crv{(-12,9) & (0,12)}; 
"1";"C" **\crv{(-12,9) & (-3,12)}; 
(9,10.5)*{\tilde{i}};
(-16,-8)*{}="Z"; 
(16,-8)*{}="W"; 
\endxy
\]
Let $g$ be the boundary of a hole in~$S$, and let $\tilde{g}$ be an arc in the universal cover that projects to $g$. Then the preimage of $\sigma$ near $g$ consists of infinitely many curves meeting $\tilde{g}$ as illustrated above on the left. These curves are alternately preimages of~$\alpha$ or~$\beta$. By definition, the coordinate $b_i$ is obtained by choosing, near each endpoint~$p$ of~$\tilde{i}$, a distinguished pair of preimages $\tilde{\alpha}_p$ and $\tilde{\beta}_p$ of $\alpha$ and $\beta$, respectively. We then count the number $n_i$ of preimages of curves from $S$ that intersect $\tilde{i}$ between the distinguished curves and the number $n_i^\circ$ of preimages of curves from $S^\circ$ that intersect $\tilde{i}$ between the distinguished curves. Then $b_i$ is given by 
\[
b_i=\frac{n_i^\circ}{2}-\frac{n_i}{2}.
\]
Since we are assuming $b_i\in\mathbb{Z}$ for all $i$, we know that $n_i^\circ-n_i\in2\mathbb{Z}$, or equivalently that $n_i^\circ+n_i\in2\mathbb{Z}$. When we modify $\sigma$ to get the cycle $\sigma'$, the preimage of $\sigma$ in $\tilde{S}$ is modified as in the illustration above. The distinguished curves $\tilde{\alpha}_p$ and $\tilde{\beta}_p$ that we chose above determine preimages $\tilde{\alpha}_p'$ and $\tilde{\beta}_p'$ of $\alpha'$ and $\beta'$, respectively. As before, we can count the number of intersections between $\tilde{i}$ and the preimages of 1-simplices of $\sigma'$. The total number of such intersections differs from $n_i^\circ+n_i$ by an even number. Hence the total number of these intersections is even. But this number equals the number of intersections of $\sigma'$ with the edge~$i$. This proves that $\sigma'$ intersects each edge of the triangulation in an even number of points.

It remains to show that $[\sigma']=0\in H_1(S,\mathbb{Z}/2\mathbb{Z})$. If $\sigma_c$ is the summand of $\sigma'$ obtained from a closed loop $c$ on $S$ or $S^\circ$, we can modify this summand to get a homologous cycle whose vertices all lie on edges of $T$. In general, the simplices of the resulting cycle may intersect, and we will deform these arcs so that each of the intersections occurs on an edge of the ideal triangulation. We will consider the refinement of the resulting cycle whose 1-simplices are the intersections of the 1-simplices with the triangles of the ideal triangulation. Abusing notation, we will denote this refinement also by $\sigma'$. We want to show that this $\sigma'$ represents the zero class in homology.

Consider a triangle on the surface with sides~$i$,~$j$, and~$k$. The simplices of $\sigma'$ contained in this triangle are arcs connecting adjacent sides of the triangle so that the total number of arcs that terminate on a given side is always even. Let $\alpha$ be an arc joining two sides, say~$i$ and~$j$. We may assume that there are no additional arcs between $\alpha$ and the vertex $u$ that it surrounds.
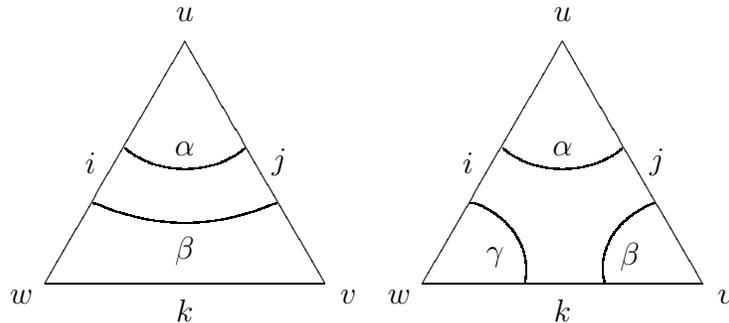
\[
\xy /l1.7pc/:
{\xypolygon3"A"{~:{(-3,0):}}},
{\xypolygon3"B"{~:{(-3.5,0):}~>{}}},
(-0.1,-1)*{}="j1";
(2.1,-1)*{}="i1";
(-0.7,0)*{}="j2";
(2.7,0)*{}="i2";
"B1"*{u};
"B3"*{v};
"B2"*{w};
(2.75,-0.75)*{i};
(-0.75,-0.75)*{j};
(1,2)*{k};
"j1";"i1" **\crv{(0.5,-0.5) & (1.5,-0.5)};
"j2";"i2" **\crv{(0.5,0.5) & (1.5,0.5)};
(1,-1)*{\alpha};
(1,0.9)*{\beta};
\endxy
\quad
\xy /l1.7pc/:
{\xypolygon3"A"{~:{(-3,0):}}},
{\xypolygon3"B"{~:{(-3.5,0):}~>{}}},
(2.1,-1)*{}="i1";
(2.7,0)*{}="i2";
(-0.1,-1)*{}="j1";
(-0.7,0)*{}="j2";
(1.7,1.5)*{}="k1";
(0.2,1.5)*{}="k2";
"B1"*{u};
"B3"*{v};
"B2"*{w};
(2.75,-0.75)*{i};
(-0.75,-0.75)*{j};
(1,2)*{k};
"i1";"j1" **\crv{(1.5,-0.5) & (0.5,-0.5)};
"j2";"k2" **\crv{(0.5,0.5) & (0.25,1.5)};
"i2";"k1" **\crv{(2.5,0) & (1.5,0.5)};
(1,-1)*{\alpha};
(-0.25,1)*{\beta};
(2.25,1)*{\gamma};
\endxy
\]
There is an even number of arcs that meet the edge $j$. It follows that $\alpha$ cannot be the only one and hence there is another arc $\beta$ that terminates on this edge. By our choice of $\alpha$, this arc must not lie between $\alpha$ and $u$, so $\beta$ looks like one of the arcs illustrated above depending on whether it ends on~$i$ or~$k$.

First consider the case where $\beta$ ends on $i$. In this case, there is a cycle formed by the arcs $\alpha$ and~$\beta$ together with the portions of~$i$ and~$j$ between the endpoints of these arcs. This cycle is obviously zero in $H_1(S,\mathbb{Z}/2\mathbb{Z})$.

Now consider the case where $\beta$ joins the edges $j$ and~$k$. We can assume that there are no arcs between $\beta$ and the vertex $v$ that it surrounds. Since the total number of arcs that meet the edge $k$ is even, there is another arc $\gamma$ that meets~$k$. By our choice of $\beta$, this arc must not lie between $\beta$ and $v$. Moreover, we can assume that it does not join $k$ and $j$ because this case was already treated above. It follows that $\gamma$ joins $k$ and $i$ as illustrated above. There is a cycle formed by the arcs $\alpha$,~$\beta$, and~$\gamma$ together with the portions of $i$,~$j$, and~$k$ between their endpoints. This cycle is again zero in $H_1(S,\mathbb{Z}/2\mathbb{Z})$.

Thus we define for any triangle a cycle with $\mathbb{Z}/2\mathbb{Z}$-coefficients. Consider the sum of this cycle with $\sigma'$ in the space of all such cycles. It is another cycle representing the same homology class as $\sigma'$. It consists of arcs on the edges of the triangulation together with arcs connecting different edges, and the total number of arcs of the latter type that meet at a given edge is always even. Thus we can iterate the above construction and continue adding cycles to get new cycles homologous to $\sigma'$. At each step, the number of arcs connecting different edges decreases, so eventually we are left with zero. This proves that $\sigma'$ is homologous to zero.

\Needspace*{2\baselineskip}
\begin{step}[2]
$2\implies1$
\end{step}

Suppose our cycle satisfies $[\sigma]=0\in H_1(S,\mathbb{Z}/2\mathbb{Z})$. Since $l\in\mathcal{D}_L(S,\mathbb{Z})$, we can realize $l$ as a collection of curves on $S_\mathcal{D}$ with integral weights. Then the portion of $l$ that lies on $S\subseteq S_{\mathcal{D}}$ is a collection of curves with integral weights. Since $\mathcal{X}_L(S,\mathbb{Z})=\mathcal{X}(\mathbb{Z}^t)$, it follows that $l$ has integral $x$-coordinates. Thus we only need to show that $l$ has integral $b$-coordinates.

To prove this fact, we first modify $\sigma$ as in Step~1 to get a new cycle $\sigma'$. In general, the interiors of two 1-simplices of~$\sigma'$ may intersect. If such an intersection point lies on an edge $i$ of~$T$, let us deform this edge slightly to avoid the intersection without changing the number of intersections between $i$ and the interior points of 1-simplices of~$\sigma'$. We then take a refinement of $\sigma'$ in which every point of self intersection is a 0-simplex. Abusing notation, we will denote this refinement also by $\sigma'$.

Since $\sigma'$ is homologous to~$\sigma$, our assumption implies $[\sigma']=0\in H_1(S,\mathbb{Z}/2\mathbb{Z})$. We can realize $S$ as a two-dimensional simplicial complex $\Delta$ such that every simplex of~$\sigma'$ is included in~$\Delta$. This simplicial complex defines a simplicial homology group $H_1^{\Delta}(S,\mathbb{Z}/2\mathbb{Z})$ with coefficients in $\mathbb{Z}/2\mathbb{Z}$. There is an isomorphism 
\[
H_1^{\Delta}(S,\mathbb{Z}/2\mathbb{Z})\cong H_1(S,\mathbb{Z}/2\mathbb{Z})
\]
between simplicial and singular homology, and so $\sigma'$ also represents the zero class in \emph{simplicial} homology. This means there is a chain $\eta$ of simplices in $\Delta$ with $\sigma'=\partial\eta$. Since $\sigma'$ does not meet a marked point or hole on~$S$, we know that $\eta$ does not meet a marked point or hole~$S$. It follows that the endpoints of an edge~$i$ of~$T$ cannot lie on~$\eta$. If $i$ crosses $\sigma'$ into the interior of a 2-simplex of~$\eta$, then it must leave $\eta$ by crossing $\sigma'$ again.

It follows that $i$ intersects $\sigma'$ in an even number of points. As we showed in Step~1, this number of intersections differs by an even number from $n_i^\circ+n_i$, so we have $n_i^\circ+n_i\in2\mathbb{Z}$ and hence $n_i^\circ-n_i\in2\mathbb{Z}$. Hence $b_i=\frac{1}{2}(n_i^\circ-n_i)\in\mathbb{Z}$ as desired.

\Needspace*{3\baselineskip}
\begin{step}[3]
$2\implies3$
\end{step}

Suppose the cycle $\sigma$ satisfies $[\sigma]=0\in H_1(S,\mathbb{Z}/2\mathbb{Z})$. By Proposition~\ref{prop:kapparational}, we may assume that the orientation of the curve $\gamma_i$ agrees with~$S$ whenever $\gamma_i$ meets a curve of~$l$. As in Steps~1 and~2, we will begin by replacing $\sigma$ by a homologous cycle.

Let $c$ be a closed curve in $\mathcal{C}$ or $\mathcal{C}^\circ$. Then Propositions~\ref{prop:Scircloop}, \ref{prop:Scirchole}, \ref{prop:Sloop}, and~\ref{prop:Shole} give an expression for $\mathbb{I}_{\mathcal{D}}(c)$ in terms of the $B$- and $X$-coordinates. This expression includes a factor which is a monomial in the $X$-coordinates with half integral exponents. By the remarks following Proposition~\ref{prop:Shole}, we know that this monomial has the form $X_{i_1}^{\pm k/2}\dots X_{i_s}^{\pm k/2}$ where $i_1,\dots,i_s$ are the edges of $T$ that $\sigma_c$ intersects. Similarly, if $c$ is an intersecting curve on~$S_\mathcal{D}$, then Proposition~\ref{prop:intersecting} gives an expression for $\mathbb{I}_{\mathcal{D}}(c)$ in terms of the $B$- and $X$-coordinates. This expression includes a factor which is a monomial in the $X$-coordinates with half integral exponents. In proving Proposition~\ref{prop:intersecting}, we essentially constructed a cycle $\sigma_c'$ with $\mathbb{Z}/2\mathbb{Z}$-coefficients such that this monomial has the form $X_{i_1}^{\pm1/2}\dots X_{i_s}^{\pm1/2}$ where $i_1,\dots,i_s$ are the edges of $T$ that $\sigma_c'$ intersects.

It is easy to see that $\sigma_c'$ is homologous to $\sigma_c$ for any intersecting curve $c$. Thus for an arbitrary $\mathcal{D}$-lamination $l$, we can replace $\sigma$ by a homologous cycle $\sigma'$ in which every summand of $\sigma$ of the form $\sigma_c$ for $c$ an intersecting curve is replaced by $\sigma_c'$. Our assumption that $[\sigma]=0\in H_1(S,\mathbb{Z}/2\mathbb{Z})$ implies that $[\sigma']=0$ as well. Arguing as in Step~2, we find that each edge of $T$ intersects $\sigma'$ in an even number of points. 

Theorem~\ref{thm:main} gives an expression for $\mathbb{I}_{\mathcal{D}}(l)$ in terms of $B$- and $X$-coordinates for such a $\mathcal{D}$-lamination $l$. This expression includes a factor which is a monomial in the $X$-coordinates with half integral exponents $h_{l,i}$, and we want to show that each $h_{l,i}$ is in fact an integer. Since $\mathbb{I}_{\mathcal{D}}(l)$ is the product of the functions $\mathbb{I}_{\mathcal{D}}(c)$ where $c$ is an intersecting curve in $l$ or a loop in~$\mathcal{C}$ or~$\mathcal{C}^\circ$, this follows immediately from the fact that any edge $i$ of the triangulation $T$ intersects $\sigma'$ in an even number of points.

\Needspace*{2\baselineskip}
\begin{step}[4]
$3\implies2$
\end{step}

Suppose $\mathbb{I}_{\mathcal{D}}(l)$ is a rational function in the variables $B_j$ and~$X_j$. Then each of the exponents $h_{l,i}$ arising from Theorem~\ref{thm:main} is an integer. This means that each edge of the ideal triangulation~$T$ intersects the cycle $\sigma'$ constructed in~Step~3 in an even number of points. Arguing as in~Step~1, we can recursively add cycles to $\sigma'$ to prove that $[\sigma']=0\in H_1(S,\mathbb{Z}/2\mathbb{Z})$. Since $\sigma'$ is homologous to $\sigma$, this implies $[\sigma]=0\in H_1(S,\mathbb{Z}/2\mathbb{Z})$.
\end{proof}

\subsection{Further properties}

Note that there is a natural map $\varphi:\mathcal{A}_0\times\mathcal{A}_0\rightarrow\mathcal{D}$ taking the product of two copies of the cluster $\mathcal{A}$-variety into the symplectic double. It is given in coordinates by the formulas 
\begin{align*}
\varphi^*(B_i) &= \frac{A_i^\circ}{A_i}, \\
\varphi^*(X_i) &= \prod_j A_j^{\varepsilon_{ij}}.
\end{align*}
where $A_i^\circ$ are the coordinates on the second copy of $\mathcal{A}_0$. This map induces maps $\varphi:\mathcal{A}_0^+\times\mathcal{A}_0^+\rightarrow\mathcal{D}^+$ and $\varphi:\mathcal{A}_0(\mathbb{Z}^t)\times\mathcal{A}_0(\mathbb{Z}^t)\rightarrow\mathcal{D}(\mathbb{Z}^t)$, which we denote by the same symbol~$\varphi$. Concretely, these maps are defined by gluing together hyperbolic surfaces or laminations, respectively.

There is also a natural projection $\pi:\mathcal{D}\rightarrow\mathcal{X}\times\mathcal{X}$ given in coordinates by the formulas 
\begin{align*}
\varphi^*(X_i) &= X_i, \\
\varphi^*(\widehat{X}_i) &= X_i\prod_j B_j^{\varepsilon_{ij}}.
\end{align*}
where $\widehat{X}_i$ are the coordinates on the second copy of $\mathcal{X}$. This projection induces maps $\pi:\mathcal{D}^+\rightarrow\mathcal{X}^+\times\mathcal{X}^+$ and $\pi:\mathcal{D}(\mathbb{Z}^t)\rightarrow\mathcal{X}(\mathbb{Z}^t)\times\mathcal{X}(\mathbb{Z}^t)$, which we denote by the same symbol~$\pi$. Concretely, these maps are defined by cutting the surface $S_{\mathcal{D}}$ to recover the surfaces~$S$ and~$S^\circ$.

We will write $\Fun(\mathcal{D}^+)$ for the set of all functions $\mathcal{D}^+\rightarrow\mathbb{R}_{>0}$ that are rational in the $B$- and $X$-coordinates. We will likewise denote by $\Fun(\mathcal{A}_0^+\times\mathcal{A}_0^+)$ the set of all functions $\mathcal{A}_0^+\times\mathcal{A}_0^+\rightarrow\mathbb{R}_{>0}$ that are rational in the coordinates on the two copies of $\mathcal{A}_0^+$, and we will denote by $\Fun(\mathcal{X}^+\times\mathcal{X}^+)$ the set of all functions $\mathcal{X}^+\times\mathcal{X}^+\rightarrow\mathbb{R}_{>0}$ that are rational in the coordinates on the two copies of $\mathcal{X}^+$. We have seen that the multiplicative canonical pairing provides a map 
\[
\mathbb{I}_{\mathcal{D}}:\mathcal{D}(\mathbb{Z}^t)\rightarrow\Fun(\mathcal{D}^+).
\]
There is a map $\mathbb{I}_{\mathcal{A}_0\times\mathcal{A}_0}:\mathcal{A}_0(\mathbb{Z}^t)\times\mathcal{A}_0(\mathbb{Z}^t) \rightarrow \Fun(\mathcal{X}^+\times\mathcal{X}^+)$ which takes a pair~$(l,l^\circ)$ of laminations to the function $\mathbb{I}(l^\circ)/\mathbb{I}(l)$. There is a similar map $\mathbb{I}_{\mathcal{X}\times\mathcal{X}}:\mathcal{X}(\mathbb{Z}^t)\times\mathcal{X}(\mathbb{Z}^t) \rightarrow \Fun(\mathcal{A}_0^+\times\mathcal{A}_0^+)$ which takes a pair $(l,l^\circ)$ to the function $\mathbb{I}(l^\circ)/\mathbb{I}(l)$.

The following proposition shows that the map $\mathbb{I}_{\mathcal{D}}$ is naturally compatible with the classical canonical pairings. It follows immediately from the definitions.

\begin{proposition}
There is a commutative diagram 
\[
\xymatrix{ 
\mathcal{A}_0(\mathbb{Z}^t)\times\mathcal{A}_0(\mathbb{Z}^t) \ar[r]^{\mathbb{I}_{\mathcal{A}_0\times\mathcal{A}_0}} \ar[d]_{\varphi} & \Fun(\mathcal{X}^+\times\mathcal{X}^+) \ar[d]^{\pi^*} \\
\mathcal{D}(\mathbb{Z}^t) \ar[d]_{\pi} \ar[r]^-{\mathbb{I}_{\mathcal{D}}} & \Fun(\mathcal{D}^+) \ar[d]^{\varphi^*} \\
\mathcal{X}(\mathbb{Z}^t)\times\mathcal{X}(\mathbb{Z}^t) \ar[r]^{\mathbb{I}_{\mathcal{X}\times\mathcal{X}}} & \Fun(\mathcal{A}_0^+\times\mathcal{A}_0^+).
}
\]
\end{proposition}

In~\cite{IHES}, Fock and Goncharov proved a number of special properties of the classical canonical pairings. In particular, they showed that the canonical functions provided by these pairing are expressible in any coordinate system as Laurent polynomials with positive integral coefficients. By Theorem~\ref{thm:main}, Proposition~\ref{prop:kapparational}, and Theorem~\ref{thm:integralDlam}, we see that each function $\mathbb{I}_{\mathcal{D}}(l)$ for $l\in\mathcal{D}(\mathbb{Z}^t)$ can be expressed as a subtraction-free rational function in the $B$- and $X$-coordinates.

One of the main results of~\cite{IHES} states that the Laurent polynomials $\mathbb{I}(l)$ for $l\in\mathcal{A}(\mathbb{Z}^t)$ an $\mathcal{A}$-lamination on a punctured surface provide a canonical vector space basis for the algebra of regular functions on the cluster $\mathcal{X}$-variety. This basis has the special property that the structure constants for multiplication of basis elements are positive integers. Since the functions $\mathbb{I}_{\mathcal{D}}(l)$ defined in the present paper for $l\in\mathcal{D}(\mathbb{Z}^t)$ are \emph{rational} rather than regular functions, we do not have a similar canonical basis parametrized by $\mathcal{D}$-laminations. Understanding the precise relationship between the construction described here and canonical bases is an interesting problem for future research.

\subsection{The intersection pairing}

In addition to the multiplicative canonical pairing $\mathbb{I}_{\mathcal{D}}:\mathcal{D}_L(S,\mathbb{Z})\times\mathcal{D}^+(S)\rightarrow\mathbb{R}_{>0}$, we have a canonical map $\mathcal{I}_{\mathcal{D}}:\mathcal{D}_L(S,\mathbb{Z})\times\mathcal{D}_L(S,\mathbb{Z})\rightarrow\mathbb{Q}$ defined as follows.

Suppose we are given an intersecting curve $l$ on $S_\mathcal{D}$. Deform this curve so that it intersects the image of~$\partial S$ in the minimal number of points. Then, starting from a component of this image which we may call $\gamma_1$, there is a segment $c_1$ of the curve~$l$ which lies entirely in~$S$ and connects the component~$\gamma_1$ to another component which we may call~$\gamma_2$. Starting from this component, there is a segment $c_2$ of $l$ which lies entirely in~$S^\circ$ and connects~$\gamma_2$ to a component~$\gamma_3$. Continue labeling in this way until the curve closes.

Now suppose we are given a point $m\in\mathcal{D}_L(S,\mathbb{Z})$. Draw this lamination on the same copy of $S_{\mathcal{D}}$ as $l$ and deform its curves so that they intersect the image of~$\partial S$ in the minimal number of points. Cut the surface $S_\mathcal{D}$ along $\partial S$. Then each $c_i$ is a curve that connects two boundary components of~$S$ or~$S^\circ$. We can wind the ends of $c_i$ around the holes infinitely many times in the direction prescribed by the orientations from $m$.

Consider a curve $c_i$ for $i$ odd. Lifting this curve to the universal cover of $S$, we obtain a geodesic $\tilde{c}_i$ connecting two points on the boundary of $\mathbb{H}$. Choose a distinguished curve intersecting $\tilde{c}_i$ near each of these boundary points, and define~$a_{c_i}$ as half the number of intersections of the lifted curves between the distinguished curves. Next consider $c_i$ for~$i$~even. Lifting to the universal cover of $S^\circ$, we again get a geodesic connecting two points on the boundary of $\mathbb{H}$, and the distinguished curves already chosen determine a pair of distinguished curves near these points. We define~$a_{c_i}^\circ$ as half the number of intersections of the lifted curves between the distinguished curves.

\Needspace*{5\baselineskip}
\begin{definition}\label{def:intersectionD} \mbox{}
Fix a point $m\in\mathcal{D}_L(S,\mathbb{Z})$.
\begin{enumerate}
\item Let $l$ be an intersecting curve of weight $k$ on $S_\mathcal{D}$, and assume that the orientation of each component of $\gamma$ agrees with the orientation of~$S$. Then 
\[
\mathcal{I}_\mathcal{D}(l,m)=k\biggr(\sum_{\text{$i$ even}}a_{c_i}^\circ-\sum_{\text{$i$ odd}}a_{c_i}\biggr).
\]

\item Let $l$ be a curve of weight $k$ on $S_\mathcal{D}$ which is not an intersecting curve and is not homotopic to a loop in the simple lamination. If $l$ lies in $S^\circ$, then $\mathcal{I}_\mathcal{D}(l,m)$ is defined as $k/2$ times the minimal number of intersections between~$l$ and~$m$. If $l$ lies in $S$ then $\mathcal{I}_\mathcal{D}(l,m)$ is minus this quantity.

\item Let $l$ be a curve of weight $k$ on $S_\mathcal{D}$ which is homotopic to a loop $\gamma_i$ in the simple lamination. If $l$ and~$m$ provide the same orientation for $\gamma_i$, then $\mathcal{I}_\mathcal{D}(l,m)$ is defined as $k/2$ times the minimal number of intersections between~$l$ and~$m$. If $l$ and~$m$ provide different orientations for $\gamma$, then $\mathcal{I}_\mathcal{D}(l,m)$ is minus this quantity.

\item Let $l_1$ and $l_2$ be laminations on $S_\mathcal{D}$ such that no curve from $l_1$ intersects a curve from~$l_2$. Then $\mathcal{I}_\mathcal{D}(l_1+l_2,m)=\mathcal{I}_\mathcal{D}(l_1,m)+\mathcal{I}_\mathcal{D}(l_2,m)$.
\end{enumerate}
\end{definition}

In the classical setting of $\mathcal{A}$- and $\mathcal{X}$-laminations, the intersection pairing is the tropicalization of the multiplicative canonical pairing. We will now prove the analogous result for $\mathcal{D}$-laminations.

\begin{proposition}
\label{prop:intersectingintersection}
Let $l$ be an intersecting curve of weight~1 on~$S_{\mathcal{D}}$. If $\gamma_i$ is a component of $\gamma$ that meets a curve of $l$, assume that the orientation that $l$ provides for $\gamma_i$ agrees with the orientation of~$S$. Then we have $\mathcal{I}_{\mathcal{D}}(l,m)=(\mathbb{I}_{\mathcal{D}}(l))^t(m)$ for any~$m\in\mathcal{D}_L(S,\mathbb{Z})$.
\end{proposition}

\begin{proof}
Let $m\in\mathcal{D}_L(S,\mathbb{Z})$ be a $\mathcal{D}$-lamination with coordinates $b_j$ and~$x_j$ for $j\in J$, and let~$e^{Cm}$ denote the point of $\mathcal{D}^+(S)$ with coordinates 
\[
B_j=e^{Cb_j}
\]
and
\[
X_j=e^{Cx_j}
\]
for $j\in J$. This point $e^{Cm}\in\mathcal{D}^+(S)$ determines hyperbolic structures on~$S$ and $S^\circ$, so we can view the universal covers $\tilde{S}$ and~$\tilde{S}^\circ$ as subsets of the hyperbolic plane.

Cut the intersecting curve $l$ along the image of $\partial S$ and then, using the natural map $S^\circ\rightarrow S$, draw all of the resulting curves on $S$. These curves form a cycle on~$S$ which we can lift to a path in the universal cover $\tilde{S}$ connecting points on $\partial\mathbb{H}$. Choose an ideal triangulation $T$ of $S$, and lift this to a triangulation $\tilde{T}$ of $\tilde{S}$. Let $P$ be a triangulated ideal polygon, formed from the triangles of $\tilde{T}$, that contains all of the triangles that intersect this path. The curves of the path that project to curves from $S$ determine a lamination $u\in\mathcal{X}_L(P)$.

In exactly the same way, we can lift $S^\circ$ to its universal cover $\tilde{S}^\circ$, and the triangulation $\tilde{T}$ and polygon $P$ determine a triangulation $\tilde{T}^\circ$ and a polygon~$P^\circ$ in~$\tilde{S}^\circ$. The curves of the path in $P^\circ$ that project to curves from $S^\circ$ determine a lamination $u^\circ\in\mathcal{X}_L(P^\circ)$.

Now consider the portion of $m$ that intersects $S$. Lifting these curves to the universal cover, we get a collection of (possibly infinitely many) curves on the polygon $P$. Let $p$ be any vertex of $P$. If there are infinitely many curves connecting the sides that meet at $p$, then we can choose one such curve $\alpha_p$ and delete all of the curves between $\alpha_p$ and the point $p$. By doing this for each vertex, we remove all but finitely many curves and get a lamination $v\in\mathcal{A}_L(P)$. 

In the same way, we can consider the portion of $m$ that intersects $S^\circ$. Lifting these curves to the universal cover, we get a collection of (possibly infinitely many) curves on the polygon~$P^\circ$. The distinguished curves $\alpha_p$ that we chose near the vertices of $P$ determine corresponding curves $\alpha_p^\circ$ near the the vertices of $P^\circ$. For each $p$, delete all of the curves on $P^\circ$ between $\alpha_p^\circ$ and $p$. In this way, we remove all but finitely many curves and get a lamination $v^\circ\in\mathcal{A}_L(P^\circ)$.

For any edge $i$ of the triangulation of~$P$, write $a_i$ for the corresponding coordinate of the lamination $v$. For any edge $i$ of the triangulation of~$P^\circ$, write $a_i^\circ$ for the corresponding coordinate of the lamination $v^\circ$. Consider the point $e^{Cv}$ of $\mathcal{A}^+(P)$ with coordinates $A_i=e^{Ca_i}$ and the point $e^{Cv^\circ}$ of $\mathcal{A}^+(P^\circ)$ with coordinates $A_i^\circ=e^{Ca_i^\circ}$. Since the cross ratios of the numbers $A_i$ are simply the coordinates~$X_j$ introduced above, we know that the point $e^{Cv}\in\mathcal{A}^+(P)$ determines the same hyperbolic structure on $P$ that we get from the point~$e^{Cm}$. Similarly, the point $e^{Cv^\circ}\in\mathcal{A}^+(P^\circ)$ determines the same hyperbolic structure on~$P^\circ$ that we get from~$e^{Cm}$. It follows that 
\[
\mathbb{I}_{\mathcal{D}}(l)(e^{Cm})=\frac{\mathbb{I}(u^\circ,e^{Cv^\circ})}{\mathbb{I}(u,e^{Cv})}.
\]
Using the definition of tropicalization and the result from the classical theory of the $\mathcal{A}$- and $\mathcal{X}$-spaces, we see that 
\begin{align*}
(\mathbb{I}_{\mathcal{D}}(l))^t(m) &= \lim_{C\rightarrow\infty}\frac{\log\mathbb{I}_{\mathcal{D}}(l)(e^{Cm})}{C} \\
&= \lim_{C\rightarrow\infty}\frac{\log\mathbb{I}(u^\circ,e^{Cv^\circ})}{C} - \frac{\log\mathbb{I}(u,e^{Cv})}{C}\\
&= \mathcal{I}_{\mathcal{D}}(l,m)
\end{align*}
as desired.
\end{proof}

\begin{proposition}
\label{prop:loopintersection}
Let $l$ be a loop of weight $k$ on $S_{\mathcal{D}}$ which is not an intersecting curve. Then we have $\mathcal{I}_{\mathcal{D}}(l,m)=(\mathbb{I}_{\mathcal{D}}(l))^t(m)$ for any~$m\in\mathcal{D}_L(S,\mathbb{Z})$.
\end{proposition}

\begin{proof}
Let $m\in\mathcal{D}_L(S,\mathbb{Z})$ be a $\mathcal{D}$-lamination with coordinates $b_j$ and~$x_j$ ($j\in J$), and let $e^{Cm}$ denote the point of $\mathcal{D}^+(S)$ with coordinates $B_j=e^{Cb_j}$ and~$X_j=e^{Cx_j}$.

If $l$ lies entirely on $S^\circ$ and is not homotopic to a curve in the simple lamination, then we define $u$ to be the point of $\mathcal{A}_L(S^\circ)$ obtained by cutting $S_{\mathcal{D}}$. On the other hand, suppose $l$ is homotopic to a loop $\gamma_i$ in the simple lamination and $l$ and~$m$ provide the same orientation for this loop. If the orientation of~$\gamma_i$ agrees with the orientation of~$S^\circ$, then we define $u$ to be the point of $\mathcal{A}_L(S)$ obtained by shifting $l$ onto the surface $S$ and cutting $S_{\mathcal{D}}$. If the orientation agrees with that of $S$, then we define $u$ to be the point of $\mathcal{A}_L(S^\circ)$ obtained by shifting $l$ onto $S^\circ$ and cutting.

Now let $v$ be the point of $\mathcal{X}_L(S)$ or $\mathcal{X}_L(S^\circ)$ that we get by drawing the lamination $m$ on $S_{\mathcal{D}}$ and cutting. Let $e^{Cv}$ denote the point of $\mathcal{X}^+(S)$ or $\mathcal{X}^+(S^\circ)$ whose coordinates are obtained by scaling the coordinates of $v$ by $C$ and exponentiating. By construction, we have 
\[
\mathbb{I}_{\mathcal{D}}(l,e^{Cm})=\mathbb{I}(u,e^{Cv}).
\]
It follows that 
\begin{align*}
(\mathbb{I}_{\mathcal{D}}(l))^t(m) &= \lim_{C\rightarrow\infty}\frac{\log\mathbb{I}_{\mathcal{D}}(l,e^{Cm})}{C} \\
&= \lim_{C\rightarrow\infty}\frac{\log\mathbb{I}(u,e^{Cv})}{C} \\
&= (\mathbb{I}(u))^t(v).
\end{align*}
Applying the result from the classical theory of the $\mathcal{A}$- and $\mathcal{X}$-spaces, we see that the last expression is $k/2$ times the minimal number of intersections between $l$ and~$m$.

If $l$ lies entirely on $S$ and is not homotopic to a curve in the simple lamination, then we define $u$ to be the point of $\mathcal{A}_L(S)$ obtained by cutting $S_{\mathcal{D}}$. On the other hand, suppose $l$ is homotopic to a loop $\gamma_i$ in the simple lamination and $l$ and $m$ provide different orientations for this loop. If the orientation that $l$ provides for~$\gamma_i$ agrees with the orientation of~$S$, then we define $u$ to be the point of $\mathcal{A}_L(S)$ obtained by shifting $l$ onto the surface $S$ and cutting. If the orientation agrees with that of $S^\circ$, then we define $u$ to be the point of $\mathcal{A}_L(S^\circ)$ obtained by shifting~$l$ onto~$S^\circ$ and cutting.

Let $v$ and $e^{Cv}$ be as above. Then by construction, we have 
\[
\mathbb{I}_{\mathcal{D}}(l,e^{Cm})=\mathbb{I}(u,e^{Cv})^{-1}
\]
so that 
\begin{align*}
(\mathbb{I}_{\mathcal{D}}(l))^t(m) &= \lim_{C\rightarrow\infty}\frac{\log\mathbb{I}_{\mathcal{D}}(l,e^{Cm})}{C} \\
&= \lim_{C\rightarrow\infty}\frac{\log\mathbb{I}(u,e^{Cv})^{-1}}{C} \\
&= -(\mathbb{I}(u))^t(v).
\end{align*}
Applying the result from the classical theory once again, we see that this last expression is minus $k/2$ times the minimal number of intersections between $l$ and~$m$.
\end{proof}

\begin{proposition}
Let $l$,~$m\in\mathcal{D}_L(S,\mathbb{Z})$. If $\gamma_i$ is a component of~$\gamma$ that meets a curve of~$l$, assume that the orientation that $l$ provides for $\gamma_i$ agrees with the orientation of~$S$. Then $\mathcal{I}_{\mathcal{D}}(l,m)=(\mathbb{I}_{\mathcal{D}}(l))^t(m)$.
\end{proposition}

\begin{proof}
Let $l\in\mathcal{D}_L(S,\mathbb{\mathbb{Z}})$. Then we can write $l=l_1+\dots+l_d$ where $l_1,\dots,l_d$ are mutually nonintersecting and nonhomotopic closed curves. We then have 
\begin{align*}
\mathcal{I}_{\mathcal{D}}(l,m) &= \mathcal{I}_{\mathcal{D}}(l_1,m)+\dots+\mathcal{I}_{\mathcal{D}}(l_d,m) \\
&= (\mathbb{I}_{\mathcal{D}}(l_1))^t(m)+\dots+(\mathbb{I}_{\mathcal{D}}(l_d))^t(m) \\
&= (\mathbb{I}_{\mathcal{D}}(l_1)\dots\mathbb{I}_{\mathcal{D}}(l_d))^t(m) \\
&= (\mathbb{I}_{\mathcal{D}}(l))^t(m)
\end{align*}
by Propositions~\ref{prop:intersectingintersection} and~\ref{prop:loopintersection} and the multiplicativity and additivity of the maps $\mathbb{I}_{\mathcal{D}}$ and~$\mathcal{I}_{\mathcal{D}}$.
\end{proof}

We can naturally extend $\mathcal{I}_{\mathcal{D}}$ to a pairing $\mathcal{D}_L(S,\mathbb{Z})\times\mathcal{D}_L(S,\mathbb{Z})\rightarrow\mathbb{Q}$ using the tropicalization of the rational map $\tau$ appearing in Proposition~\ref{prop:kapparational}.

\section*{Acknowledgments}
\addcontentsline{toc}{section}{Acknowledgements}

I am grateful to Alexander~Goncharov for advising me on this project and for contributing many important ideas including the definition of canonical pairing, the idea of expressing this pairing in terms of $F$-polynomials, and the homological criterion for getting rational functions from $\mathcal{D}$-laminations. I thank Yair~Minsky for a conversation which helped clarify some aspects the problem. I thank Linhui~Shen for many helpful discussions and for contributing important ideas to the proof of Proposition~\ref{prop:intersecting}.

This work was partially supported by the NSF grant DMS-1301776.

\end{document}